\documentclass[12pt,vi,twoside]{mitthesis}

\usepackage{algorithmic}
\usepackage{lgrind}
\usepackage{amsmath, amsthm, amssymb}
\usepackage{textcomp}
\usepackage[usenames]{color}
\usepackage[dvips]{graphics}
\usepackage{amssymb}
\usepackage{epsf}
\usepackage[dvips]{graphicx}
\usepackage{epsfig}
\newcommand{\smskip}{\par\vspace{5pt}}

\providecommand{\abs[1]}{\left|#1\right|}

\providecommand{\1}{\textbf{1}}
\newcommand{\ao}[1]{{#1}}
\newcommand{\aoc}[1]{{\color{red} #1}}
\newcommand{\JMH}[1]{{#1}}
\def\R{\mathbb{R}}
\newtheorem{theorem}{Theorem}[chapter]
\newtheorem{proposition}{Proposition}[chapter]
\newtheorem{lemma}{Lemma}[chapter]
\def\jnt#1{{{#1}}}
\long\def\red#1{{{\color{red}#1}}}
\def\ao#1{{{#1}}}
\def\alexo#1{{{#1}}}
\def\an#1{{{#1}}}
\def\E{\mathcal{E}}

\newtheorem{assumption}{Assumption}[chapter]

\def\0{{\bf 0}}
\def\sin{\sum_{i=1}^n}
\def\sjn{\sum_{j=1}^n}
\usepackage{url}
\def\E{\mathcal{E}}
\def\L{\mathcal{L}}
\def\Q{\mathcal{Q}}
\def\A{\mathcal{A}}
\def\Ps{P_{\sigma}}
\def\e{{\bf 1}}
\def\ei{{\e^{(i)}}}
\def\ej{{\e^{(j)}}}
\def\ek{{\e^{(k)}}}
\def\B{{\mathcal B}}
\def\Pm{{\mathcal P}}
\newtheorem{definition}[theorem]{Definition}

\def\R{\mathbb{R}}

\newtheorem{corollary}[theorem]{\indent Corollary}

\def\aoe{}
\def\ao#1{{#1}}

\def\rt#1{#1}  
\def\mod#1{#1} 
\long\def\jnt#1{{#1}}

\def\jblue#1{{#1}}
\def\jred#1{{#1}}
\def\jmj#1{{#1}}

\definecolor{orange}{rgb}{1,0.5,0}
\long\def\ora#1{{#1}}
\def\aoc#1{{#1}}

\def\jh#1{{#1}}
\long\def\old#1{}

\def\red#1{{#1}}
\def\blue#1{{#1}}
\def\green#1{{#1}}
\def\blue#1{{#1}}
\def\E{\mathcal{E}}

\def\oM{\overline{M}}
\def\oT{\overline{T}}
\providecommand{\floor[1]}{\lfloor  #1 \rfloor }
\providecommand{\ceil[1]}{\lceil #1  \rceil}
\def\Rin{\mbox{Rin}}
\def\Rout{\mbox{Rout}}

\providecommand{\com[2]}{\begin{tt}[#1: #2]\end{tt}}

\def\bc{\begin{center}}
\def\ec{\end{center}}

\begin{document}
\setcounter{tocdepth}{1}

%
%
%
%
%
%
%
%

\title{Efficient Information Aggregation Strategies for Distributed
Control and Signal Processing}

\author{Alexander Olshevsky}
       \prevdegrees{B.S., Mathematics, Georgia Institute of Technology (2004) \\
                    B.S., Electrical Engineering, Georgia Institute of Technology (2004) \\
                    M.S., Electrical Engineering and Computer Science, Massachusetts Institute of Technology (2006)}
\department{Department of Electrical Engineering and Computer Science}

 \degree{Doctor of Philosophy}

\degreemonth{September}
\degreeyear{2010}
\thesisdate{August 30, 2010}


\supervisor{John N. Tsitsiklis}{Clarence J Lebel Professor of Electrical Engineering}

\chairman{Terry P. Orlando}{Professor of Electrical Engineering and Computer
Science, \\Chair, Committee for Graduate Students}

\maketitle



\cleardoublepage
\setcounter{savepage}{\thepage}
\begin{abstractpage}
%
%
%

This thesis will be concerned with distributed control and coordination of networks
consisting of multiple, potentially mobile, agents. This is motivated mainly by the
emergence of large scale networks characterized by the lack of centralized access to information
and time-varying connectivity. Control and optimization algorithms deployed
in such networks should be completely distributed, relying only on local observations
and information, and robust against unexpected changes in topology such as link failures.

We will describe protocols to solve certain control and signal processing problems in this 
setting. We will demonstrate that a key challenge for such systems is the problem of computing averages in a decentralized way. 
Namely, we will show that a number of distributed control and signal processing problems can be solved straightforwardly
if solutions to the averaging problem are available. 

The rest of the thesis will be concerned with algorithms for the averaging problem and its generalizations. We will
(i) derive the fastest known averaging algorithms in a variety of settings and subject to a variety of 
communication and storage constraints (ii) prove a lower bound identifying a fundamental barrier
for averaging algorithms (iii) propose a new model for distributed function computation which reflects 
the constraints facing many large-scale networks, and nearly characterize the general class of functions which can be computed in this model.
\end{abstractpage}


\cleardoublepage

\section*{Acknowledgments}

I am deeply indebted to my advisor, John Tsitsiklis, for his invaluable guidance and tireless 
efforts in supervising this thesis. I have learned a great deal from our research collaborations. 
I have greatly benefitted from his constantly insightful suggestions and his subtle understanding when a technique
is powerful enough to solve a problem and when it is bound to fail.

I would like to thank my thesis committee members Emilio Frazzoli and  Asuman Ozdaglar for
their careful work in reading this thesis, which was considerably improved as a result of their
comments. 

I would like to thank the many people with whom I've had conversations which have
helped me in the course of my research: Amir Ali Ahmadi, Pierre-Alexandre Bliman, Vincent Blondel, Leonid Gurvits, Julien Hendrickx, Ali Jadbabaie, Angelia Nedic, and Pablo Parrilo. 

I would like to thank my family for the help they have given me
over the years. Most importantly, I want to thank my wife Angela for her constant 
encouragement and support.

This research was sponsored by the National Science Foundation under an NSF Graduate Research 
Fellowship and grant ECCS-0701623.


\pagestyle{plain}
\begin{singlespace}
\tableofcontents
\newpage

\chapter{Introduction}

This thesis is about certain control and signal processing problems over networks with unreliable
communication links. Some motivating scenarios are: 
\begin{enumerate}
\renewcommand{\theenumi}{\alph{enumi}.}
\renewcommand{\labelenumi}{\theenumi}
\item Distributed estimation: a collection of sensors are trying to estimate an unknown parameter from 
observations at each sensor.
\item Distributed state estimation: a collection of sensors are trying to estimate the (constantly evolving) state
of a linear dynamical system from observations of its output at each sensor. 
\item Coverage control: A group of robots wish to position themselves so as to optimally monitor an
 environment of interest.
\item Formation control: Several UAVs or vehicles are attempting to maintain a formation against random disturbances
to their positions.
\item Distributed task assignment: allocate a collection of tasks among agents with individual preferences in a distributed way.
\item Clock synchronization. A collection of clocks are constantly drifting apart, and would like to maintain
a common time as much as this is possible. Various pairs of clocks can measure noisy versions of the time offsets between them.
\end{enumerate} We will use ``nodes" as a common word for sensors, vehicles, UAVs, and so on. A key assumption we will be making
is that the communication links by means of which the nodes exchange messages are unreliable. A variety of simple and standard techniques can be used for any of the above problems if the communication links never fail. By contrast, we will be interested in the case when the links may ``die" and ``come online" unpredictably. We are interested in algorithms for the above problems which
work even in the face of this uncertainty.

It turns out that a key problem for systems of this type is the problem of computing averages, described next. Nodes $1,\ldots,n$ each begin with a real number $x_i$. We are given a discrete sequence of times $t=1,2,3,\ldots,$ and at each time step a communication graph $G(t)=(\{1,\ldots,n\}, E(t))$ is exogenously provided by ``nature" determining which nodes can
communicate: node $i$ can send messages to node $j$ at time $t$ if and only if $(i,j) \in E(t)$. For simplicity, no
 restrictions are placed on the messages nodes can send to each other, and in particular, the nodes
 may broadcast their initial values to each other. The nodes need to compute $(1/n) \sum_{i=1}^n x_i$, subject to as few assumptions as possible about the communication sequence $G(t)$. We will call this the {\em averaging problem,} and we will call any algorithm for it an {\em averaging algorithm}.

The averaging problem is key in the sense that a variety of results are available describing how to use averaging algorithms to solve many other distributed problems with unreliable communication links. In particular, for each
of the above problems (distributed estimation, distributed state estimation, coverage control, formation control,
clock synchronization), averaging algorithms are a cornerstone of the best currently known solutions.

The remainder of this chapter will begin by giving a historical survey of averaging algorithms, followed by a list 
of the applications of averaging, including the problems on the previous page. 

\section{A history of averaging algorithms} The first paper to introduce distributed averaging
algorithms was by DeGroot \cite{DG74}. DeGroot considered a simple model of ``attaining agreement:" $n$ individuals
 are on a team or committee and would like to come to a consensus about the probability distribution of a certain parameter $\theta$. Each individual $i$ begins with a probability distribution $F_i$ which they believe is the
 correct distribution of $\theta$. For simplicity, we assume that $\theta$ takes on a value in some finite set $\Omega$, so that each $F_i$ can be described by $|\Omega|$ numbers.

 The individuals now update their probability distributions as a result of interacting with each other. Letting $F_i(t)$ be the distribution believed by $i$ at time $t$, the agents update as $F_i(t+1) = \sum_{j} a_{ij} F_j(t)$, with
  the initialization $F_i(0)=F_i$. Here, $a_{ij}$ are weights chosen by the individuals. Intuitively, people may give high weights to a subset of the people they interact with, for example if a certain person is believed to be an expert in the subject. On the other hand, some weights may be zero, which corresponds to the possibility that some individual ignore each other's opinions. It is assumed, however, that all the weights $a_{ij}$ are nonnegative, and $a_{i1},\ldots,a_{in}$ add up to $1$ for every $i$. Note that the coefficients $a_{ij}$ are independent of time, corresponding to a ``static" communication pattern: individuals do not change how much they trust the opinions of others.

  DeGroot gave a condition for these dynamics to converge, as well as a formula for the limiting opinion; subject to some natural symmetry conditions, the limiting opinion distribution will equal the average of the initial opinion
  distributions $F_i$. DeGroot's work
  was later extended by Chatterjee and Seneta \cite{CS77} to the case where the weights $a_{ij}$ vary with time. The
  paper \cite{CS77} gave some conditions on the time-varying sequence $a_{ij}(t)$ required for convergence to agreement on a single distribution among the individuals.

  The same problem of finding conditions on $a_{ij}(t)$ necessary for agreement was addressed in the works \cite{T84,TBA86, BT89}, which were motivated by problems in parallel computation. Here, the problem was phrased
  slightly differently: $n$ processors each begin with a number $x_i$ stored in memory, and the processors need to
  agree on a single number within the convex hull $[ \min_i x_i, \max_i x_i]$. This is accomplished by iterating
  as $x_i(t+1)=\sum_{j} a_{ij} x_j(t)$. This problem was a subroutine of several parallel optimization algorithms \cite{T84}. It is easy to see that it is equivalent to the formulation in terms of probability distributions addressed by DeGroot \cite{DG74} and Chatterjee and Seneta \cite{CS77}.

  The works \cite{T84, TBA86, BT89} gave some conditions necessary for the estimates $x_i(t)$ to converge to a common
  value. These were in the same spirit as \cite{CS77}, but were more combinatorial than \cite{CS77} being expressed
  directly in terms of the coefficients $a_{ij}(t)$. These conditions boiled down to a series of requirements suggesting
  that the agents have repeated and nonvanishing influence over each other. For example, the coefficients $a_{ij}(t)$
  should not be allowed to decay to zero, and the graph sequence $G(t)$ containing the edges $(i,j)$ for which $a_{ji}(t)>0$ needs to be ``repeatedly connected.'' 

  Several years later, a similar problem was studied by Cybenko \cite{C89} motivated by load balancing problems. In
  this context, $n$ processors each begin with a certain number of jobs $x_i$. The variable $x_i$ can only be an
  integer, but assuming a large number of jobs in the system, this assumption may be dispensed with. The processors would like to equalize the load. To that end, they pass around jobs: processors with many jobs try to offload
  their jobs on their neighbors, and processors with few job ask for more requests from their neighbors. The
  number of jobs of processor $i$ behave approximately as $x_i(t+1) = x_i(t) + \sum_{j} \alpha_{ij} (x_j(t) - x_i(t))$, which subject to some conditions on the coefficients $\alpha_{ij}$, may be viewed as a special case of the iterations
  considered in \cite{CS77, T84, TBA86, BT89}.

  Cybenko showed that when the neighborhood structure is a hypercube (i.e., we associate with each processor a string
  of $\log n$ bits, and $\alpha_{ij}(t) \neq 0$ whenever processors $i$ and $j$ differ by at most $1$ bit), the
  above processes may converge quite fast: for some natural simple processes, an appropriately defined convergence time is on the order of $\log n$ iterations.

  Several years later, a variation on the above algorithms was studied by Vicsek et al. \cite{V95}. Vicsek et al. simulated the following scenario: $n$ particles were placed randomly on a torus with random initial direction
  and constant velocity. Periodically, each particle would try to align its angle with the angles of all the
  particles within a certain radius. Vicsek et al. reported that the end result was that the particles aligned
  on a single direction.

  The paper \cite{jad} provided a theoretical justification of the results in \cite{V95} by proving the
  convergence of a linearized version of the update model of \cite{V95}. The results of \cite{jad} are 
  very similar to the results in \cite{T84, TBA86}, modulo a number of minor modifications. The main difference
  appers to be that \cite{T84, TBA86} makes certain assumptions on the sequence $G(t)$ that are not made
  in \cite{jad}; these assumptions, however, are never actually used in the proofs of \cite{T84, TBA86}. We
  refer the reader to \cite{BT07} for a discussion. 

  The paper \cite{jad} has created an explosion of interest in averaging algorithms, and the subsequent 
  literature expanded in a number of directions. It is impossible to give a complete account of the literature
  since \cite{jad} in a reasonable amount of space, so we give only a brief overview of several research
  directions. 
  
  \bigskip
  
  \noindent {\bf Convergence in some natural geometric settings.} One interesting direction of research has been
  to analyze the convergence of some plausible geometric processes, for which there is no guarantee that any of
  the sufficient conditions for consensus (e.g. from \cite{TBA86} or \cite{jad}) hold. Notable in this direction
  was \cite{CS07}, which proved convergence of one such process featuring all-to-all communication
  with decaying strength. In a different direction, \cite{C09, C09b} give tight bounds on the convergence of averaging
  dynamics when the communication graph corresponds to nearest neighbors of points in $R^k$. 
  
  \bigskip
  
  \noindent {\bf General conditions for averaging.} A vast generalization of the consensus conditions in \cite{TBA86}
  and \cite{jad} was given in \cite{Moreau}. Using set-valued generalizations of the Lyapunov functions in \cite{TBA86, 
  jad}, it was shown that a large class of possible nonlinear maps lead to consensus. Further investigation of 
  these results was given in \cite{angeli} and \cite{LL10}.
  
  \bigskip
  
  \noindent {\bf Quantized consensus.} The above models assume that nodes can transmit real numbers to each other. 
  It is natural to consider quantized versions of the above schemes. One may then ask about the tradeoffs between
  storage and performance. A number of papers explored various aspects of this tradeoffs. In \cite{KBS06}, a simple
  randomized scheme for achieving approximate averaging was proposed. Further research along the same lines can be
  found in \cite{ZM09} and \cite{FCFZ08}. A dynamic scheme which allows us to approximately compute the average
  as the nodes communicate more and more bits with each other can be found in \cite{CBZ10}. 
  
  We will also consider this issue in this thesis, namely in Chapters \ref{qanalysis} and \ref{ch:constantstorage}, 
  which are based on the papers \cite{NOOT07} and \cite{HOT10}, respectively. In Chapter \ref{qanalysis}, we will give
  a recipe for quantizing any linear averaging scheme to compute the average approximately. In Chapter \ref{ch:constantstorage}, we will consider the problem of computing the averaging approximately with a deterministic
  algorithm, when each node can store only a constant number of bits for each link it maintains. 
  
  \bigskip
  
  \noindent {\bf Averaging with coordinates on geometric random graphs.} Geographic random graphs are common models 
  for sensor networks. It is therefore of interest to try to specialize results for averaging to the case of 
  geometric random graphs. Under the assumption that every node knows its own exact coordinates, an averaging algorithm
  with a lower than expected averaging cost was developed in \cite{DSW08}. Further research in \cite{BDTV07,TR10} reduced
  the energy cost even further. Substantial progress towards removing the assumption that each node knows its
  coordinates was recently made in \cite{OCR10}.
  
  \bigskip
  
  \noindent {\bf Design of fast averaging algorithms on fixed graphs.} It is interesting to consider the fastest 
  averaging algorithm for a given, fixed graph. It is hoped that an answer would give some insight into
  averaging algorithms: what their optimal speed is, how it relates to graph structure, and so on. In \cite{XB04},
  the authors showed how to compute optimal symmetric linear averaging algorithms with semidefinite programming. 
  Some further results for specific graphs were given in \cite{BDSX06,BDPX09}. Optimization over a larger class of
  averaging algorithms was consider in \cite{OT09} and also in \cite{KM09}. 
  
  \bigskip 
  
  \noindent {\bf Analysis of social networks.} We have already mentioned the work of DeGroot \cite{DG74} which was aimed at modeling the
  interactions of individual via consensus-like updates. A number of recent works has taken this line of analysis further by analyzing
  how the combinatorial structure of social networks affects the outcome. In particular, \cite{GJ10} studied how good social networks are at
  aggregating distributed information in terms of various graph-cut related quantities. The recent work \cite{AOP09} quantified the extent to
  which ``forceful'' agents which are not influenced by others interefere with information aggregation.

  \section{Applications of averaging}

  We give an incomplete list of the use of averaging algorithms in various applications.

  \begin{enumerate} 
  \renewcommand{\theenumi}{\alph{enumi}.}
\renewcommand{\labelenumi}{\theenumi}
\item Consider the following distributed estimation problem: a sensor network would like
  to estimate some unknown vector of parameters. At a discrete set of times, some sensors make noise-corrupted measurements of a linear  function of the unknown vector. The sensors would like to combine the measurements that are coming in
  into a maximum likelihood estimate. 
  
  We discuss a simpler version of this problem in the following chapter, and describe some
  averaging-based algorithms. In brief, other known techniques (flooding, fusion along a spanning tree) suffer
  from either high storage requirements or lack of robustness to link failures. The use of averaging-based
  algorithms allows us to avoid these downfalls, as we will explain in Chapter \ref{why}. For some literature
  on this subject, we refer the reader to \cite{XBL05, XBL06} and \cite{CA09}.
  
  \item Distributed state estimation: a collection of sensors are trying to estimate the (constantly evolving) state
of a linear dynamical system. Each sensor is able to periodically make a noise corrupted measurement of the system output.
The sensors would like to cooperate on synthesizing a Kalman filter estimate. 

There are a variety of challenges involving in such a problem, not least of which is the delay involved in a receiving 
at one node the measurements from other nodes which are many hops away. Several ideas have been presented in the literature
for solving this sort of problem based on averaging algorithms. We refer the reader to \cite{olfatikalman1, AR06, olfatikalman2, CCSZ08, DZ09}
  
  \item Coverage control is the problem of optimally positioning a set of robots to monitor
  an area. A typical case involves a polygon-shaped area along with robots which can measure distances to the boundary as well as to each other. Based on these distances, it is desirable to construct controllers which cover the entire area, yet assign as little area as possible to each robot. A common addition to this setup involves
  associating a number $f(x)$ to each point $x$ in the polygon, representing the importance of monitoring
  this point. The robots then optimize a corresponding objective function which weights regions according to
  the importance of the points in them.

  It turned out that averaging algorithms have proven very useful in designing distributed controllers for such systems. We refer the reader to \cite{GCB08} for the connection between distributed controllers for these systems
  and averaging over a certain class of graphs defined by Voronoi diagrams. A similar approach was adopted in \cite{SSS08}. Note, also, the related paper  \cite{LSTB10} and the extension to nonuniform coverage in \cite{LL09}.
  
  \item Formation control is the problem of maintain a set formation, defined by a collection of relative distances, against random
  or adversarial disturbances. Every once in a while pairs of agents manage to measure the relative offset between them. The challenge is for the agents to use these measurements and take actions so tht in the end everyone gets into formation. 
  
  We discuss this problem in greater detail in Chapter \ref{why}, where we explain several formation control ideas originating in
  \cite{olfatisurvey}. Averaging theorems can show the possibility of working formation control in this setting, subject only to
  very intermittent communication.

  \item The task assignment problem consists in distributing a set of tasks among a collection of agents. This arises,   for example, in the case of a group of aircraft or robots who would like to make decisions autonomously without communication with a common base. A typical example is dividing a list of locations to be monitored among a group of aircraft.

  In such cases, various auction based methods are often used to allocate tasks.  Averaging and consensus algorithms provide the means by which these auctions are implemented in a distributed way;  we refer the reader to the papers \cite{CBH09, WYX10} for details.
  
  \item Consider a collection of clocks which are constantly drifting apart. This is a common scenario, because clocks drift randomly
  depending on various factors within their environment (e.g. temperature), and also because clocks have a (nonzero) drift relative to the ``true" time. Maintaining a common time as much as possible is important for a number of estimation problems (for example, direction
  of arrival problems). 
  
  An important problem is to design distributed protocols to keep the clocks synchronized. These try to keep clock drift to a minimum,
  at least between time periods when at outside source can inform each node of the correct time. This problem has a natural similarity
  to averaging, except that one does not care very much getting the average right, but rather agreement on any time will do. Moreover, 
  the constantly evolving times present a further challenge.
  
  A natural approach, explored in some of the recent literature, is to adopt averaging techniques to work in this setting. We refer
  the reader to \cite{SG07, CCSZ07, B07, FK07, XK09}. 
  \end{enumerate}
  
  \section{Main contributions}

This thesis is devoted to the analysis of the convergence time of averaging schemes and to the degradation in
performance as a result  
of quantized communication. What follows is a brief summary of our contributions by chapter. 

Chapters \ref{why} and \ref{basic} are introductory. We begin with Chapter \ref{why} which seeks to motivate
the practical use of averaging algorithms. We compare averaging algorithms to other ways of aggregating information,
such as flooding and leader-election based methods, and discuss the various advantages and disadvantages. Our main
point is that schemes based on distributed averaging posess two unique strenghts: robustness to link failures and
economical storage requirements at each node. 

Next, in Chapter \ref{basic}, we discuss the most elementary known results on the convergence of averaging methods. The rest of this thesis will be spent on improving and refining the basic
results in this chapter. 

In Chapter \ref{chapter:poly} we give an exposition of the first polynomial-time convergence bound on the convergence 
time of averaging algorithms. Previously known bounds, such as those described in Chapter \ref{basic}, took exponentially
many steps in the number of nodes $n$ to converge, in the worst case. In the subsequent Chapter \ref{nsquared}, we give an averaging algorithm whose convergence time scales as $O(n^2)$ steps on nearly arbitrary time-varying graph sequences. This is the
currently best averaging algorithm in terms of convergence time bounds. 

We next wonder if it is possible to design averaging algorithms which improve on this quadratic scaling. In Chapter \ref{ch:optimality}, we prove that it is in fact impossible to beat the $n^2$ time steps bound within a large class of (possibly nonlinear) update schemes. The schemes we consider do not exhaust all possible averaging algorithms, but they do encompass the 
majority of averaging schemes proposed thus far in the literature. 

We then move on to study the effect of quantized communication and storage. Chapter \ref{qanalysis} gives a
recipe for quantizing any linear averaging scheme. The quantization performs averaging while storing and
transmitting only $c \log n$ bits. It is shown that this quantization preserves the convergence time bounds of the
scheme, and moreover allows one to compute the average to any desired accuracy: by picking $c$ large (but not dependent on $n$), one can make the final result be as close to the average as desired. 

In Chapter \ref{ch:constantstorage}, we investigate
whether it is possible to push down the $\log n$ storage down even futher; in particular, we show how to the average may be approximately with a deterministic algorithm in which each node stores only
a constant number of bits per every connection it maintains. An algorithm for fixed graphs is given; the dynamic graph case remains 
an open question. 

Finally, Chapter \ref{ch:function} tackles the more general question: which functions can be computed with a decentralized
algorithm which uses a constant number of bits per link? The chapter assumes a consensus-like termination requirement in which the nodes only have to get the right answer eventually, but are not required to know when they have
done so. The main result is a nearly tight characterization of the functions which can be computed
deterministically in this setting.

\chapter{Why averaging? \label{why}}

Our goal in this chapter is to motivate the study of distributed
averaging algorithms. We will describe two settings in which
averaging turns out to be singularly useful. The first is an
estimation problem in a sensor network setting; we will describe an
averaging-based solution which avoids the pitfalls which plague
alternative schemes. The second is a formation maintenance problem;
we will show how basic theorems on averaging allow us to establish
the satisfactory performance of some formation control schemes.

\section{A motivating example: distributed estimation in sensor networks}



\bc \begin{figure}[h] \bc \label{fig:sensors}
        \epsfig{file=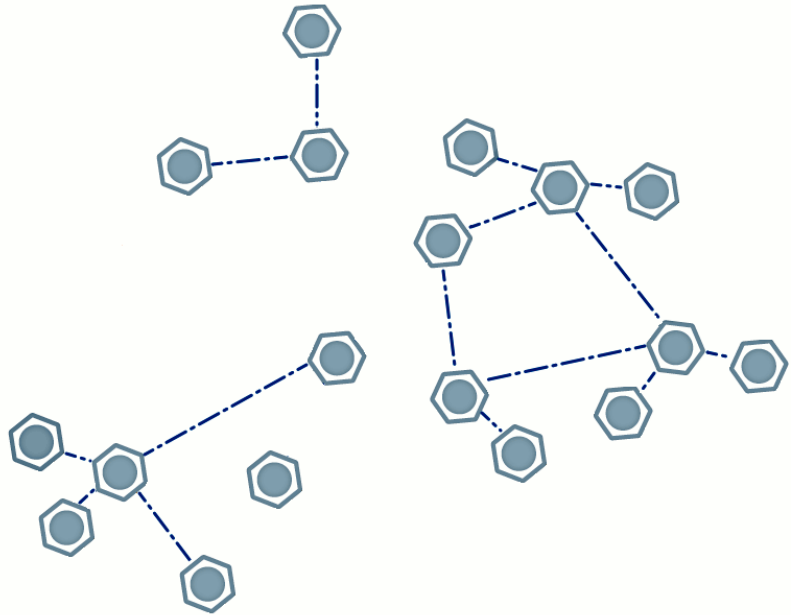,width=11cm} \caption{The set of
online links at some given time. } \ec
  \end{figure} \ec
  
Consider a large collection of sensors, $1, \ldots, n,$ that want to
estimate an unknown parameter $\theta \in R^k$. Some of these
sensors are able to measure a noise corrupted version of $\theta$;
in particular, all nodes $i$ in some subset $S \subset
\{1,\ldots,n\}$ measure \[ x_i = \theta + w_i.\] We will assume, for
simplicity, that that the noises $w_i$ are jointly Gaussian and
independent at different sensors. Moreover, only node $i$ knows the
statistics of its noise $w_i \sim N(0, \sigma_i)$.

It is easy to see that the maximum likelihood estimate is given by
\[ \hat{\theta} =  \frac{\sum_{i \in S} x_i /\sigma_i^2}{\sum_{i \in
S} 1/\sigma_i^2}. \] Note that if $S=\{1,\ldots,n\}$ (i.e. every
node makes a measurement), and all the variances $\sigma_i^2$ are
equal, the maximum likelihood estimate is just the average
$\hat{\theta} = (1/n) \sum_{i=1}^n x_i$.

The sensors would like to compute $\hat{\theta}$ in a distributed way. We do not
assume the existence of a ``fusion center'' to which the sensors can transmit measurements; rather,
the sensors have to compute the answer by exchanging messages with their neighbors and performing computations.

The sensors face an additional problem: there are communication links available through which they can exchange messages, but these links are unreliable.  In particular, links fail and come online in unpredictable ways.  For example, there is no guarantee that any link will come online if the sensors wait long enough. It is possible for a link to be online for some time, and then fail forever. Figure \ref{fig:sensors} shows an example of what may happen: at any
given time, only some pairs of nodes may exchange messages, and the network is effectively split into
disconnected clusters.

More concretely, we will assume a discrete sequence of times
$t=1,2,3, \ldots,$ during which the sensors may exchange messages.
At time $t$, sensor $i$ may send a message to its neighbors in the {\em
undirected} graph $G(t)=(\{1, \ldots, n\}, E(t))$. We will also
assume that the graph $G(t)$ includes all self loops $(i,i)$. The
problem is to devise good algorithms for computing $\hat{\theta}$,
and to identify minimal connectivity assumptions on the sequence
$G(t)$ under which such a computation is possible.

\subsection{Flooding}

We now describe a possible answer. It is very plausible to make an additional assumption that sensors
possess unique identifiers; this is the case in almost any wireless system. The sensors can use these
identifiers to ``flood" the network so that eventually, every sensor knows every single measurement
that has been made.

At time $1$, sensor $i$ sends its own triplet
$({\rm id}, x_i, \sigma_i^2)$ to each of its neighbors. Each sensor stores all the messages it has received.
Moreover, a sensor maintains a ``to broadcast'' queue, and each time it hears a message with an ${\rm id}$ it has not heard
before, it adds it to the tail of the queue. At times $t=2,3,\ldots$, each sensor broadcasts the top message from its queue.

If $G(t)$ is constant with time, and connected, then eventually each
sensor learns all the measurements that have been made. Once that
happens, each sensor has all the information it needs to compute the
maximum likelihood estimate. Moreover, the sensors do not even need
to try to detect whether they have learned everything; each sensor
can simply maintain an estimate of $\hat{\theta}$, and revise that
estimate each time it learns of a new measurement.

If $G(t)$ is not fixed, flooding can still be expected
to work. Indeed, each time a link appears, there is opportunity for a piece of information to be learned.  One can show that subject to only very minimal requirements on connectivity,
every sensor eventually does learn every measurement.

Let us state a theorem to this effect. We will use the notation
$\cup_{t \in X} G(t)$ to mean the graph obtained by forming the
union of the edge sets of $G(t), t \in X$, i.e. \[ \bigcup_{t \in X}
G(t) =  \left( \{1,\ldots,n\}, \bigcup_{t \in X} E(t) \right). \]  A relatively
light connectivity assumption is the following.

\begin{assumption} \label{assumpt:infinitec}(Connectivity) The graph \[ \cup_{s \geq t} G(s) \] is connected for every $t$. \end{assumption}

In words, this assumption says that the graph sequence $G(t)$ has
enough edges for connectivity, and that moreover this remaisn true
after some finite set of graphs is removed from the sequences. 

\bigskip

\begin{theorem} If Assumption \ref{assumpt:infinitec} holds, then under the flooding protocol every
node eventually learns each triplet $({\rm id}, x_i, \sigma_i^2)$.
\end{theorem}

\bigskip

\begin{proof} (Sketch). Suppose that some triplet $({\rm id}, x_i, \sigma_i^2)$ is not learned by node $j$. Let
$A$ be the nonempty set of nodes that do learn this triplet; one can easily argue that the number of edges in
the graphs $E(t)$ between $A$ and $A^c$ is finite. But this contradicts Assumption \ref{assumpt:infinitec}.
\end{proof}

It is possible to relax the assumption of this theorem: $\cup_{s
\geq t} G(s)$ actually only needs to be connected for a sufficiently
long but finite time interval.

The problem with flooding, however, lies with its storage
requirements: each sensor needs to store $n$ pieces of information,
i.e., it needs to store a list of id's whose measurements it has
already seen. This means that the 
total amount of storage throughout the network is at least on the
order of $n^2$. 

We count the storage requirements of the numbers $x_i, \sigma_i^2$ as constant. When dealing with estimation problems,
it is convenient to assume that $x_i, \sigma_i^2$ are real numbers, and we will maintain this technical assumption for now.
Nevertheless, in practice, these numbers will be truncated to some fixed number of bits (independent of $n$). Thus
it makes sense to think of transmitting each of the $x_i, \sigma_i^2$ as incurring a fixed cost independent of $n$.

Thus tracing out the dependence on $n$, we have that at least $n^2$ bits must be stored, in addition to a fixed number of bits
at each node to maintain truncated versions of $x_i, \sigma_i^2$ and
the estimate $\hat{\theta}$.  One hopes for the existence of a
scheme whose storage requirements scale more gracefully with the
number of nodes $n$.

\subsection{Leader election based protocols}

The protocol we outline next has considerably nicer storage requirements.
On the other hand, it will require some stringent assumptions on connectivity.
 We describe it next for the case of a fixed communication graph, i.e.,  when the graph
sequence $G(t)$ does not depend on $t$.

First, the sensors elect one of them as a leader. There are various
protocols for doing this. If sensors have unique identifiers they
may pick (in a distributed way) the sensor with the largest or
smallest id as the leader. Even in the absence of identifiers, there
are randomized protocols for leader election (see \cite{AM94}) which
take on the order of the diameter\footnote{The diameter of the graph
$G$, denoted by $d(G)$, is the largest distance between any two
nodes.} of the network time steps, provided that messages of size
$O(\log n)$ bits can be sent at each time step.

Next, the sensors can build a spanning tree with the leader as the
root. For example, each sensor may pick as its parent the node one
hop closer to the root. Finally, the sensors may forward all of
their information (i.e., their $(x_i, \sigma_i^2)$) to the root,
which can compute the answer and forward it back.

Our description of the algorithm is deliberately vague, as the
details do not particularly matter (e.g., which leader election
algorithm is used). We would like to mention, however, that it is
even possible to avoid having the leader learn all the information
$(x_i, \sigma_i^2)$. For example, once a spanning tree is in place,
each node may wait to hear from all of its children, and then
forward to the leader a sufficient statistic for the measurements in
its subtree.

We state the existence of such protocols as a theorem:

\begin{theorem} If $G(t)$ does not depend on time, i.e., $G(t)=G$ for all $t$,
it is possible to compute $\hat{\theta}$ with high probability in $O(d(G))$ time steps.
The nodes need to store and forward messages of size $O(\log n)$ bits in each time step,
as well a constant number of real numbers which are smooth functions of the measurements $x_i, \sigma_i^2$.
\end{theorem}

Observe that that the theorem allows for the existence of protocols which (only) work with high probability; 
this is due to the necessarily randomized nature of leader election protocols in the absence of identifiers.

The above theorem is nearly optimal for any fixed graph. In other words,
if the communication links are reliable, there is no reason to choose anything but the above algorithm.

On the other hand, if the graph sequence $G(t)$ is changing unpredictably, this sort of approach immediately runs into problems.
Maintaining a spanning tree appears to be impossible if the graph sequence changes dramatically from step to step, and other approaches are needed.

\section{Using distributed averaging}


We now describe a scheme which is both robust to link failures (it only needs Assumption \ref{assumpt:infinitec} to work) and has nice storage
requirements (only requires nodes to maintain a constant number of real numbers which are smooth functions of the
data $x_i, \sigma_i^2$). However, it needs the additional assumption that the graphs $G(t)$ are undirected. This condition is often satisfied in practice, for example if the sensors are connected by
links whenever they are within a certain distance of each other.

First, let us introduce some notation. Let $N_i(t)$ be the set of
neighbors of node $i$ in the graph $G(t)$. Recall that we assume self-arcs are always 
present, i.e., $(i,i) \in E(t)$ for all $i,t$, so that $i \in N_i(t)$ for all $i,t$. Let $d_i(t)$ be the degree
of node $i$ in $G(t)$.

Let us first describe the scheme for the case where $S=\{1,\ldots,n\}$, i.e., every node makes a measurement,
and all $\sigma_i^2$ are the same. In this case, the maximum likelihood estimate is just the
average of the numbers $x_i$: $\hat{\theta}=(1/n) \sum_{i=1}^n x_i$.

The scheme is as follows. Each node sets $x_i(0)=x_i$ and updates as
\begin{equation} \label{basicupdate} x_i(t+1) = \sum_{j \in N_i(t)} a_{ij}(t) x_j(t),\end{equation} where
\begin{eqnarray*} a_{ij}(t) & = & \min \left( \frac{1}{d_i(t)}, \frac{1}{d_j(t)} \right), ~~~\mbox{ for } j \in N_i(t), j \neq i \\
& = & 0, ~~~~~~~~~~~~~~~~~~~~~~~~~~~\mbox{ otherwise}. \\
 a_{ii}(t) & = & 1 - \sum_{j \in N_i(t)} a_{ij}.
\end{eqnarray*} Then: \begin{proposition} \label{convergence} If Assumption \ref{assumpt:infinitec} holds and all the graphs $G(t)$ are undirected, then
\[ \lim_{t \rightarrow \infty} x_i(t) = \frac{1}{n} \sum_{i=1}^n x_i = \hat{\theta}.\]
\end{proposition} The above proposition is true because it is a special case of the following theorem: \begin{theorem}  \label{thm:convergence} Consider the iteration \[ x(t+1) = A(t) x(t) \] where: \begin{enumerate} \item $A(t)$ are doubly stochastic\footnote{A
matrix is called doubly stochastic if it is nonnegative and all of its rows and columns add up to $1$.} matrices.
\item If $(i,j) \in G(t)$, then $a_{ij}>0$ and $a_{ji}>0$, and if $(i,j) \notin G(t)$, then $a_{ij}=a_{ji}=0$.
\item There is some $\eta > 0$ such that if $a_{ij}>0$ then $a_{ij}>\eta$.
\item The graph sequence $G(t)$ is undirected.
\item Assumption \ref{assumpt:infinitec} on the graph sequence $G(t)$ holds. \end{enumerate}
Then: \[ \lim_{t \rightarrow \infty} x_i(t) = \frac{1}{n} \sum_{j=1}^n x_j(0),\] for all $i$.
\end{theorem} We will prove this theorem in Chapter \ref{basic}. In this form, this theorem is a trivial
modification of the results in \cite{Li-Wang, Cao-Morse-Anderson, Moreau,H-VB, BHOT05} which themselves are
based on the earlier results in \cite{TBA86,jad}.

Accepting the truth of this theorem for now, we can conclude that we have described
a simple way to compute $\hat{\theta}$. Every node just needs to store and update a
single real number $x_i(t)$. Nevertheless, subject to only the weak connectivity Assumption \ref{assumpt:infinitec}, every $x_i$
approaches the correct $\hat{\theta}$.  This scheme thus manages to avoid the downsides that plague flooding and leader election (high storage, lack of
robustness to link failures).

Let us describe next how to use this idea in the general case where
$S$ is a proper subset of $\{1,\ldots,n\}$ and the $\sigma_i^2$ are
not all equal.  Each node sets $x_i(0)=x_i/\sigma_i^2$,
$y_i(0)=1/\sigma_i^2$, and $z_i(0)=x_i(0)/y_i(0)$. If the node did
not make a measurement, it sets $x_i(0) = y_i(0) = 0$, and leaves
$z_i(0)$ undefined. Each node updates as
\begin{eqnarray*} x_i(t+1) & = & \sum_{j \in N_i(t)} a_{ij}(t) x_j(t) \\
y_i(t+1) & = & \sum_{j \in N_i(t)} a_{ij}(t) y_j(t) \\
z_i(t+1) & = &  \frac{x_i(t)}{y_i(t)}
\end{eqnarray*} Observe that and we have:
\begin{proposition} If Assumption \ref{assumpt:infinitec} holds, then
\[ \lim_{t \rightarrow \infty} z_i(t) = \hat{\theta}.\]
\end{proposition}

\bigskip

\begin{proof} By Theorem \ref{convergence},
\begin{eqnarray*} \lim_{t \rightarrow \infty} x_i(t) & = & (1/n) \sum_{i \in S} x_i/\sigma_i^2 \\
\lim_{t \rightarrow \infty} y_i(t) & = & (1/n) \sum_{i \in S} 1/\sigma_i^2 \\
\end{eqnarray*} and so
\[ \lim_{t \rightarrow \infty} z_i(t) = \frac{\lim_{t \rightarrow \infty} x_i(t)}{\lim_{t \rightarrow \infty} y_i(t)} =
\frac{\sum_{i \in S} x_i/\sigma_i^2}{\sum_{i \in S} 1/\sigma_i^2} = \hat{\theta}. \]
\end{proof}

The punchline is that even this somewhat more general problem can be
solved by an algorithm that relies on Theorem \ref{convergence}.
This solution has nice storage requirements (nodes store only a
constant number of real numbers which are smooth functions of $y_i,
\sigma_i^2$) and is robust to link failures.

\section{A second motivating example: formation  control}

We now give another application of Theorem \ref{convergence}, this time to a certain formation control
problem. Our exposition is based on \cite{olfatisurvey}.

Suppose that the nodes have real positions $x_i(t)$ in $R^k$; the
initial positions $x_i(0)$ are arbitrary, and the nodes want to move
into a formation characterized by positions $p_1,\ldots,p_n$ in
$R^k$ (formations are defined up to translation). This formation is
uniquely characterized by the offset vectors $r_{ij}=p_i-p_j$. We
assume that at every $t$, various pairs of nodes $i,j$
succeed in measuring the offsets $x_i(t) - x_j(t)$. Let $E(t)$ be
the set of (undirected) edges $(i,j)$ corresponding to the  measurements.
The problem is how to use these intermittent measurements to get
into the desired formation.

Let us assume for simplicity that our $x_i(t)$ lie in $R$ (we will dispense with this assumption shortly). A very natural idea is to perform gradient descent on the function
\[ \sum_{(i,j) \in E(t)} (x_i(t) - x_j(t) - r_{ij})^2. \] This leads to the following control law:
\begin{equation} \label{update} x_i(t+1) = x_i(t) - 2\Delta \sum_{j \in N_i(t)} \left( x_i(t) - x_j(t) \right) ~+~ 2 \Delta \sum_{j \in N_i(t)} r_{ij},\end{equation}
where $\Delta$ is the stepsize. Essentially, every node repeatedly ``looks around'' and moves to a new position
depending on the positions of its neighbors and its desired offset vectors $r_{ij}$.

Defining $b_i(t) = 2  \Delta \sum_{j \in N_i(t)} r_{ij}$, the above equations may be rewritten as
\[ x(t+1) = A(t) x(t) + b(t).\] Now let $z$ be any translate of the given positions $(p_1,\ldots,p_n)$. Then, $z$ satisfies
\[ z = A(t) z + b(t),\] because the gradient at $x_i=z_i$ equals $0$. Subtracting the two
equations we get
\[ x(t+1) - z = A(t) (x(t) -z).\] Observe that if $\Delta \leq 1/(2n)$, then the above matrix is nonnegative, symmetric,
and has rows that add up to $1$. Applying now Theorem \ref{convergence}, we get the following statement:

\begin{proposition} \label{formation} Suppose that the nodes implement the iteration of Eq. (\ref{update}). If:
\begin{enumerate} \item $z'$ is the translate of $p$ whose
average is the same as the average of $x_i(0)$.
\item The communication graph sequence $G(t)$ satisfies Assumption 1 (connectivity).
\item $\Delta \leq 1/(2n)$
\end{enumerate}
Then, \[ \lim_{t \rightarrow \infty} x_i(t) = z_i',\] for all $i$.
\end{proposition} The proof is a straightforward application of Theorem \ref{convergence}. This theorem
tells us that subject to only minimal conditions on connectivity, the scheme of Eq. (\ref{update}) will
converge to the formation in question.

We remark that we can replace the assumption that the $x_i$
are real numbers with the assumption that the $x_i$ belong to $R^k$. In
this case, we can apply the control law of Eq. (\ref{update}), which
decouples along each component of $x_i(t)$, and apply Proposition
\ref{formation} to each component of $x_i(t)$.

Finally, we note that continuous-time versions of these updates may be
presented; see \cite{olfatisurvey}. 

\section{Concluding remarks}

Our goal in this chapter has been to explain why averaging algorithms are useful. 
We have described averaging-based algorithms for estimation and formation problems
which are robust to link failures and have very light storage requirements at each node. 

Understanding the tradeoff between various types of distributed algorithms is still
very much an open question, and the discussion in this chapter has only scratched the surface
of it. One might additionally wonder how averaging algorithms perform on a variety of other
dimensions: convergence time, energy expenditure, robustness to node failures, performance degradation with noise, 
and so on. 

Most of this thesis will be dedicated to exploring the question of convergence time. 
In the next chapter we will give a basic introduction to averaging algorithms, and in
particular we will furnish a proof of Theorem
\ref{convergence}. In the subsequent chapters, will turn to the question of
designing averaging algorithms with good convergence time guarantees.

\chapter{The basic convergence theorems \label{basic}}

\aoc{Here we begin our analysis of averaging algorithm by proving
some basic convergence results. Our main goal is to prove Theorem
\ref{convergence} from the previous chapter, as well as some natural
variants of it. Almost all of this thesis will be spent refining and
improving the simple results obtained by elementary means in this
chapter. The results we will present may be found in \cite{T84,
TBA86, BT89, jad, Li-Wang,BHOT05}. We also recommend the paper
\cite{Moreau} for a considerable generalization of the results found
here. The material presented here appeared earlier in the paper
\cite{BHOT05} and the M.S. thesis \cite{O04}.}

\section{\aoc{Setup and assumptions}}

\aoc{We consider a set $N=\{1,\ldots,n\}$ of nodes, each starting with a real number stored
in memory. The nodes attempt to compute the average of these numbers by broadcasting
these numbers and repeatedly combining them by forming convex combination. We will first
only be concerned with the convergnece of this process. }




Each node $i$ starts with a scalar value $x_i(0)$. The vector
$x(t)=(x_1(t),\ldots,x_n(t))$ with the values held by the nodes at
time $t$, is updated according to the equation $x(t+1)=A(t)x(t)$,
or
\begin{equation} \label{eq:basicupdate} x_i(t+1)=\sum_{j=1}^n a_{ij}(t)x_j(t),\end{equation}
where $A(t)$ is a nonnegative matrix with entries $a_{ij}(t)$, and where the updates are carried out at some discrete set of times
which we will take, for simplicity, to be the nonnegative integers.
\aoc{We will refer to this scheme as {\em the agreement algorithm.}}

We will assume that the row-sums of $A(t)$ are equal to
1, so that $A(t)$ is a stochastic matrix. In particular,
$x_i(t+1)$ is a weighted average of the values $x_j(t)$ held by
the nodes at time $t$. We are interested in conditions that
guarantee the convergence of each $x_i(t)$ to a constant,
independent of $i$.

Throughout, we assume the following.

\begin{assumption}[non-vanishing weights] \label{assumpt:weights} The matrix $A(t)$ is nonnegative, stochastic, and has
positive diagonal. Moreover, there exists some $\eta>0$ such that if $a_{ij}(t)>0$ then
$a_{ij}(t)>\eta$.
\end{assumption}

Intuitively, whenever $a_{ij}(t)>0$, node $j$ communicates its
current value $x_j(t)$ to node $i$. Each node $i$ updates its own
value by forming a weighted average of its own value and the values
it has just received from other nodes.

The communication pattern at each time step can be described in
terms of a directed graph $G(t)=(N,E(t))$, where $(j,i)\in E(t)$ if
and only if $a_{ij}(t)>0$. Note that $(i,i) \in E(t)$ for all $i,t$
since $A(t)$ has positive diagonal.  A minimal assumption is that
starting at an arbitrary time $t$, and for any $i$, $j$, there is a
sequence of communications through which node $i$ will influence
(directly or indirectly) the value held by node $j$. \aoc{This is
Assumption \ref{assumpt:infinitec} from the previous chapter. }
\smskip We note various special cases of possible interest.

\smskip \noindent {\bf Fixed coefficients:} There is a fixed matrix
$A$, with entries $a_{ij}$ such that, for each $t$, and for each $i
\neq j$, we have $a_{ij}(t)\in \{0\}\cup \{a_{ij}\}$ (depending on
whether there is a communication from $j$ to $i$ at that time). This
is the case presented in \cite{BT89}.

\smskip
\noindent
{\bf Symmetric model:} If $(i,j)\in E(t)$ then $(j,i)\in E(t)$. That
is, whenever $i$ communicates to $j$, there is a simultaneous
communication from $j$ to $i$.

\bigskip

\noindent
{\bf Equal neighbor model:}
Here,
$$a_{ij}(t)=\begin{cases} 1/d_i(t),& \mbox{ if } j\in
N_i(t),\cr
0,& \mbox{ if } j \notin N_i(t),\\ \end{cases}$$
This model is a linear version of a model considered by Vicsek et al.\
\cite{V95}. Note that here the constant $\eta$ of Assumption \ref{assumpt:weights} is
equal to $1/n$.

\smskip

\noindent \aoc{{\bf Metropolis model:}
Here,
$$a_{ij}(t)=\begin{cases} 1/\max(d_i(t), d_j(t)),& \mbox{ if } j\in
N_i(t), i \neq j\cr 0,& \mbox{ if } j \notin N_i(t),\\ \end{cases}$$
and
$$a_{ii}(t) = 1 - \sum_{j=1}^n a_{ij}(t).$$
The Metropolis model is similar to the equal-neighbor model, but has
the advantage of symmetry: $a_{ij}(t)=a_{ji}(t)$. }

\smskip

\smskip \noindent {\bf Pairwise averaging model (\cite{gossip}):} This is the
special case of both the symmetric model and of the equal neighbor model
in which, at each time, there is a set of disjoint pairs of nodes
who communicate with each other. If $i$
communicates with $j$, then $x_i(t+1)=x_j(t+1)=(x_i(t)+x_j(t))/2$.
Note that the sum $x_1(t)+\cdots+x_n(t)$ is conserved; therefore,
if consensus is reached, it has to be on the average of the
initial values of the nodes.

\smskip

\aoc{The assumption below is a strengthening of Assumption \ref{assumpt:infinitec} on connectivity. We will see that it is sometimes
necessary for convergence.}

\begin{assumption}[$B$-connectivity] \label{assumpt:boundedintervals} There exists an integer $B>0$
such that the directed graph \[ (N, E(kB) \cup E((k+1)B) \cup \cdots \cup E((k+1)B-1)) \]
is strongly connected for all integer $k \geq 0$.
\end{assumption}

\section{Convergence results in the absence of delays.}
We say that the agreement algorithm {\it guarantees asymptotic
consensus} if the following
holds: for every $x(0)$, and for every sequence $\{A(t)\}$ allowed by
whatever assumptions have been placed, there exists some $c$ such
that $\lim_{t\to\infty}x_i(t)=c$, for all $i$.

\smskip
\noindent
\begin{theorem} \label{thm:simplestconvergence} Under Assumptions \ref{assumpt:weights} (non-vanishing weights) and
\ref{assumpt:boundedintervals} ($B$-connectivity), the agreement
algorithm guarantees asymptotic consensus.
\end{theorem}

Theorem \ref{thm:simplestconvergence} may be found in \cite{jad}; a
slightly different version is in \cite{TBA86, T84}. \aoc{We next
give an informal account of its proof.}

\begin{proof}[Sketch of proof] \aoc{The proof has several steps.}

\smallskip

\noindent \aoc{{\bf Step 1:} Let us define the notion of a path in
the time-varying
 graph $G(t)$. A path $p$ from $a$ to $b$ of length $l$ starting at time $t$
 is a sequence of edges $(k_0,k_1), (k_1,k_2), \ldots, (k_{l-1},k_l)$ such that $k_0=a$, $k_l=b$, and
 $(k_0,k_1) \in E(t), (k_1, k_2) \in E(t+1), \ldots,$ and so on.  We will
 use $c(p)$ to denote the product \[ c(p) = \prod_{i=0}^{l-1} a_{k_{i}, k_{i+1}}.\]}

\aoc{Define \[ \Phi(t_1,t_2) = A(t_2-1) A(t_2-2) \cdots A(t_1), \]
and let the $(i,j)$'th entry of this matrix be denoted by
$\phi_{i,j}(t_1,t_2)$. The following fact can be established by
induction:
\[ \phi_{i,j}(t_1,t_2) = \sum_{\mbox{ paths } p \mbox{ from } i \mbox{ to } j \mbox{ of length } t_2 - t_1 \mbox{ starting at time } t_1} c(p).\]
A consequence is that if $\phi_{i,j}(t_1,t_2) > 0$, then Assumption
\ref{assumpt:weights} implies $\phi_{i,j}(t_1,t_2) \geq \eta^{t_2 -
t_1}$.}

\smallskip

\aoc{\noindent {\bf Step 2:} Assumptions \ref{assumpt:weights} and \ref{assumpt:boundedintervals} have the following implication: for any two nodes $i,j$, there is a path of length $nB$ that begins at $i$ and ends at $j$.}

\aoc{This implication may be proven by induction on the following statement: for any node $i$ there are at least $m$ distinct nodes $j$ such that there is a path of length $mB$ from $i$ to $j$. The proof crucially relies on Assumption \ref{assumpt:weights} which implies that all the self loops $(i,i)$ belong to every edge set $E(t)$.}

\smallskip

\aoc{\noindent {\bf Step 3:} Putting Steps 1 and 2 together, we get
that $\Phi(kB, (k+n)B)$ is a matrix whose every entry is bounded
below by $\eta^{nB}$. The final step is to argue that \[ Q_{n}
Q_{n-1} \cdots Q_1 x \] converges to a multiple of the all-ones
vector $\1$ for any sequence of matrices $Q_i$ having this property
and any initial vector $x$. This is true because for any such matrix
$Q_i$, \[ \max_k (Q_i x)_k - \min_k (Q_i x)_k \leq (1-\eta^{nB})
(\max_k x_k - \min_k x_k). \]}\end{proof}

\bigskip

In the absence of $B$-connectivity,  the algorithm does not
guarantee asymptotic consensus, as shown by Example 1 below
(Exercise 3.1, in p.\ 517 of \cite{BT89}). In particular,
convergence to consensus fails even in the special case of the equal
neighbor model. The main idea is that the agreement algorithm can
closely emulate a nonconvergent algorithm that keeps executing the
three instructions $x_1:=x_3$, $x_3:=x_2$, $x_2:=x_1$, one after the
other.

\smskip \noindent {\bf Example 1.} Let $n=3$, and suppose that
$x(0)=(0,0,1)$. Let $\epsilon_1$ be a small positive constant.
Consider the following sequence of events. Node 3 communicates to
node 1; node 1 forms the average of its own value and the received
value. This is repeated $t_1$ times, where $t_1$ is large enough so
that $x_1(t_1)\geq 1-\epsilon_1$. Thus, $x(t_1)\approx (1,0,1)$. We
now let node 2 communicates to node 3, $t_2$ times, where  $t_2$ is
large enough so that $x_3(t_1+t_2)\leq \epsilon_1$. In particular,
$x(t_1+t_2)\approx (1,0,0)$. We now repeat the above two processes,
infinitely many times. During the $k$th repetition, $\epsilon_1$ is
replaced by $\epsilon_k$ (and $t_1,t_2$ get adjusted accordingly).
Furthermore, by permuting the nodes at each repetition, we can
ensure that Assumption \ref{assumpt:infinitec} is satisfied. After
$k$ repetitions, it can be checked that $x(t)$ will be within
$1-\epsilon_1-\cdots -\epsilon_k$ of a unit vector. Thus, if we
choose the $\epsilon_k$ so that $\sum_{k=1}^{\infty}\epsilon_k<1/2$,
asymptotic consensus will not be obtained.

\smskip On the other hand, in the presence of symmetry, the
$B$-connectivity Assumption \ref{assumpt:boundedintervals} is
unnecessary. This result is proved in \cite{Li-Wang} and
\cite{Cao-Morse-Anderson} for the special case of the symmetric
equal neighbor model and in \cite{Moreau,H-VB}, for the more general
symmetric model. A more general result will be established in
Theorem \ref{thm:symmetric} below.

\smskip
\noindent
\begin{theorem} \label{thm:basicsymmetric} Under Assumptions \ref{assumpt:infinitec} and \ref{assumpt:weights}, and for the symmetric
model, the agreement algorithm guarantees asymptotic consensus.
\end{theorem}

\section{Convergence in the presence of delays.}
The model considered so far assumes that messages from one node to
another are immediately delivered. However, in a distributed
environment, and in the presence of communication delays, it is
conceivable that a node will end up averaging its own value with an
{\em outdated} value of another node. A situation of this type
falls within the framework of distributed asynchronous computation
developed in \cite{BT89}.

Communication delays are incorporated into the model as follows: when
node $i$, at time $t$, uses the value $x_j$ from another node, that
value is not necessarily the most recent one, $x_j(t)$, but rather an
outdated one, $x_j(\tau^i_j(t))$, where $0\leq \tau^i_j(t)\leq t$,
and where $t-\tau^i_j(t))$ represents communication and possibly
other types of delay. In particular, $x_i(t)$ is updated according to
the following formula:
\begin{equation}x_i(t+1)=\sum_{j=1}^n a_{ij}(t)
x_j(\tau^i_j(t)).
\label{eq:asyn}\end{equation}
We make the following assumption on the $\tau^i_j(t)$.

\begin{assumption}(Bounded delays) \label{assumpt:delays} (a) If $a_{ij}(t)=0$, then
$\tau^i_j(t)=t$.\\
(b) $\tau^i_i(t)=t$, for all $i$, $t$.\\
(c) There exists some $B>0$ such that
$t-B+1\leq \tau^i_j(t)\leq t$, for all $i$, $j$, $t$.
\end{assumption}

\smskip

Assumption \ref{assumpt:delays}(a) is just a convention: when $a_{ij}(t)=0$, the value
of $\tau^i_j(t)$ has no effect on the update. Assumption \ref{assumpt:delays}(b) is quite natural, since an node generally
has access to its own most recent value. Assumption \ref{assumpt:delays}(c)
requires delays to be bounded by some constant $B$. 

\smskip The next result, from \cite{T84,TBA86}, is a
generalization of Theorem \ref{thm:simplestconvergence}. The idea of the proof is similar to
the one outlined for Theorem \ref{thm:simplestconvergence}, except that we now define
$m(t)=\min_{i} \min_{s=t,t-1,\ldots,t-B+1} x_i(s)$ and
$M(t)=\max_{i} \max_{s=t,t-1,\ldots,t-B+1} x_i(s)$. For convenience, we will adopt the 
definition that $x_i(t)=x_i(0)$ for all negative $t$. Once more, one
shows that the difference $M(t)-m(t)$ decreases by a constant
factor after a bounded amount of time.
\begin{theorem} \label{thm:boundedstuff} Under Assumptions \ref{assumpt:weights}, \ref{assumpt:boundedintervals},
\ref{assumpt:delays} (non-vanishing weights, bounded
intercommunication intervals, and bounded delays), the agreement
algorithm with delays [cf.\ Eq.\ (\ref{eq:asyn})] guarantees
asymptotic consensus.
\end{theorem}

Theorem \ref{thm:boundedstuff} assumes bounded
intercommunication intervals and bounded delays. The example that
follows (Example 1.2, in p.\ 485 of \cite{BT89}) shows that
Assumption \ref{assumpt:delays}(d) (bounded delays) cannot be relaxed. This is the
case even for a symmetric model, or the further special case where
$E(t)$ has exactly four arcs $(i,i)$, $(j,j)$, $(i,j)$, and $(j,i)$ at any given time
$t$, and these satisfy  $a_{ij}(t)=a_{ji}(t)=1/2$, as in the
pairwise averaging model.

\smskip\noindent {\bf Example 2.} We have two nodes who initially
hold the values $x_1(0)=0$ and $x_2(0)=1$, respectively.
Let $t_k$ be an increasing sequence of times, with $t_0=0$ and $t_{k+1}-t_k\to\infty$.
If $t_k\leq t<t_{k+1}$,
the nodes update according to
\begin{eqnarray*} x_1(t+1)&=&(x_1(t)+x_2(t_k))/2,\\
x_2(t+1)&=&(x_1(t_k)+x_2(t))/2.
\end{eqnarray*}
We will then have
$x_1(t_1)=1-\epsilon_1$ and $x_2(t_1)=\epsilon_1$, where
$\epsilon_1>0$ can be made arbitrarily small, by choosing $t_1$
large enough. More generally, between time $t_k$ and $t_{k+1}$,  the absolute difference
$|x_1(t)-x_2(t)|$ contracts by a factor of $1-2\epsilon_k$, where the corresponding contraction
factors $1-2\epsilon_k$ approach 1. If the
$\epsilon_k$ are chosen so that $\sum_k \epsilon_k<\infty$, then
$\prod_{k=1}^{\infty}(1-2\epsilon_k)>0$, and the disagreement
$|x_1(t)-x_2(t)|$ does not converge to zero.

\smskip
According to the preceding example, the assumption of bounded delays
cannot be relaxed.
  On the other hand, the assumption of bounded intercommunication
intervals can be relaxed,
  in the presence of symmetry, leading to the following generalization
of Theorem \ref{thm:basicsymmetric}.

\begin{theorem} \label{thm:symmetric} Under Assumptions \ref{assumpt:infinitec}
(connectivity), \ref{assumpt:weights} (non-vanishing weights), and
\ref{assumpt:delays} (bounded delays), and for the symmetric
model, the agreement algorithm with delays [cf.\ Eq.\
(\ref{eq:asyn})] guarantees asymptotic consensus. \end{theorem}

\begin{proof} Let
\begin{eqnarray*}
M_i(t)&=&\max\{x_i(t),x_i(t-1),\ldots,x_i(t-B+1)\},\\
M(t)&=&\max_i M_i(t),\\
m_i(t)&=&\min\{x_i(t),x_i(t-1),\ldots,x_i   (t-B+1)\},\\
m(t)&=&\min_i m_i(t).
\end{eqnarray*}  Recall that we are using the convention that $x_i(t)=x_i(0)$ for
all negative $t$. An easy inductive argument,
as in p. 512 of \cite{BT89}, shows that the sequences $m(t)$ and
$M(t)$ are nondecreasing and nonincreasing, respectively. The
convergence proof rests on the following lemma. \smskip \noindent
\begin{lemma} \label{lemma:shrinkage} If $\ao{m(\tau-B)}=0$ and $M(\tau)=1$, then there exists
a time $\tau' \geq \tau$ such \JMH{that $M(\tau')-m(\tau'-B)\leq
1-\eta^{nB}$.} \end{lemma}
%
%
%
Given Lemma 1, the convergence proof is completed as follows.
Using the linearity of the algorithm, there exists a time
\JMH{$\tau_1$ such that $M(\tau_1)-m(\tau_1-B)\leq
(1-\eta^{nB})\ao{(M(B)-m(0))}$.} By applying Lemma 1, with $\tau$
replaced by $\JMH{\tau_1}$, and using induction, we see that for
every $k$ there exists a time $\tau_k$ such that
\JMH{$M(\tau_k)-\ao{m}(\tau_k-B)\leq
(1-\eta^{nB})^k(M(B)-m(0))$}, which converges to zero. This,
together with the monotonicity properties of $m(t)$ and $M(t)$,
implies that $m(t)$ and $M(t)$ converge to a common limit, which
is equivalent to asymptotic consensus. \end{proof}
%
%
%
%

\begin{proof}[Proof of Lemma \ref{lemma:shrinkage}] For $k=1,\ldots,n$, we say that ``Property
$P_k$ holds at time $t$'' if there exist at least
$k$ indices $i$ for which $m_i(t)\geq \eta^{kB}$.

We assume, without loss of generality, \JMH{that $ m(\tau-B)=0$ and}
\ao{$M(\tau)=1$}. Then, $m(t) \geq 0$ for all $t \geq \tau-B$ by the
monotonicity of $m(t)$. Furthermore, there exists some $i$ and some
$t' \in \ao{\{\tau-B+1, \tau-B+2,\ldots,\tau\}}$ such that $x_i(t')=1$. Using the
inequality $x_i(t+1)\geq \eta x_i(t)$, we obtain $m_i(t'+B)\geq
\eta^B$. This shows that there exists a time at which property
$P_1$ holds.

We continue inductively. Suppose that $k<n$ and that Property $P_k$
holds at some time $t$. Let $S$ be a set of cardinality $k$
containing indices $i$ for which $m_i(t)\geq \eta^{kB}$, and let
$S^{\rm c}$ be the complement of $S$.
Let $t'$ be the first time, greater than or equal to $t$, at which
$a_{ij}(t')\neq 0$, for some $j\in S$ and $i\in S^{\rm c}$ (i.e.,
an node $j$ in $S$ gets to influence the value of an node $i$ in
$S^{\rm c}$). Such a time exists
by the connectivity assumption (Assumption \ref{assumpt:infinitec}).

Note that between times $t$ and $t'$, the nodes $\ell$ in the set
$S$ only form convex combinations between the values of the nodes in
the set $S$ (this is a consequence of the symmetry assumption). Since
all of these values are bounded below by $\eta^{kB}$, it follows
that this lower bound remains in effect, and that $m_{\ell}(t')\geq
\eta^{kB}$, for all $\ell \in S$.

For times $s\geq t'$, and for every $\ell\in S$, we have
$x_{\ell}(s+1)\geq \eta x_{\ell}(s)$, which implies that
$x_{\ell}(s)\geq \eta^{kB}\eta^B$, for
$s\in\{t'+1,\ldots,t'+B\}$. Therefore, $m_{\ell}(t'+B)\geq
\eta^{(k+1)B}$, for all $\ell\in S$.

Consider now an node $i\in S^{\rm c}$ for which $a_{ij}(t') \neq 0$. We have
$$x_i(t'+1)\geq a_{ij}(t') x_j(\tau^i_j(t')) \geq \eta m_i(t') \geq
\eta^{kB+1}.$$
Using also the fact $x_i(s+1)\geq \eta x_i(s)$, we obtain that
$m_i(t'+B)\geq \eta^{(k+1)B}$. Therefore, at time $t'+B$, we
have $k+1$ nodes with $m_{\ell}(t'+B)\geq \eta^{(k+1)B}$
(namely, the nodes in $S$, together with node $i$). It follows that
Property $P_{k+1}$ is satisfied at time $t'+B$.

This inductive argument shows that there is a time \JMH{$\tau'$}
at which Property $P_n$ is satisfied. At that time
\JMH{$m_i(\tau')\geq \eta^{nB}$} for all $i$, which implies
that \JMH{$m(\tau')\geq \eta^{nB}$}. On the other hand,
$M(\tau'+B )\leq M(0)=1$, which proves that
\JMH{$M(\tau'+B)-m(\tau')\leq 1-\eta^{nB}$}. \end{proof}

\smskip

\aoc{Now that we have proved Theorem \ref{thm:symmetric}, let us give a variation of it
which will ensure not only convergence but convergence to the average.}

\aoc{\begin{assumption} \label{assumpt:ds}(Double stochasticity) The
matrix $A(t)$ is column-stochastic for all $t$, i.e., \[ \sum_{i=1}^n
a_{ij}(t) = 1,\] for all $j$ and $t$.
\end{assumption}}

\aoc{Note that Assumption \ref{assumpt:weights} ensures that the matrix $A(t)$ is only row-stochastic. The above assumption together
with Assumption \ref{assumpt:weights} ensures
that $A(t)$ is actually doubly stochastic.}

\aoc{\begin{theorem} \label{thm:averaging} Under Assumptions
\ref{assumpt:infinitec} (connectivity), \ref{assumpt:weights}
(non-vanishing weights), and \ref{assumpt:ds} (double
stochasticity), and for the symmetric model, the agreement algorithm
(without delays) satisfies
\[ \lim_{t \rightarrow \infty} x_i(t) = \frac{1}{n} \sum_{i=1}^n x_i(0). \] \end{theorem}}

\begin{proof} \aoc{By Theorem \ref{thm:symmetric}, every $x_i(t)$ converges to the same value. Assumption \ref{assumpt:ds} ensures that $\sum_{i=1}^n x_i(t)$ is preserved from iteration to iteration:
\[  \sum_{i=1}^n x_i(t+1) = \1^T x(t+1) = \1^T A(t) x(t) = \1^T x(t) = \sum_{i=1}^n x_i(t), \] where we use the double stochasticity of $A(t)$ to conclude that $\1^T A(t) = \1^T$. This immediately implies that the final limit is the average of the initial values.}
\end{proof}

\aoc{Finally, let us observe that Theorem
\ref{thm:convergence} from the previous chapter is a special case of the theorem we have just proved.}

\aoc{\begin{proof}[Proof of Theorem \ref{convergence}] Observe that the assumptions of Theorem \ref{thm:averaging} are
present among the assumptions of Theorem \ref{convergence}, except for Assumption \ref{assumpt:ds} which needs
to be verified. To argue that the matrices $A(t)$ in Theorem \ref{convergence} are doubly stochastic, we just observe
that they are symmmetric and stochastic.
\end{proof}}

\section{Relaxing symmetry}

The symmetry condition [$(i,j)\in E(t)$ iff $(j,i)\in E(t)$] used in
Theorem \ref{thm:symmetric} is somewhat unnatural in the presence of communication
delays, as it requires perfect synchronization of the update times. A
looser and more natural assumption is the following.

\begin{assumption}[Bounded round-trip times] \label{assumpt:roundtrip}
There exists some $B>0$ such that whenever
$(i,j)\in E(t)$, then there exists some $\tau$ that satisfies
$|t-\tau|<B$ and $(j,i)\in E(\tau)$.
\end{assumption}

\smskip Assumption \ref{assumpt:roundtrip} allows for protocols such
as the following. Node $i$ sends its value to node $j$. Node $j$
responds by sending its own value to node $i$. Both nodes update
their values (taking into account the received messages), within a
bounded time from receiving the other node's value. In a realistic
setting, with unreliable communications, even this loose symmetry
condition may be impossible to enforce with absolute certainty. One
can imagine more complicated protocols based on an exchange of
acknowledgments, but fundamental obstacles remain (see the
discussion of the ``two-army problem'' in pp.\ 32-34 of
\cite{coordinated_attack}). A more realistic model would introduce a
positive probability that some of the updates are never carried out.
(A simple possibility is to assume that each $a_{ij}(t)$, with
$i\neq j$, is changed to a zero, independently, and with a fixed
probability.) The convergence result that follows remains valid in
such a probabilistic setting (with probability 1). Since no
essential new insights are provided, we only sketch a proof for the
deterministic case.

\begin{theorem} \label{thm:relaxedsymmetry}
Under Assumptions \ref{assumpt:infinitec} (connectivity),
\ref{assumpt:weights} (non-vanishing weights), \ref{assumpt:delays}
(delays) and \ref{assumpt:roundtrip} (bounded round-trip times) the
agreement algorithm with delays [cf.\ Eq.\ (\ref{eq:asyn})]
guarantees asymptotic consensus.
\end{theorem}

\begin{proof}[Proof outline] A minor change is needed in the
proof of Lemma 1. In particular, we define $P_k$ as the event that
there exist at least $k$ indices $l$ for which $m_l(t) \geq
\eta^{2kB}$. It follows that $P_1$ holds at time $t=2B$.

By induction, let $P_k$ hold at time $t$, and let $S$ be the set
of cardinality $k$ containing indices $l$ for which $m_l(t) \geq
\eta^{2kB}$. Furthermore, let $\tau$ be the first time after time
$t$ that $a_{ij}(\tau) \neq 0$ where exactly one of $i,j$ is in
$S$. Along the same lines as in the proof of Lemma 1, $m_l(\tau) \geq
\eta^{2kB}$ for $l \in S$; since $x_l(t+1) \geq \eta x_l(t)$,
it follows that $m_l(\tau+2B) \geq \eta^{2(k+1)B}$ for each $l
\in S$. By our assumptions, exactly one of $i$,$j$ is in $S^c$. If
$i\in S^c$, then $x_i(\tau+1) \geq
a_{ij}(\tau) x_j (\tau_j^i(\tau)) \geq \eta^{2kB+1}$ and
consequently $x_i(\tau + 2B) \geq \eta^{2B-1} \eta^{2kB+1} =
\eta^{2(k+1)B}$. If  $j\in S^c$, then there must
exist a time $\tau_j \in \{ \tau +1, \tau+2,\ldots,\tau+B-1 \}$
with $a_{ji}(\tau_j) > 0$. It follows that:
\begin{eqnarray*} m_j(\tau+2B) & \geq  & \eta^{\tau+2B -
(\tau_j+1)} x_j(\tau_j+1)
\\
& \geq & \eta^{\tau+2B -
\tau_j-1} \eta x_i(\tau_j)\\
  & \geq & \eta^{\tau+2B -
\tau_j-1} \eta \eta^{\tau_j - \tau} \eta^{2kB} \\
& = & \eta^{2(k+1)B} \end{eqnarray*} Therefore, $P_{k+1}$ holds
at time $\tau + 2B$ and the induction is complete.
\end{proof}

\section{Concluding remarks}

\aoc{In this chapter, we have presented some basic convergence results on the
averaging iterations $x(t+1)=A(t)x(t)$. In particular, we proved Theorem \ref{thm:convergence}
from the previous chapter, as well as several variations of it involving asymmetry and
delays.}

\aoc{There is, however, one troubling feature of the results so far: the
convergence time bounds which follow from our proofs are quite large. We have
shown that after $nB$ steps, an appropriately defined measure of convergence
shrinks by a factor of $1-\eta^{nB}$. Considering that $\eta$ can be as small
as $1/n$ (for example, in the Metropolis model), this means that one must
wait $n^{nB} nB \log (1/\epsilon)$ steps for the same measure of convergence to shrink by $\epsilon$.
It goes without saying that this is an enormous number even for relatively small
$n,B$.}

\aoc{One might hope that the bounds we have derived are lax.
Unfortunately, one can actually construct examples of graph
sequences on which convergence takes time exponential in $n$. An
example may be found in \cite{OT09}; an example with an undirected
graph may be found in an unpublished manuscript by Cao, Spielman,
and Morse.}

\aoc{In the next several chapters, we will be concerned with the possibility of designing
averaging algorithms with better guarantees. A first goal is to replace the exponential scaling with
$n$ by a polynomial one.}

\chapter{Averaging in polynomial time \label{chapter:poly}}
\section{Convergence time \label{convtimesection}}

\aoc{In the previous chapter, we proved Theorem \ref{thm:averaging}, which
states that subject to a few
natural conditions the iteration \[ x(t+1) = A(t) x(t) \] results in
\[ \lim_{t \rightarrow \infty} x(t) = \frac{1}{n} \sum_{i=1}^n x_i(0). \] A special case of this
result is Theorem \ref{thm:convergence} stated earlier. The assumptions include: Assumption \ref{assumpt:infinitec}
ensuring that the graphs $G(t)$ contain
enough links, Assumption \ref{assumpt:weights} on the weights,
and the ``double stochasticity'' Assumption \ref{assumpt:ds}.  Finally, we
also had to assume we were in the ``symmetric model.'' Here our goal will be to
reproduce the result but with better bounds on convergence time, which scale polynomially,
rather than exponentially, in $n$. }

\aoc{For this, we will need to slightly strengthen our connectivity assumptions.
It is obvious that with Assumption \ref{assumpt:infinitec} no effective bounds on
convergence can hold since the sequence $G(t)$ may contain arbitrarily many
empty graphs. Thus we will replace Assumption \ref{assumpt:infinitec} with the
slightly stronger Assumption \ref{assumpt:boundedintervals}. On the positive side, with the
stronger Assumption \ref{assumpt:boundedintervals}, we can dispense with the assumption that
we are in the ``symmetric model.''}

As a convergence measure, we use the ``sample variance'' of a vector
$x \in \R^n$, defined as \[ V(x) = \sum_{i=1}^n (x_i - \bar{x} )^2,
\] where $\bar{x}$ is the average of the entries of $x$: \[ \bar{x} = \frac{1}{n} \sum_{i=1}^n x_i. \]

We are interested in providing an upper bound on the number of
iterations it takes for the ``sample variance'' $V(x(t))$ to
decrease to a small fraction of its initial value $V(x(0))$. \aoc{The main result of
this chapter is the following theorem.}

\aoc{\noindent \begin{theorem} \label{uqboundver1} Let Assumptions
\ref{assumpt:weights} (non-vanishing weights),
\ref{assumpt:boundedintervals} ($B$-connnectivity), and
\ref{assumpt:ds} (double stochasticity) hold. Then there exists an
absolute constant\footnote{\ao{We say $c$ is an absolute constant
when it does not depend on any of the parameters in the problem, in
this case $n,B, \eta,\epsilon$.}} $c$ such that we have
\[V({t})
\leq \epsilon V(0)\qquad \hbox{for all } t\ge {c} (n^2/\eta)B \log
(1/\epsilon).\]
\end{theorem}}

\smskip

\aoc{This is the ``polynomial time averaging'' result alluded to in
the title of this chapter. Our exposition follows the paper
\cite{NOOT07} where this material first appeared. Note that this
bound is exponentially better than the convergence time bound of
$O\left((1/\eta)^{nB} nB \log 1/\epsilon \right)$ which follow
straightforwardly from the arguments of the previous chapter.}

\aoc{We now proceed to the task of proving this theorem.} We
first establish some technical preliminaries that will be key in the
subsequent analysis. In particular, in the next subsection, we
explore several implications of the double stochasticity assumption
on the weight matrix $A(t)$.

\subsection{Preliminaries on doubly stochastic matrices}

We begin by analyzing how the sample variance $V(x)$ changes when
the vector $x$ is multiplied by a doubly stochastic matrix $A$. The
next lemma shows that $V(Ax) \leq V(x)$. Thus, under Assumptions
\ref{assumpt:weights} and \ref{assumpt:ds}, the sample variance $V(x(t))$ is nonincreasing in
$t$, and $V(x(t))$ can be used as a Lyapunov function.

\begin{lemma}
\label{vl} Let $A$ be a doubly stochastic matrix. Then,\footnote{In
the sequel, the notation $\sum_{i<j}$ will be used to denote the
double sum $\sum_{j=1}^n\sum_{i=1}^{j-1}$.} for all $x \in \R^n$,
\[ V(Ax) = V(x) - \sum_{i<j} w_{ij} (x_i - x_j)^2, \] where $w_{ij}$ is
the $(i,j)$-th entry of the matrix $A^T A$.
\end{lemma}

\begin{proof}
Let ${\bf 1}$ denote the vector in $\R^n$ with all entries equal to
$1$. The double stochasticity of $A$ implies
\[A {\bf 1}= {\bf 1}, \qquad {\bf 1}^T A = {\bf 1}^T. \]
Note that multiplication by a doubly stochastic matrix $A$ preserves
the average of the entries of a vector, i.e., for any $x\in\R^n$,
there holds
\[ \overline{Ax} = \frac{1}{n}\, {\bf 1}^T Ax
=\frac{1}{n}\, {\bf 1}^T x = \bar{x}. \] We now write the
quadratic form $V(x)-V(Ax)$ explicitly, as follows:
\begin{eqnarray} V(x) - V(Ax) & = &  (x - \bar{x} {\bf 1})^T (x-\bar{x}{\bf 1}) -
(Ax-\overline{Ax}{\bf 1})^T (Ax - \overline{Ax}{\bf 1}) \nonumber \\
& = &  (x - \bar{x} {\bf 1})^T (x-\bar{x}{\bf 1}) - (Ax-\bar{x}A{\bf
1})^T (Ax - \bar{x} A {\bf 1}) \nonumber \\ & = &  (x - \bar{x} {\bf
1})^T (I-A^T A) (x-\bar{x}{\bf 1}). \label{vdec} \end{eqnarray}

Let $w_{ij}$ be the $(i,j)$-th entry of $A^T A$.  
Note that $A^T A$ is symmetric and stochastic, so that
$w_{ij}=w_{ji}$ and $w_{ii} = 1 - \sum_{j \neq i} w_{ij}$. Then, it
can be verified that
\begin{equation} A^T A  = I - \sum_{i<j} w_{ij} (e_i - e_j) (e_i -
e_j)^T, \label{asquareform}
\end{equation} where $e_i$ is a unit vector with the $i$-th entry equal to 1,
and all other entries equal to 0 \ao{(see also \cite{XBK07} where a
similar decomposition was used)}.

By combining Eqs.\ (\ref{vdec}) and (\ref{asquareform}), we obtain
\begin{eqnarray*}
V(x) - V(Ax) & =&  (x - \bar{x} {\bf 1})^T
\Big(\sum_{i<j} w_{ij} (e_i - e_j) (e_i - e_j)^T\Big) (x-\bar{x}{\bf 1}) \nonumber \\
& = &  \sum_{i<j} w_{ij} (x_i - x_j)^2. 
\end{eqnarray*}
\end{proof} Note that the entries $w_{ij}(t)$ of $A(t)^T A(t)$ are nonnegative,
because the weight matrix $A(t)$ has nonnegative entries. In view of this,
Lemma \ref{vl} implies that
\[V(x(t+1))\le V(x(t))\qquad \hbox{for all }t.\]
Moreover, the amount of variance decrease is given by
\[V(x(t)) - V(x(t+1)) = \sum_{i<j} w_{ij}(t) (x_i(t) - x_j(t))^2.\]
We will use this result to provide a lower bound on the amount of
decrease of the sample variance $V(x(t))$ in between iterations.

Since every positive entry of $A(t)$ is at least $\eta$, it follows
that every positive entry of $A(t)^T A(t)$ is at least $\eta^2$.
Therefore, it is immediate that
\[ \hbox{if }\quad w_{ij}(t) > 0, \quad\hbox{then }\quad w_{ij}{(t)} \geq \eta^2. \]
In our next lemma, we establish a stronger lower bound. In
particular, we find it useful to focus not on an individual
$w_{ij}$, but rather on all $w_{ij}$ associated with edges $(i,j)$
{that} cross {a particular} cut in the graph $(N,\E(A^TA))$. For
such groups of $w_{ij}$, we prove a lower bound which is linear in
$\eta$, as seen in the following.

\begin{lemma} \label{db}
Let $A$ be a \alexo{row-}stochastic matrix with positive diagonal
\an{entries},
and assume that \an{the smallest positive entry in $A$} is at least $\eta$.
Also, let $(S^{-},S^{+})$ be a partition of the set
$N=\{1,\ldots,n\}$ into two disjoint sets. If
\[ \sum_{i \in S^{-}, ~j \in S^{+}} w_{ij} > 0, \] then
\[ \sum_{i
\in S^{-}, ~j \in S^{+}} w_{ij} \geq \frac{\eta}{2}. \]
\end{lemma}

\begin{proof}
Let $\sum_{i \in S^{-}, ~j \in S^{+}} w_{ij} >0$. From the
definition of the weights $w_{ij}$, we have $w_{ij}=\sum_k
a_{ki}a_{kj}$, which shows that there exist $i\in S^-$, $j\in S^+$,
and some $k$ such that $a_{ki}>0$ and $a_{kj}>0$. For either case
where $k$ belongs to $S^-$ or $S^+$, we see that there exists an
edge in the set $\E(A)$ that crosses the cut $(S^{-},S^{+})$. Let
$(i^*,j^*)$ be such an edge. Without loss of generality, we assume
that $i^* \in S^{-}$ and $j^* \in S^{+}$.


We define
\begin{eqnarray*} C_{j^*}^{+} & = & \sum_{i \in S^{+}} a_{j^*i}, \\
C_{j^*}^{-} & =&  \sum_{i \in S^{-}} a_{j^*i}. \end{eqnarray*} See
Figure \ref{cdef}(a) for an illustration.
Since $A$ is a \an{row}-stochastic matrix, we have
\[
C_{j^*}^{-} + C_{j^*}^{+}  =  1, 
\]
implying that at least one of the following is true:
\begin{eqnarray*}
\mbox{ Case (a): } & & C_{j^*}^{-} \geq \frac{1}{2}, \\
\mbox{ Case (b): } & & C_{j^*}^{+} \geq \frac{1}{2}.
\end{eqnarray*}
We consider these {two cases} separately. In both cases,
we focus on a subset of the edges and we use the fact
that the
elements $w_{ij}$ correspond to paths of length $2$, with one step
in $\E(A)$ and another in $\E(A^T)$.

\begin{center}
\begin{figure}
\hspace{1cm}
{\includegraphics[width=4cm]{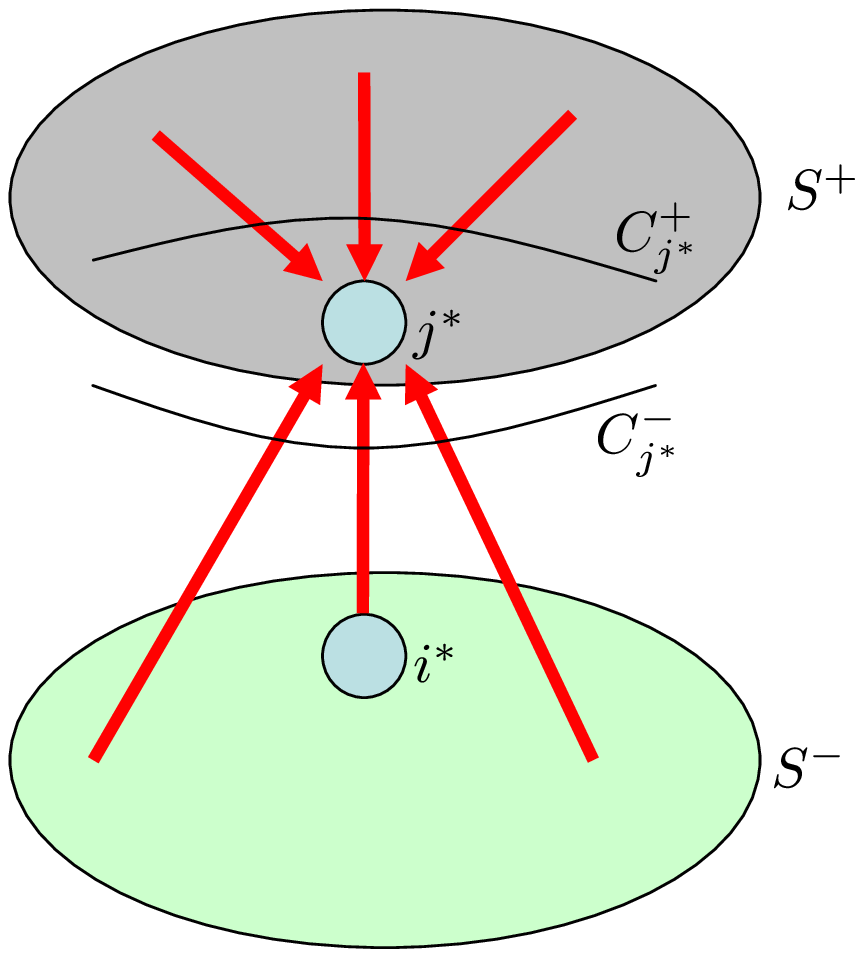} \quad
\includegraphics[width=4cm]{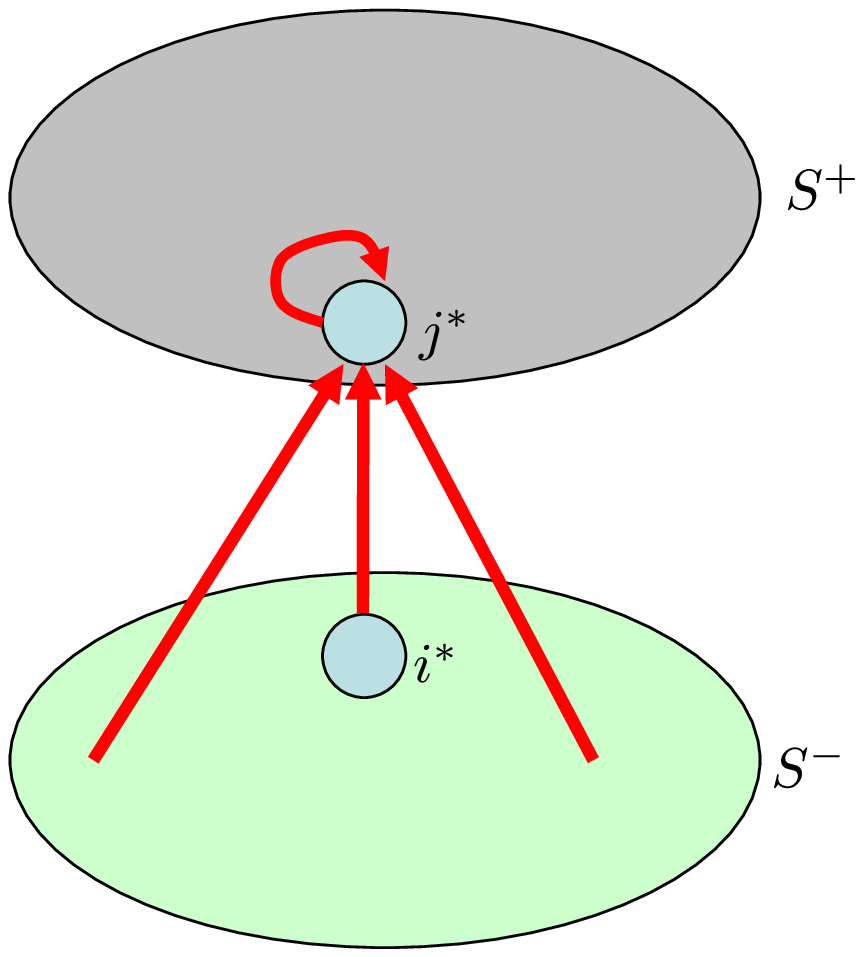}
\quad
\includegraphics[width=4cm]{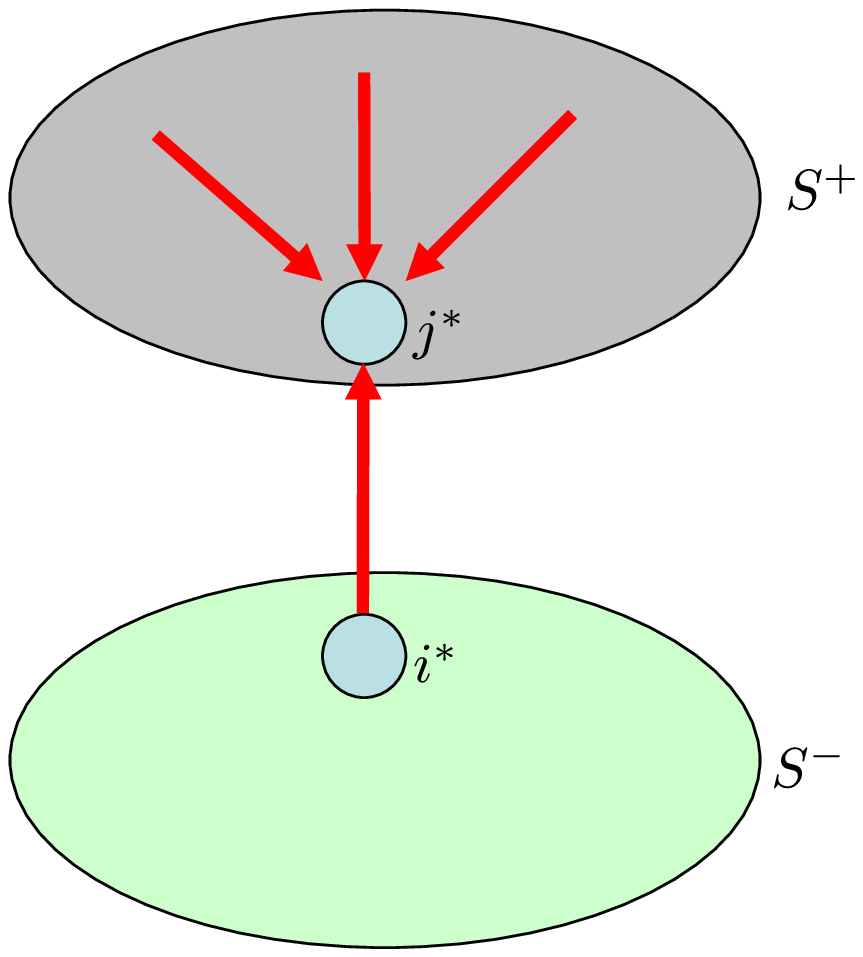}
} \caption{\label{cdef}\small (a) Intuitively,  $C_{j^*}^{+}$
measures how much weight $j^*$ assigns to nodes in $S^{+}$ (including
itself), and $C_{j^*}^{-}$ measures how much weight $j^*$ assigns to
nodes in $S^{-}$. Note that the {edge} $(j^*,j^*)$ is also present,
but not shown.
(b) For the case {where} $C_{j^*}^{-} \geq 1/2$, we
only focus on
two-hop paths between $j^*$ and elements $i \in S^{-}$
obtained by taking $(i,j^*)$ as the first step and the self-edge
$(j^*,j^*)$ as the second step.
(c) For the case where
$C_{j^*}^{+} \geq 1/2$, we
only focus on two-hop paths between $i^*$ and elements $j\in S^{+}$
obtained by taking $(i^*,j^*)$ as the first step in
$\E(A)$ and $(j^*,j)$ as the second step in $\E(A^T)$.}
\end{figure}
\end{center}

\noindent {\it Case (a):} $C_{j^*}^{-} \geq 1/2.$\\
We focus on those $w_{ij}$ with $i\in S^-$ and $j=j^*$. Indeed,
since all $w_{ij}$ are nonnegative, we have
\begin{equation} \label{wsubset} \sum_{i \in S^{-}, ~j \in S^{+}}
w_{ij} \geq \sum_{i \in S^{-}} w_{ij^*}.
\end{equation}
For each element in the sum on the right-hand {side}, we have
\[ w_{ij^*} = \sum_{k=1}^n a_{ki}\ a_{kj^*}
\geq a_{j^*i}\ a_{j^* j^*} \geq a_{j^*i}\ \eta, \] where the
inequalities follow from the facts that $A$ has
nonnegative entries, its diagonal entries are positive, and
its positive entries are at least $\eta$. Consequently,
\begin{equation} \sum_{i \in S^{-}} w_{ij^*}
\geq \eta\ \sum_{i \in S^{-}} a_{j^*i}
= \eta\, C_{j^*}^{-}. \label{wlowerbound}
\end{equation}
Combining Eqs.\ (\ref{wsubset}) and (\ref{wlowerbound}), and
recalling {the assumption} $C_{j^*}^{-} \geq 1/2$, the {result}
follows. {An} illustration of {this argument} can be found in Figure
\ref{cdef}{(b)}.

\noindent {\it Case (b):} $C_{j^*}^{+} \geq 1/2.$\\
We focus on those $w_{ij}$ with $i=i^*$ and $j \in S^{+}$.
We have
\begin{equation} \label{wsubset2}
\sum_{i \in S_{-}, ~j \in S^{+}} w_{ij} \geq \sum_{j \in
S^{+}} w_{i^* j},
\end{equation}
since all $w_{ij}$ are nonnegative. For each element in the sum on
the right-hand side, we have
\[ w_{i^* j} = \sum_{k=1}^n a_{ki^*}\ a_{k j}
\geq a_{j^* i^*}\ a_{j^* j} \geq \eta\, a_{j^* j}, \]
where the inequalities follow {because all entries of $A$ are
nonnegative, and because the choice $(i^*, j^*) \in \E(A)$ implies
that $a_{j^* i^*} \geq{\eta}$.} Consequently,
\begin{equation}
\sum_{j \in S^{+}} w_{i^* j} \geq \eta \sum_{j \in
S^{+}} a_{j^*j} = \eta\, C_{j^*}^{+}. \label{wlowerbound2}
\end{equation} {Combining Eqs.\ (\ref{wsubset2}) and}
(\ref{wlowerbound2}), and recalling {the assumption} $C_{j^*}^+ \geq
1/2$, the {result} follows. An illustration of this argument can
be found in Figure \ref{cdef}{(c)}.
\end{proof}

\subsection{A bound on convergence time}\label{convergtime}

With the preliminaries on doubly stochastic matrices in place, we
{can now proceed to derive} bounds on the decrease of $V(x(t))$ in
between iterations. We will first somewhat relax our connectivity
assumptions. In particular, we consider the following relaxation of
Assumption \ref{assumpt:boundedintervals}.\vspace{5pt}

\begin{assumption}[Relaxed connectivity] \label{weakconnect} Given an integer $t\ge 0$,
suppose that the components of $x(tB)$ have been reordered so that
they are in nonincreasing order. We assume that for every
$d\in\{1,\ldots,n-1\}$, we either have $x_d(tB)=x_{d+1}(tB)$, or
there exist some time $t\in\{tB,\ldots,(t+1)B-1\}$ and some
$i\in\{1,\ldots,d\}$, $j\in\{d+1,\ldots,n\}$ such that $(i,j)$ or
$(j,i)$ belongs to $\E(A(t))$.
\end{assumption}

\begin{lemma} \label{as2implas3}
Assumption \ref{assumpt:boundedintervals}
implies Assumption \ref{weakconnect}, {with the same value of $B$.}
\end{lemma}

\begin{proof} If Assumption \ref{weakconnect} does not hold,
then \ao{there must exist an index $d$
[for which $x_d(tB) \neq x_{d+1}(tB)$ holds] such that there are no
edges between nodes $1,2,\ldots,d$ and nodes $d+1,\ldots,n$ during
times $t=tB,\ldots,(t+1)B-1$. But this implies that the graph}
\[ \Big(N,\E({A}(tB)) \bigcup \E(A(tB+1)){\bigcup} \cdots \bigcup
\E(A((t+1)B-1))\Big) \]
is disconnected, which violates Assumption 2.
\end{proof}

For our convergence time results, we will use the weaker Assumption
\ref{weakconnect}, rather than the stronger Assumption \ref{assumpt:boundedintervals}. Later on,
in Chapter \ref{nsquared}, we will exploit the sufficiency of
Assumption \ref{weakconnect} to {design} a 
decentralized algorithm for selecting the weights $a_{ij}(t)$, which
satisfies Assumption \ref{weakconnect}, but not  Assumption \ref{assumpt:boundedintervals}.

We now proceed to bound the decrease of our Lyapunov function
$V(x(t))$ during the interval $[tB, (t+1) B-1]$. In
what follows, we denote by $V(t)$ the sample variance $V(x(t))$ at
time $t$.

\begin{lemma}
\label{vardiff} Let Assumptions \ref{assumpt:weights} (non-vanishing
weights), \ref{assumpt:ds} (double stochasticity) and
\ref{weakconnect} (relaxed connectivity) hold. Let $\{x(t)\}$ be
generated by the update rule (\ref{eq:basicupdate}). Suppose that
the components $x_i(tB)$ of the vector $x(tB)$ have been ordered
from largest to smallest, with ties broken arbitrarily. Then, 
\[
V(tB) - V((t+1)B) \geq
\frac{\eta}{2} \sum_{i=1}^{n-1} (x_{i}(tB) - x_{i+1}(tB))^2.
\]
\end{lemma}

\begin{proof} By Lemma \ref{vl}, we have for all $t$,
\begin{equation}
V(t)-V(t+1) = \sum_{i<j} w_{ij}(t)
(x_i(t)-x_j(t))^2, \label{firstvdec0}
\end{equation}
where $w_{ij}(t)$ is the $(i,j)$-th entry of $A(t)^T A(t)$. Summing
up the variance differences $V(t)-V(t+1)$ over different values of $t$, we obtain
\begin{equation}
V(tB) - V((t+1)B)
= \sum_{k=tB}^{(t+1)B-1} \sum_{i<j}
w_{ij}(k) (x_i(k)-x_j(k))^2. \label{firstvdec}
\end{equation}

\ao{We next introduce some notation.}

\begin{itemize}
\item[(a)]
For all $d\in\{1,\ldots,n-1\}$, let $t_d$ be the first time larger than
or equal to $tB$
(if it exists) at which there is a communication between two nodes
belonging to the two sets $\{1,\ldots,d\}$ and $\{d+1,\ldots,n\}$, to
be referred to as a communication across the cut $d$.

\item[(b)]
For all $t\in \{tB,\ldots,(t+1)B-1\}$, let $D(t)=\{d\mid
t_d=t\}$, i.e., $D(t)$ consists of ``cuts" $d\in \{1,\ldots,n-1\}$
such that time $t$ is the first communication time larger than or equal
to $tB$ between nodes in the sets $\{1,\ldots,d\}$ and
$\{d+1,\ldots,n\}$. Because of Assumption \ref{weakconnect}, the
union of the sets $D(t)$ includes all indices $1,\ldots,n-1$, except
possibly for indices for which $x_d(tB)=x_{d+1}(tB)$.

\item[(c)]
For all $d\in \{1,\ldots,n-1\}$, let $C_d=\{(i,j) \alexo{,~(j,i)} 
 \mid i\leq d, \ d+1\leq j\}$.

\item[(d)]
For all $t\in \{tB,\ldots,(t+1)B-1\}$, let $F_{ij}(t) =\{d\in D(t)\
|\ (i,j) \alexo{\mbox{ or } (j,i)} \in C_d\}$, i.e., $F_{ij}(t)$
consists of all cuts $d$ such that the edge $(i,j)$ \alexo{ or
$(j,i)$} at time $t$ is the first communication across the cut at a
time larger than or equal to $tB$.

\item[(e)] To simplify notation, let $y_i=x_i(tB)$. By assumption, we have
$y_1\geq\cdots\geq y_n$.
\end{itemize}

We make two observations, as follows:
\begin{itemize}
\item[(1)]
Suppose that $d\in D(t)$. Then, for some $(i,j)\in C_d$, we have
either $a_{ij}(t)>0$ or $a_{ji}(t)>0$.
\alexo{Because $A(t)$ is
nonnegative with positive diagonal \an{entries, we have}
\[ w_{ij}(t) = \sum_{k=1}^n
a_{ki} a_{kj} \geq a_{ii}(t) a_{ij}(t) + a_{ji}(t) a_{jj}(t) > 0,\]
and} by Lemma \ref{db}, we obtain
\begin{equation}
\sum_{(i,j)\in C_d} w_{ij}(t)\geq \frac{\eta}{2}. \label{eq:w}
\end{equation}
\item[(2)]
Fix some $(i,j)$, with $i<j$, and time $t'\in
\{tB,\ldots,(t+1)B-1\}$, and suppose that $F_{ij}(t')$ is nonempty.
Let $F_{ij}\ao{(t')}=\{d_1,\ldots,d_k\}$, where the $d_j$ are
arranged in increasing order. Since $d_1\in F_{ij}\ao{(t')}$, we have
$d_1\in D(t)$ and therefore $t_{d_1}=t'$. By the definition of
$t_{d_1}$, this implies that there has been no communication between
a node in $\{1,\ldots,d_1\}$ and a node in $\{d_1+1,\ldots,n\}$
during the time interval $[tB,t'-1]$. It follows that $x_i(t') \geq
y_{d_1}$. By a symmetrical argument, we also have
\begin{equation}
x_j(t')\leq y_{d_k+1}.\label{cutbound}
\end{equation}
These relations imply that
$$x_i(t')-x_j(t') \geq y_{d_1}-y_{d_k+1}  \ao{\geq } \sum_{d\in F_{ij}\ao{(t')}} (y_d-y_{d+1}),$$
 Since the components of $y$ are sorted in nonincreasing order, we
have $y_d-y_{d+1}\geq 0$, for every $d\in F_{ij}\ao{(t')}$. For any
nonnegative numbers $z_i$, we have
$$(z_1+\cdots+z_k)^2\geq z_1^2+\cdots+z_k^2,$$
which implies that 
\begin{equation}
(x_i(t')-x_j(t'))^2 \geq \sum_{d\in F_{ij}\ao{(t')}} (y_d-y_{d+1})^2.
\label{eq:dx} \end{equation} \end{itemize} We now use these two observations to provide a lower bound on the
expression on the right-hand side of Eq.\ (\ref{firstvdec0}) at time
$t'$.
We use Eq.\ (\ref{eq:dx}) and then Eq.\ (\ref{eq:w}),
to obtain
\begin{eqnarray*}
\sum_{i<j} w_{ij}(t')(x_i(t')-x_j(t'))^2&
\geq& \sum_{i<j} w_{ij}(t') \sum_{d\in F_{ij}\ao{(t')}} (y_d-y_{d+1})^2\\
&=&\sum_{d\in D(t')} \sum_{(i,j)\in C_d} w_{ij}(t') (y_d-y_{d+1})^2\\
&\geq& \frac{\eta}{2} \sum_{d\in D(t')}(y_d-y_{d+1})^2.
\end{eqnarray*}
We now sum both sides of the above inequality for different values of $t$, and
use Eq.\ (\ref{firstvdec}), to obtain
\begin{eqnarray*}
V(tB)-V((t+1)B)&=& \sum_{k=tB}^{(t+1)B-1} \sum_{i<j} w_{ij}(k)(x_i(k)-x_j(k))^2\\
&\geq&
\frac{\eta}{2} \sum_{k=tB}^{(t+1)B-1}\sum_{d\in D(k)}(y_d-y_{d+1})^2\\
&=&\frac{\eta}{2}\sum_{d=1}^{n-1} (y_d-y_{d+1})^2,
\end{eqnarray*}
where the last inequality follows from the fact that the union of
the sets $D(k)$ is only missing those $d$ for which $y_d=y_{d+1}$.
\end{proof}

\vskip 1pc We next establish a bound on {the} variance decrease that
plays {a} key role in our convergence analysis.

\begin{lemma} \label{lboundvar} Let Assumptions
\ref{assumpt:weights} (non-vanishing weights),  \ref{assumpt:ds}
(double stochasticity), and \ref{weakconnect} (connectivity
relaxation) hold, and suppose that $V(tB)>0$. Then,
\[\frac{V(tB)-V((t+1)B)}{V(tB)} \geq \frac{\eta}{2 n^2}\qquad
\hbox{for all }t.\]
\end{lemma}

\begin{proof}
{Without loss of generality, we assume that the components of
$x(tB)$ have been sorted in nonincreasing order.} By Lemma
\ref{vardiff}, we have
\[ V(tB) -
V((t+1)B) \geq \frac{\eta}{2}\, \sum_{i=1}^{n-1} (x_{i}(tB) -
x_{i+1}(tB))^2.\]
This implies that
\[ \frac{V(tB)-V((t+1)B)}{V(tB)}\ge\frac{\eta}{2} \,
\frac{\sum_{i=1}^{n-1} (x_{i}(tB) - x_{i+1}(tB))^2}{\sum_{i=1}^n
(x_i(kB)-\bar{x}(kB))^2}. \] Observe that the right-hand side does
not change when we add a constant to every $x_i(tB)$. We can
therefore assume, without loss of generality, that $\bar{x}(tB)=0$,
{so that}
\[ \frac{V(tB)-V((t+1)B)}{V(tB)} \geq \frac{\eta}{2}\,
\min_{{x_1\ge x_2\ge \cdots\ge x_n \atop \sum_i x_i=0}}
\frac{\sum_{i=1}^{n-1} (x_{i} - x_{i+1})^2}{\sum_{i=1}^n x_i^2}. \]
Note that the right-hand side is unchanged if we multiply each
$x_i$ by the same constant. Therefore, we can assume, without loss
of generality, that $\sum_{i=1}^n x_i^2=1$, {so that}
\begin{equation}\label{eq:vv}
\frac{V(tB)-V((t+1)B)}{V(tB)} \geq \frac{\eta}{2}\, \min_{{x_1\ge
x_2\ge \cdots\ge x_n  \atop \sum_i x_i=0, \ \sum_i x_i^2 = 1}}
~~\sum_{i=1}^{n-1} (x_{i} - x_{i+1})^2.\end{equation} The
requirement $\sum_i x_i^2 = 1$ implies that the average value of
$x_i^2$ is $1/n$, which implies that there exists some $j$ such
that $|x_j| \ge 1/\sqrt{n}$. Without loss of generality, let us
suppose that this $x_j$ is positive.\footnote{Otherwise, we can
replace $x$ with $-x$ and subsequently reorder to maintain the
property that the components of $x$ are in descending order. It can
be seen that these operations do not affect the objective value.}

The rest of the proof relies on a technique from \cite{LO81} to
provide a lower bound on the right-hand side of Eq.\ (\ref{eq:vv}).
Let
\[z_i = x_{i} - x_{i+1} \ \ \hbox{for } i<n,\quad
\hbox{and}\quad z_n=0.\] Note that $z_i \geq 0$ for all $i$ \ao{and
\[ \sum_{i=1}^n z_i = x_1 - x_n.\]}Since $x_j\ge 1/\sqrt{n}$ for
some $j$, \ao{we have that $x_1 \geq 1/\sqrt{n}$; since
$\sum_{i=1}^n x_i = 0$, it follows that at least one $x_i$ is
negative, and therefore $x_n < 0$. This gives us \[ \sum_{i=1}^n z_i
\geq \frac{1}{\sqrt{n}} .\]} Combining with {Eq.\ (\ref{eq:vv})}, we
obtain
\[ \frac{V(tB)-V((t+1)B)}{V(tB)} \geq
\frac{\eta}{2} \min_{z_i \geq 0,\ \sum_i z_i \geq 1/\sqrt{n}}
\sum_{i=1}^{n} z_i^2. \] {The minimization problem on the right-hand
side is a symmetric convex optimization problem, and therefore has a
symmetric optimal solution, namely} $z_i = 1/n^{1.5}$ for all $i$.
This results in an optimal value of $1/n^2$. Therefore,
\[\frac{V(tB)-V((t+1)B)}{V(kB)} \geq \frac{\eta}{2 n^2}, \]
which {is} the desired result. \end{proof}

\vskip 1pc

We are now ready for our main result, which establishes that the
convergence time of the sequence of vectors $x(k)$ generated by Eq.\
(\ref{eq:basicupdate}) is of order $O(n^2B/\eta)$.

\noindent \begin{theorem} \label{uqbound} Let Assumptions
\ref{assumpt:weights} (non-vanishing weights), \ref{assumpt:ds}
(double stochasticity), and \ref{weakconnect} (connectivity
relaxation) hold. Then there exists an absolute constant $c$ such
that we have
\[V({t})
\leq \epsilon V(0)\qquad \hbox{for all } t\ge {c} (n^2/\eta)B \log
(1/\epsilon).\]
\end{theorem}

\begin{proof}
The result follows immediately from Lemma \ref{lboundvar}.
\end{proof}

\vskip 1pc Recall that, according to  Lemma \ref{as2implas3},
Assumption \ref{assumpt:boundedintervals} implies Assumption
\ref{weakconnect}. In view of this, the convergence time bound of
Theorem \ref{uqbound} holds for any $n$ and any sequence of weights
satisfying Assumptions \ref{assumpt:weights} (non-vanishing
weights), \ref{assumpt:ds} (double stochasticity), and
\ref{assumpt:boundedintervals} ($B$-connectivity). \aoc{This proves
Theorem \ref{uqboundver1} from the beginning of this chapter.}

\section{Concluding remarks}

\aoc{In this chapter, we have presented a polynomial convergence-time bound on the performance of a class of averaging 
algorithms. Several open research directions naturally present themselves.}

\aoc{First, is it possible to design faster algorithms which nevertheless compute averages correctly on arbitrary (time-varying, 
undirected) graph sequences? 
We make some headway on this question in Chapter \ref{nsquared}, where we design an algorithm whose convergence time
scales as $O(n^2)$ with $n$, and in Chapter \ref{ch:optimality}  we prove a lower bound of at least $\Omega(n^2)$
 for a somewhat restricted class of algorithms. However, the general question of how fast averaging algorithms
can scale with $n$ is still open.}

\aoc{Secondly, where is the dividing line before polynomial and exponential convergence time? In particular, how far
may we relax the double stochasticity Assumption \ref{assumpt:ds} while still having polynomial convergence
time? For example, does polynomial time convergence still hold if we replace Assumption \ref{assumpt:ds} with the 
requirement that the matrices $A(t)$ be (row) stochastic and each column sum is in $[1-\epsilon, 1+\epsilon]$ for
some small $\epsilon>0$?}

\chapter{Averaging in quadratic time \label{nsquared} \label{matrixpicking}}

In the previous section, we have shown that a large class of
averaging algorithms have  $O(B (n^2/\eta) ~\alexo{ \log
1/\epsilon}) $ \alexo{convergence time. 
}

In this section, we consider decentralized ways of synthesizing the
weights $a_{ij}(t)$ while satisfying Assumptions \ref{assumpt:weights}, \ref{assumpt:ds}, and
\ref{weakconnect}. We assume that the sequence of (undirected) graphs $G(t)=(N, E(t))$ is given exogenously,
but the nodes can pick the coefficients $a_{ij}(t)$. \alexo{Our focus is on improving convergence time
bounds by \an{constructing} ``good'' schemes.}


Naturally, several ways to pick the coefficients present
themselves. For example,  each node
may assign
\begin{eqnarray*} a_{ij}{(t)} & = & \epsilon, \qquad\qquad \mbox{ \ \ if
} (j,i)
\in E(t) {\mbox{ and } i\neq j,} \\
a_{ii}{(t)} & = & 1 - \epsilon\cdot \mbox{deg}(i),
\end{eqnarray*}
where {\rm deg}($i$) is the degree of $i$ in $G(t)$. If $\epsilon$
is small enough and the graph $G(t)$ is undirected {[i.e.,
$(i,j)\in E(t)$ if and only if $(j,i)\in E(t)$],} this results in
a nonnegative, doubly stochastic matrix (see \cite{olfati}). \aoc{The Metropolis algorithm from
Chapter \ref{why} is a special case of this method.} However,
{if a node has $\ao{\Theta}(n)$ neighbors, $\eta$ will be of order
$\ao{\Theta}(1/n)$, resulting in $\ao{\Theta}(n^3)$ convergence
time.} Moreover, this
argument applies to all protocols in which nodes assign equal
weights to all their neighbors; see \cite{XB04} and \cite{BFT05} for
more examples.

In this section, we examine whether it is possible to synthesize the
weights $a_{ij}(t)$ in a decentralized manner, \aoc{so that the above convergence
time is reduced. Our main result shaves a factor of $n$ off the convergence
time of the previous paragraph.}

\aoc{\noindent \begin{theorem} \label{thm:nsquared} Suppose $G(t)=(N,E(t))$ is a sequence of
undirected graphs such that $(N,E(tB)\cup E(tB+1)\cup \cdots \cup
E((t+1)B-1))$ is connected, for all integers $t$. Then, there exists a decentralized
way to pick the coefficients $a_{ij}(t)$ such that
\[V({t}) \leq \epsilon V(0)\qquad \hbox{for all }t\ge {c}
n^2B \log (1/\epsilon).\]
\end{theorem}}

\aoc{This theorem has appeared in the paper \cite{NOOT07}, which we
will follow here.}

\aoc{Our approach we will be to pick $a_{ij}(t)$} so that
$a_{ij}(t)\geq \eta$ whenever $a_{ij}(t)\neq 0$, where $\eta$ is a
positive constant independent of $n$ and $B$. We show that this is
indeed possible, under the additional assumption that the graphs
$G(t)$ are undirected (\aoc{in Chapter \ref{basic}, we referred to
this case as the ``symmetric model.''}). Our algorithm is
data-dependent, in that $a_{ij}(t)$ depends not only on the graph
$G(t)$, but also on the data vector $x(t)$. Furthermore, it is a
decentralized 3-hop algorithm, in that $a_{ij}{(t)}$ depends only on
the data at nodes within a distance of at most $3$ from $i$. Our
algorithm is such that the resulting sequences of vectors $x(t)$ and
graphs $G(t)=(N,\E(t))$, with $\E(t)=\{(j,i)\mid a_{ij}(t)> 0\}$,
satisfy Assumptions \ref{assumpt:weights}, \ref{assumpt:ds} and
\ref{weakconnect}. Thus, a convergence time result can be obtained
from Theorem \ref{uqbound}.

\section{The algorithm}

The algorithm we present here is a variation of an old {\em load
balancing} algorithm (see \cite{C89} and Chapter 7.3 of
\cite{BT89}).\footnote{This algorithm was also considered in
\cite{OT09}, but in the absence of a result such as Theorem
\ref{uqbound}, a weaker convergence time bound was derived.}

At each step of the algorithm, each node offers some of its value to
its neighbors, and accepts or rejects such offers from its
neighbors. Once an offer from $i$ to $j$, of size $\delta>0$, has
been accepted, the updates $x_i \leftarrow x_i - \delta$ and $x_j
\leftarrow x_j + \delta$ are executed.

We next describe the formal steps the nodes execute at each time
${t}$. For clarity, we refer to the node executing the steps below
as node $\ao{C}$. Moreover, the instructions below sometimes refer
to the neighbors of node $\ao{C}$; this always means current
neighbors at time $t$, when the step is being executed, {as
determined by the current graph $G(t)$.}
We assume that at each time $t$, all nodes execute these steps in
the order described below, while the graph remains unchanged.

\vspace{1pc}
\noindent {\bf Balancing Algorithm:}
\begin{enumerate}\itemsep=0pt
\item[1.]  Node $\ao{C}$ broadcasts its current value $x_{\ao{C}}$ to all its
neighbors.

\item[2.]  Going through the values it just received from its
neighbors, Node $\ao{C}$ finds the smallest value that is less than
{its own}. Let $\ao{D}$ be a neighbor with this value. Node $\ao{C}$
makes an offer of $(x_{\ao{C}} - x_{\ao{D}})/3$ to node $D$.

If no {neighbor of $\ao{C}$} has a value smaller than $x_{\ao{C}}$,
node $C$ does nothing at this stage.

\item[3.] Node $\ao{C}$ goes through the incoming offers. It sends an
acceptance to the sender of {a} largest offer, and a rejection to
all the other senders. It updates the value of $x_{\ao{C}}$ by
adding the value of the accepted offer.

If node $C$ did not receive any offers, it does nothing at this
stage.

\item[4.] If an acceptance arrives {for} the offer made by node
$\ao{C}$, node $\ao{C}$ updates $x_{\ao{C}}$ by subtracting the
value of the offer.
\end{enumerate}

Note that the new value of each node is a linear combination of the
values of its neighbors. Furthermore, the weights $a_{ij}(t)$ are
completely determined by the data and the graph at most $3$ hops
from node $i$ in $G(t)$. \alexo{This is true because in the course
of execution of the above steps, each node makes at most three
transmissions to its neighbors, so the new value of node $C$ cannot
depend on information more than $3$ hops away from $C$.}


\section{Performance analysis}

\aoc{The following theorem (stated at the beginning of the chapter as Theorem \ref{thm:nsquared}) allows us
to remove a factor of $n$ \jnt{from} 
the worst-case convergence time bounds of Theorem
\ref{uqbound}}.

\noindent \begin{theorem} \label{savingn}
Consider the balancing algorithm,
and suppose that $G(t)=(N,E(t))$ is a sequence of
undirected graphs such that $(N,E(tB)\cup E(tB+1)\cup \cdots \cup
E((t+1)B-1))$ is connected, for all integers $t$. There exists an
absolute constant $c$ such that we have
\[V({t}) \leq \epsilon V(0)\qquad \hbox{for all }t\ge {c}
n^2B \log (1/\epsilon).\]
\end{theorem}

\begin{proof}
Note that with this algorithm, the new value at some node $i$ is a
convex combination of the previous values of itself and its
neighbors. Furthermore, the algorithm keeps the sum of the nodes'
values constant, because every accepted offer involves an increase
at the receiving node equal to the decrease at the offering node.
These two properties imply that the algorithm can be written in the
form
\[ x(t+1) = A(t) x(t), \] where $A(t)$ is a doubly stochastic matrix,
determined by $G(t)$ and $x(t)$. It can be seen that the diagonal
entries of $A(t)$ are positive and, furthermore, all nonzero entries
of $A(t)$ are larger than or equal to 1/3; thus, $\eta=1/3$.

We claim that the algorithm [in particular, the sequence $\E(A(t))$]
satisfies Assumption \ref{weakconnect}. Indeed, suppose that at time
$tB$, the nodes are reordered so that the values $x_i(tB)$ are
nonincreasing in $i$. Fix some $d\in\{1,\ldots,n-1\}$, and suppose
that $x_d(tB)\neq x_{d+1}(tB)$. Let $S^+=\{1,\ldots,d\}$ and
$S^-=\{d+1,\ldots,n\}$.

Because of our assumptions on the graphs $G(t)$, there will be a
first time $t'$ in the interval $\{tB,\ldots,(t+1)B-1\}$, at which
there is an edge in $E(t)$ between some $i^*\in S^+$  and $j^*\in
S^-$. Note that between times $tB$ and $t'$, the two sets of nodes,
$S^+$ and $S^-$, do not interact, which implies that $x_i(t')\geq
x_d(tB)$, for $i\in S^+$, and $x_j(t') < x_d(tB)$, for $j\in S^-$.

At time $t$, node $i^*$ sends an offer to a neighbor with the
smallest value; let us denote that neighbor by $k^*$. Since
$(i^*,j^*)\in E(t')$, we have $x_{k^*}(t')\leq x_{j^*}(t')<x_d(tB)$,
which implies that $k^*\in S^-$. Node $k^*$ will accept the largest
offer it receives, which must come from a node with a value no
smaller than $x_{i^*}(t')$, \ao{and therefore no smaller than
$x_d(tB)$}; hence the latter node belongs to $S^+$. It follows that
$\E(A(t'))$ contains an edge between $k^*$ and some node in $S^{+}$,
showing that Assumption \ref{weakconnect} is satisfied.

The claimed result follows from Theorem \ref{uqbound}, because we
have shown that all of the assumptions in that theorem are satisfied
\alexo{with $\eta=1/3$}.
\end{proof}

\section{Concluding remarks}

\aoc{In this chapter, we have presented a specific averaging algorithm whose
scaling with the number of nodes $n$ is $O(n^2)$. An interesting direction to explore is to what extent, 
if any, this can be reduced. The next chapter provides a partial negative answer
to this question, showing that one cannot improve on quadratic scaling with a limited
class of algorithms.}
\chapter{On the optimality of quadratic time \label{ch:optimality}}

The goal of this chapter is to analyze the fundamental limitations of
the kind of distributed averaging algorithms we have been studying. \aoc{The previous
chapter described a class of averaging algorithms whose convergence time scales with
the number of agents $n$ as $O(n^2)$. Our aim in this chapter is to show that this
is the best scaling for} a common class of such algorithms; namely, that any distributed averaging algorithm
 that uses a single scalar state variable at each agent and satisfies a natural ``smoothness'' condition will have
 this property. Our exposition will follow the preprint \cite{OT10}
 where these results have previously appeared. 

 We next proceed to define \aoc{the class of distributed averaging algorithms} we are considering and
informally state our result. \section{\ao{Background and basic
definitions.}} \noindent {\bf Definition of local averaging
algorithms:} Agents $1,\ldots,n$
 begin with real numbers $x_1(0),\ldots,x_n(0)$ stored in memory. At each round $t=0,1,2,\ldots$, agent $i$ broadcasts $x_i(t)$ to each of
its neighbors in some {\em undirected} graph $G(t)=(\{1,\ldots,n\}, E(t))$, and then sets $x_i(t+1)$ to be some
  function of $x_i(t)$ and of the values $x_{i'}(t), x_{i''}(t), \ldots $ it has just received from its own neighbors:
 \begin{equation} \label{onehopupdate} x_i(t+1) = f_{i, G(t)} ( x_i(t), x_{i'}(t), x_{i''}(t), \ldots ). \end{equation} We require each $f_{i,G(t)}$ to be a differentiable function. Each
 agent uses the incoming messages $x_{i'}(t), x_{i''}(t),\ldots$ as the arguments of $f_{i,G(t)}$ in some arbitrary order; we
 assume that this order does not change, i.e. if $G(t_1)=G(t_2)$, then the message coming from the same neighbor of agent $i$
  is mapped to the same argument of $f_{i,G(t)}$ for $t=t_1$ and $t=t_2$.  It is desired that
 \begin{equation} \label{eq:convergence} \lim_{t \rightarrow \infty} x_i(t) = \frac{1}{n} \sum_{i=1}^n x_i(0),\end{equation} for every $i$,  for every sequence of graphs $G(t)$ having the property that
 \begin{equation} \label{connectivity}  \mbox{ the graph } (\{1,\ldots,n\}, \cup_{s \geq t} E(s)) \mbox{  is connected for every } t, \end{equation} and for every possible way for the agents to map incoming messages to arguments of $f_{i,G(t)}$.

  In words, as the number of rounds $t$ approaches infinity, iteration (\ref{onehopupdate}) must converge to the average of the numbers $x_1(0), \ldots,x_n(0)$. Note that the agents have no control over the communication graph sequence $G(t)$, which is exogenously provided by ``nature." However, as we stated previously, every element of the sequence $G(t)$ must be undirected: this corresponds to symmetric models of communication between agents. Moreover, the sequence $G(t)$ must satisfy the mild connectivity condition of Eq. (\ref{connectivity}), which says that the network cannot become
 disconnected after a finite period.

 Local averaging algorithms are useful tools for information fusion due to their efficient utilization of
resources (each agent stores only a single number in memory) as well
as their robustness properties (the sequence of graphs $G(t)$ is
time-varying, and it only needs to satisfy the relatively weak
connectivity condition in Eq. (\ref{connectivity}) for the
convergence in Eq. (\ref{eq:convergence}) to hold). As explained in Chapter \ref{why}, 
no other class of schemes for averaging (e.g.,
flooding, fusion along a spanning tree) is known to produce similar
results under the same assumptions.

\bigskip

\noindent {\bf Remark:}  As can be seen from the
subscripts, the update function $f_{i,G(t)}$ is allowed to depend on
the agent and on the graph. Some dependence on the graph is
unavoidable since in different graphs an agent may have a different
number of neighbors, in which case nodes will receive a different
number of messages, so that even the number of arguments of
$f_{i,G(t)}$ will depend on $G(t)$. It is often practically desired
that $f_{i,G(t)}$ depend only weakly on the graph, as the entire
graph may be unknown to agent $i$. For example, we might require
that $f_{i,G(t)}$ be completely determined by the degree of $i$ in
$G(t)$. However, since our focus is on what distributed algorithms
{\em cannot} do, it does not hurt to assume the agents have
unrealistically rich information; thus we will not assume any
restrictions on how $f_{i,G(t)}$ depends on $G(t)$.

\bigskip

\noindent {\bf Remark:}  We require the functions
$f_{i,G(t)}$ to be smooth, for the following reason. First, we
need to exclude unnatural algorithms that encode vector information in the
infinitely many bits of a single real number. Second, although we
make the convenient technical assumption that agents can transmit
and store real numbers, we must be aware that in practice agents
will transmit and store a quantized version of $x_i(t)$. Thus, we are mostly interested
in algorithms that are not disrupted much by quantization. For this reason,
 we must prohibit the agents from using {\em discontinuous}
update functions $f_{i, G(t)}$. For technical reasons, we actually
go a little further, and prohibit the agents from using {\em
non-smooth} update functions $f_{i,G(t)}$.

 \bigskip

%

\subsection{ \ao{Examples.}}
In order to provide some context, let us mention just a few of the
distributed averaging schemes that have been proposed in the
literature:

\begin{enumerate} \item The max-degree method \cite{olfati} involves picking $\epsilon(t)$ with the property $\epsilon(t) \leq 1/(d(t)+1)$,  where $d(t)$ is the largest degree of any agent in $G(t)$, and updating by
\[ x_i(t+1) = x_i(t) + \epsilon(t) \sum_{i \in N_i(t)} \left( x_j(t) - x_i(t) \right) .\] Here we use $N_i(t)$ to denote the set of neighbors of agent $i$ in $G(t)$. In practice, a satisfactory $\epsilon(t)$ may not be known to all of
the agents, because this requires some global information. However,
in some cases a satisfactory choice for $\epsilon(t)$ may be
available, for example when an a priori upper bound on $d(G(t))$ is
known.
\item The Metropolis method (\aoc{see Chapter \ref{why}}) from \cite{XB04} involves setting $\epsilon_{ij}(t)$ to satisfy $\epsilon_{ij}(t) \leq \min(1/(d_i(t)), 1/(d_j(t)))$, where $d_i(t), d_j(t)$ are the degrees of agents $i$ and $j$ in $G(t)$, and
    updating by
    \[ x_i(t+1) = x_i(t) + \sum_{j \in N_i(t)} \epsilon_{ij}(t) \left( x_j(t) - x_i(t) \right). \]
    \item The load-balancing algorithm of Chapter \ref{nsquared} involves
    updating by
    \[ x_i(t+1) = x_i(t) + \sum_{i \in N_i(t)} a_{ij}(t) \left( x_j(t) - x_i(t) \right),\] where $a_{ij}(t)$ is determined by the following rule: each agent selects exactly two neighbors, the neighbor with the largest value above its own and with the smallest value below its own. If $i,j$ have both selected each other, then
    $a_{ij}(t)=1/3$; else $a_{ij}(t)=0$. The intuition comes from load-balancing: agents think of $x_i(t)$
     as load to be equalized among their neighbors; they try to offload on their lightest neighbor and take from their heaviest neighbor.

\end{enumerate}

We remark that the above load-balancing algorithm is not a ``local
averaging algorithm'' according to our definition because $x_i(t+1)$
does not depend only on $x_i(t)$ and its neighbors; for example,
agents $i$ and $j$ may not match up because  $j$ has a neighbor $k$
with $x_k(t) > x_j(t)$. By contrast, the max-degree and Metropolis
algorithm are indeed ``local averaging algorithms.''


\subsection{\aoc{Our results}}

Our goal is to study the worst-case convergence time over all graph
sequences. \ao{This convergence time may be arbitrarily bad since
one can insert arbitrarily many empty graphs into the sequence
$G(t)$ without violating Eq.\ (\ref{connectivity}). To avoid this
trivial situation, we require that there exist some integer $B$
such that the graphs}
\begin{equation} \label{strongerconnect} (\{1,\ldots,n\}, \cup_{i=kB}^{(k+1)B} E(k))
 \end{equation} are connected for every integer $k$.

Let $x(t)$ be the vector in $\Re^n$ whose $i$th component is $x_i(t)$.
We define the convergence time $T(n,\epsilon)$ of a local averaging algorithm as the \aa{time until
``sample variance'' } \[ V(x(t)) = \sum_{i=1}^n \left( x_i(t)-\frac{1}{n}
\sum_{j=1}^n x_j(0) \right)^2 \] \aa{permanently shrinks by a factor of $\epsilon$, i.e., $V(x(t)) \leq \epsilon V(x(0))$ for all $t \geq T(n, \epsilon)$, for all possible  $n$-node graph sequences satisfying
Eq.\ (\ref{strongerconnect}), and all initial vectors $x(0)$ for
which not all $x_i(0)$ are equal; $T(n,\epsilon)$ is defined to be the smallest number
with this property.} We are interested in how
$T(n,\epsilon)$ scales with $n$ and $\epsilon$.

Currently, the best available upper bound for the convergence time is
obtained with the load-balancing algorithm of \aoc{of Chapter \ref{nsquared}, where it was proven}
\[ T(n,\epsilon) \leq C n^2 B \log \frac{1}{\epsilon},
\] for some absolute constant\footnote{By ``absolute constant'' we
mean that $C$ does not depend on the problem parameters $n, B,
\epsilon$.} $C$.We are primarily interested in whether its possible to improve the
scaling with $n$ to below $n^2$.  Are there  nonlinear update
functions $f_{i,G(t)}$ which speed up the convergence time?

Our main result is that the answer to this question is ``no" within
the class of local averaging algorithms. For such algorithms we
prove a general lower bound of the form \[ T(n,\epsilon) \geq c n^2
B \log \frac{1}{\epsilon}, \] for some absolute constant $c$.
Moreover, this lower bound holds even if we assume that the graph
sequence $G(t)$ is the same for all $t$; in fact, we prove it for
the case where $G(t)$ is a fixed ``line graph.''

\section{Formal statement and proof of main result}

We next state our main theorem. The theorem
begins by  specializing our definition of local averaging
algorithm to the case of a fixed line graph, and states a lower
bound on the convergence time in this setting.

We will use the notation $\1$ to denote the vector in $\R^n$ whose
entries are all ones, and $\0$ to denote the vector whose
entries are all $0$. The average of the initial values
$x_1(0),\ldots,x_n(0)$ will be denoted by $\bar{x}$.

\noindent \begin{theorem} \label{mainthm} Let $f_1$, $f_n$ be two differentiable functions from $\R^2$ to $\R$, and
let $f_2,f_3,\ldots,f_{n-1}$ be differentiable functions from $\R^3$ to $\R$. Consider
the dynamical system \begin{eqnarray} x_1(t+1) & = &
f_1(x_1(t),x_2(t)), \nonumber \\
x_i(t+1) & = & f_i(x_i(t),
x_{i-1}(t),x_{i+1}(t)), ~~ i = 2,\ldots,n-1, \nonumber \\
x_n(t+1) & = & f_n(x_{n-1}(t),x_n(t)). \label{dynsys} \end{eqnarray} Suppose that there exists a function $\tau(n,\epsilon)$ such that
 \[ \frac{ \|x(t)- \bar{x} \1 \|_2}{\|x(0)- \bar{x} \1\|_2}  < \epsilon, \] for all $n$ and $\epsilon$, all $t \geq \tau(n,\epsilon)$, and all initial conditions $x_1(0),\ldots,x_n(0)$ for
which not all $x_i(0)$ are equal. Then,
\begin{equation} \label{mainresult} \tau(n,\epsilon) \geq \frac{n^2}{30} \log \frac{1}{\epsilon},\end{equation} for all $\epsilon > 0$ and $n \geq 3$.
\end{theorem}

\bigskip

\noindent {\bf Remark:} The dynamical system described in the
theorem statement is simply what a local averaging algorithm looks
like on a line graph. The functions $f_1, f_n$ are the update
functions at the left and right endpoints of the line (which have
only a single neighbor),  while the update functions $f_2, f_3,
\ldots, f_{n-1}$ are the ones used by the middle agents (which have
two neighbors). As a corollary, the convergence time of any local averaging algorithm must satisfy  the lower bound
$T(n,\epsilon) \geq (1/30) n^2 \log (1/\epsilon)$.

\bigskip

\noindent {\bf Remark:} Fix some $n \geq 3$. A corollary of our theorem is that there are no
``local averaging algorithms'' which compute the average in finite time.  More precisely,
 there is no local averaging algorithm which, starting from initial conditions $x(0)$ in some
  ball around the origin, always results in $x(t) = {\bar x} \bf 1$ for all times $t$ larger than
 some $T$.  We will sketch a proof of this after proving Theorem 1. By contrast, the existence of such algorithms in slightly
different models of agent interactions was demonstrated in
\cite{C06} and \cite{SH07}.

\subsection{\aoc{Proof}.}

We first briefly sketch the proof strategy. We will begin by pointing out that $\0$ must be an equilibrium of
 Eq. (\ref{dynsys}); then, we will argue that an upper bound  on the convergence time of Eq. (\ref{dynsys}) would imply a similar convergence time bound on the linearization of Eq. (\ref{dynsys}) around the equilibrium of $\0$.
This will allow us to apply a previous $\Omega(n^2)$ convergence time lower
bound for {\em linear} schemes, proved by the authors in \cite{OT09}.

Let $f$ (without a subscript) be the mapping from $\R^n$ to itself that
maps $x(t)$ to $x(t+1)$ according to Eq.\ (\ref{dynsys}).

\begin{lemma} $f(a\1)=a\1$, for any $a \in \R$. \label{fixedpoint} \end{lemma}
\begin{proof} Suppose that $x(0)=a \1$. Then, the initial average is $a$, so that
 \[ a \1 = \lim_t x(t) = \lim_t x(t+1) = \lim_t f(x(t)).\] We use the continuity of $f$ to get
 \[ a \1 = f(\lim_t x(t)) = f(a \1). \]
\end{proof}

%

For $i,j=1,\ldots,n$, we define  $a_{ij} = \frac{\partial
f_i(0)}{\partial x_j}$,
and the matrix
\[ A=  f'({\bf 0}) = \left(
                         \begin{array}{cccccc}
                           a_{11} & a_{12} & 0 & 0 & \cdots & 0 \\
                           a_{21} & a_{22} & a_{23} & 0 & \cdots & 0 \\
                           0 & a_{32} & a_{33} & a_{34} & \cdots & 0 \\
                           \vdots & \vdots & \vdots & \vdots & \vdots & \vdots \\
                           0 & \cdots & 0 & 0 & a_{n,n-1} & a_{nn} \\
                         \end{array}
                       \right). \]

\begin{lemma} For any integer $k \geq 1$, \[ \lim_{x \rightarrow \0} \frac{\|f^{k}(x) - A^k x\|_2}{\|x\|_2} = 0,  \]
\label{compositionapprox} \aa{where $f^{k}$ refers to the $k$-fold composition of $f$ with itself.} \end{lemma} \begin{proof} The fact that
$f(\0)=\0$ implies by the chain rule that the derivative of $f^k$ at
$x=\0$ is $A^k$. The above equation is a restatement of this fact.
\end{proof}
\begin{lemma} Suppose that $x^T \1=0$. Then, \[ \lim_{m \rightarrow \infty} A^m x = {\bf 0}.\] \label{conv}
\end{lemma} \begin{proof} \ao{Let ${\cal B}$ be a ball around the origin such that for all $x \in {\cal B}$, with $x\neq {\bf 0}$, we have}
\[ \frac{\|f^k ( x )- A^k x \|_2}{\|x\|_2} \leq \frac{1}{4}, ~~\mbox{ for  } k=\tau(n,1/2). \] \ao{Such a ball can be found due to  Lemma \ref{compositionapprox}. Since we can scale $x$ without affecting the assumptions or conclusions of the lemma we are trying to prove, we can assume that $x \in {\cal B}$.}
It follows that that for $k=\tau(n,1/2)$, we have
\begin{eqnarray*} \frac{\|A^{k} x\|_2}{\|x\|_2} & = &   \frac{\|A^{k} x - f^{k}(x) + f^{k} (x)\|_2}{\|x\|_2} \\
 & \leq &  \frac{1}{4} + \frac{\|f^{k} (x)\|_2}{\|x\|_2} \\
  & \leq &  \frac{1}{4} + \frac{1}{2} \\
  &  \leq & \frac{3}{4}.
   \end{eqnarray*}  \ao{Since this inequality implies that $A^{k} x \in {\cal B}$, we can apply} the same argument recursively to get
\[ \lim_{m \rightarrow \infty} (A^{k})^m x = \0,\]  which implies the conclusion of the lemma.
\end{proof}
\bigskip
\begin{lemma} A\1 = \1. \label{righteigen} \end{lemma}
\begin{proof} We have \[ A \1 = \lim_{h \rightarrow 0} \frac{f(\0 + h  \1)-f(\0)}{h} = \lim_{h \rightarrow 0} \frac{ h \1}{h} = \1,\] where we used Lemma \ref{fixedpoint}. \end{proof}

\begin{lemma} For every vector $x \in \R^n$, \[ \lim_{k \rightarrow \infty} A^k x = {\bar x} \1 ,\]
where ${\bar x}=(\sum_{i=1}^n x_i)/n$.
\label{convlemma}
\end{lemma}
\begin{proof} Every vector $x$ can be written as
\[ x = 
{\bar x} \1  + y,\] where $y^T \1 = 0$. Thus,
\[ \lim_{k \rightarrow \infty} A^k x = \lim_{k \rightarrow \infty} A^k \left( {\bar x} \1  + y \right) =
{\bar x} \1 + \lim_{k \rightarrow \infty} A^k y = {\bar x} \1, \] where we used
Lemmas \ref{conv} and \ref{righteigen}.
\end{proof}

\bigskip

\begin{lemma} \label{aproperties} The matrix $A$ has the following properties:
\begin{enumerate} \item $a_{ij}=0$ whenever $|i-j| >1$.
\item The graph $G = (\{1,\ldots,n\},E)$, with $E = \{ (i,j) ~|~ a_{ij} \neq 0 \}$, is strongly connected.
\item $A \1 = \1$ and $\1^T A = \1$.
\item An eigenvalue of $A$ of largest modulus has modulus $1$.
\end{enumerate}
\end{lemma}

\begin{proof} \begin{enumerate} \item True because of the definitions of $f$ and $A$.
\item Suppose not. Then, there is a nonempty set $S \subset \{1,\ldots,n\}$ with the property that
$a_{ij}=0$ whenever $i\in S$ and $j\in S^c$. Consider the vector $x$ with
$x_i=0$ for $i \in S$, and $x_j=1$ for $j \in S^c$. Clearly, $(1/n) \sum_i x_i > 0$, but
$(A^k x)_i=0$ for $i \in S$. This contradicts Lemma \ref{convlemma}.
\item The first equality was already proven in Lemma \ref{righteigen}. For the second, let $b=\1^T A$. Consider the
vector \begin{equation} z =  \lim_{k \rightarrow \infty} A^k e_i, \end{equation} where $e_i$ is the $i$th unit
 vector. By Lemma \ref{convlemma}, \[
z = \frac{\1^T e_i}{n} \1 = \frac{1}{n} \1.\] On the
other hand, \[ \lim_{k \rightarrow \infty} A^k e_i = \lim_{k \rightarrow \infty} A^{k+1} e_i = \lim_{k \rightarrow \infty} A^k (A e_i).\] Applying Lemma \ref{convlemma} again, we get \[
z = \frac{\1^T \cdot (A e_i)}{n} \1 = \frac{b_i}{n} \1, \] where $b_i$ is the $i$th component of $b$. We conclude
that $b_i=1$; since no assumption was made on $i$, this implies that $b=\1$, which is what we needed to show.
\item We already know that $ A \1 = \1$, so that an eigenvalue with modulus $1$ exists.
Now suppose there is an eigenvalue with larger modulus, that is, there is some vector $x \in \mathbb{C}^n$ such that
 $Ax = \lambda x$ and $|\lambda| > 1$. Then $\lim_k \|A^k x\|_2 = \infty$. By writing $x = x_{\rm real} + i x_{\rm imaginary}$, we immediately have that $A^k x = A^k x_{\rm real} + i A^k x_{\rm imaginary}$. But by Lemma \ref{convlemma} both $A^k x_{\rm real}$ and $A^k x_{\rm imaginary}$ approach some finite multiple of $\1$ as $k \rightarrow \infty$, so $\|A^k x\|_2$ is bounded above. This is a contradiction.
\end{enumerate}
\end{proof} \aoc{\begin{theorem}[Eigenvalue lemma] \label{thm:lb} If $A$ satisfies all of the conclusions of Lemma \ref{aproperties}, then $A$ has an eigenvector $v$, with real eigenvalue $\lambda \in (1-\frac{6}{n^2},1)$, such that $v^T {\bf 1}=0$.
\end{theorem}} \aoc{This proof of this fact needs its own section.}

\section{\aoc{Proof of the eigenvalue lemma}}

\begin{lemma}\label{th:2} Consider an $n\times n$ matrix $A$ and let
$\lambda_1,\lambda_2,\ldots,\lambda_n$, be its eigenvalues, sorted in order of decreasing maginitude.
Suppose that the following conditions hold.
\begin{itemize}
\item[(a)] We have $\lambda_1=1$ and $A {\bf 1} ={\bf 1}$.
\item[(b)] There exists a positive vector $\pi$ such that $\pi^T A=\pi^T$.
\item[(c)] For every $i$ and $j$, we have $\pi_i a_{ij}=\pi_j a_{ji}$.
\end{itemize}
Let
$$S=\Big\{x\ \Big|\  \sum_{i=1}^n \pi_i x_i =0,\ \sum_{i=1}^n \pi_i x_i^2=1\Big\}$$
Then, all eigenvalues of $A$ are real, and
\begin{equation}\label{eq:gg}
\lambda_2=1-\frac{1}{2} \min_{x\in S}\sum_{i=1}^n \sum_{j=1}^n \pi_i a_{ij}
(x_i-x_j)^2.\end{equation}
In particular, for any vector $y$ that satisfies $\sin \pi_i y_i=0$, we have
\begin{equation} \lambda_2 \geq 1 - \frac{\displaystyle{\sin\sjn \pi_i a_{ij}
 (y_i - y_j)^2}}
{\displaystyle{2\sin \pi_i y_i^2}}.
 \label{mineigensym}
\end{equation}
\end{lemma}
\begin{proof}
Let $D$ be a diagonal matrix whose $i$th diagonal entry
is $\pi_i$. Condition (c) yields $DA=A^TD$. We define the inner product
$\langle\cdot,\cdot\rangle_{\pi}$ by
$\langle x,y\rangle_{\pi}=x^TDy$. We then have
$$\langle x,Ay\rangle_{\pi}=x^TDAy=x^TA^T Dy=\langle Ax,y\rangle_{\pi}.$$
Therefore, $A$ is self-adjoint with respect to this inner product,
which proves that $A$ has real eigenvalues.

Since the
largest eigenvalue is $1$, with an eigenvector of ${\bf 1}$, we use the variational
characterization of the eigenvalues of a self-adjoint matrix
(Chapter 7, Theorem 4.3 of \cite{T05}) to obtain
\begin{eqnarray*}
\lambda_2 &=& \max_{x\in S} \langle x,Ax\rangle_{\pi}\\
 & =
&  \max_{x\in S} \sum_{i=1}^n\sum_{j=1}^n \pi_i a_{ij}
x_i x_j \\
 &=& \frac{1}{2} \max_{x\in S} \sum_{i=1}\sum_{j=1}
\pi_i a_{ij}(x_i^2 + x_j^2 - (x_i - x_j)^2).
\end{eqnarray*}
For $x \in S$, we have
$$\sum_{i=1}^n \sum_{j=1}^n \pi_i a_{ij} (x_i^2 +
x_j^2) =2\sin\sjn \pi_i a_{ij} x_i^2 = 2\sum_{i=1}^n \pi_i x_i^2 =
\alexo{2 \langle x,x \rangle_{\pi}=} 2,$$ which yields
$$
\lambda_2 = 1- \frac{1}{2}\min_{x\in S}\sum_{i=1}^n\sum_{j=1}^n
\pi_i a_{ij} (x_i - x_j)^2. \label{mineigeq} $$ Finally,
Eq.~(\ref{mineigensym}) follows from (\ref{eq:gg}) by considering
the vector $x_i=y_i/\sqrt{(\sjn \pi_j y_j^2)}$.
\end{proof}

\aoc{With the previous lemma in place, we can now prove Theorem \ref{thm:lb}.}

\begin{proof}[Proof of Theorem \ref{thm:lb}]
If the entries of $A$ were all nonnegative, we would be dealing with a birth-death Markov chain. Such a chain is reversible, i.e., satisfies the detailed balance equations $\pi_ia_{ij}= \pi_j a_{ji}$ (condition (c) in Theorem \ref{th:2}). In fact the derivation of the detailed balance equations does not make use of nonnegativity; thus, detailed balance holds in our case as well. Since $\pi_i=1$ by assumption, we have that our matrix $A$ is symmetric.

\aoc{For $i=1,\ldots,n$, let $y_i=i-(n+1)/2$; observe that
$\sin y_i=0$. We will make use of the inequality (\ref{mineigensym}).
Since $a_{ij}=0$ whenever $|i-j|>1$, we have
\begin{equation}\sin\sjn \pi_i a_{ij} (y_i-y_j)^2 \leq \sin\sjn a_{ij} =n.\label{eq:aa}\end{equation}
Furthermore,
\begin{equation}\label{eq:bb}
\sin \pi_i y_i^2
= \sin (i-\frac{n+1}{2})^2
 \geq
  \frac{n^3}{12}.\end{equation} The last inequality follows
from the well known fact ${\rm var}(X)= (n^2-1)/12$ for a discrete uniform
random variable $X$.
Using the inequality (\ref{mineigensym}) and Eqs.\ (\ref{eq:aa})-(\ref{eq:bb}), we obtain
the desired bound on $\rho$.}\end{proof}

\bigskip

\alexo{\noindent {\bf Remark:} Note that if the matrix $A$ is as in
the previous theorem, it is possible for the iteration
$x(t+1)=Ax(t)$ not to converge at all. Indeed, nothing in the
argument precludes the possibility that the smallest eigenvalue is
$-1$, for example. In such a case, the lower bounds of the theorem ---
derived based on bounding the second largest eigenvalue --- still hold
as the convergence rate and time are infinite.}

\smallskip

\noindent \aoc{ {\bf Remark:}  We could have saved ourselves a few
lines by appealing to the results of \cite{BDSX06} once we showed
$A$ is symemtric.  }

\bigskip

\section{Proof of the main theorem }

\aoc{We are now in a position to finally prove the main result of this chapter.}

\noindent {\bf Proof of Theorem \ref{mainthm}}:  Let $v$ be an eigenvector of $A$ with the properties in  part 5 of Lemma \ref{aproperties}. Fix a positive integer $k$. Let $\epsilon>0$ and pick $x \neq \0$ to be a small enough multiple of $v$ so that \[ \frac{\|f^k(x)-A^k(x)\|_2}{\|x\|_2} \leq \epsilon. \] This is possible by Lemma \ref{compositionapprox}.
Then, we have \[ \frac{\|f^k(x)\|_2}{\|x\|_2} \geq \frac{\|A^k x\|_2}{\|x\|_2} - \epsilon \geq  \left( 1-\frac{6}{n^2} \right) ^k - \epsilon. \] Using the orthogonality property $x^T {\bf 1} =0$, we have ${\bar x}=0$. Since we placed no restriction on $\epsilon$, this implies that
\[ \inf_{x \neq 0} \frac{\|f^k(x)-{\bar x}{\bf 1}\|_2}{\|x -{\bar x}{\bf 1}\|_2}=
\inf_{x \neq 0} \frac{\|f^k(x)\|_2}{\|x\|_2} \geq \left( 1 - \frac{6}{n^2}\right)^k \] Plugging $k=\tau(n,\epsilon)$ into this equation, we see that
\[ \left(1- \frac{6}{n^2}\right)^{\tau(n,\epsilon)} \leq \epsilon.\] Since $n \geq 3$, we have $1-6/n^2 \in (0,1)$, and
\[ \tau(n,\epsilon) \geq \frac{1}{\log (1-6/n^2)} \log \epsilon.\] Now using the bound $\log(1-\alpha) \geq 5(\alpha-1)$ for $\alpha \in [0,2/3)$, we get
\[ \tau(n,\epsilon) \geq \frac{n^2}{30} \log \frac{1}{\epsilon}.\] \noindent \qed
\bigskip

\noindent {\bf Remark:} \ao{We now sketch the proof of
 the claim we made earlier that a local averaging algorithm cannot average in finite time. \ao{Fix $n \geq 3$.} Suppose that for any $x(0)$ in some ball ${\cal B}$ around the origin, a local averaging algorithm results in $x(t) = {\bar x} \bf 1$ for all $t \geq T$.}

\ao{The proof of Theorem 1 shows that given any $k, \epsilon >0$, one can pick a vector $v(\epsilon)$ so that if $x(0)=v(\epsilon)$ then \ao{ $V(x(k))/V(x(0)) \geq (1-6/n^2)^k - \epsilon$. Moreover, the vectors $v(\epsilon)$ can be chosen
to be arbitrarily small. One simply picks $k = T$ and $\epsilon < (1-6/n^2)^k$ to get that $x(T)$ is not a multiple of $\1$; and furthermore, picking $v(\epsilon)$ small enough in norm to be in ${\cal B}$ results in a contradiction. } }

\bigskip

\noindent {\bf Remark:} Theorem \ref{mainthm} gives a lower bound on how long we must wait for the $2$-norm $\|x(t)-\bar{x} \1\|_2$ to
shrink by a factor of $\epsilon$. What if we replace the $2$-norm with other norms, for example with the $\infty$-norm? Since
 ${\cal B}_{\infty}(\0, r/\sqrt{n}) \subset {\cal B}_{2}(\0,r) \subset {\cal B}_{\infty}(\0, r)$, it follows that if the $\infty$-norm shrinks by a factor of $\epsilon$, then the $2$-norm must shrink by at least $\sqrt{n} \epsilon$. Since $\epsilon$ only enters the lower bound of Theorem \ref{mainthm} logarithmically, the  answer only changes by a factor of $\log n$ in passing to the $\infty$-norm. A similar argument shows that, modulo some logarithmic factors, it makes no difference which $p$-norm is used.

\section{Concluding remarks}

We have proved a lower bound on the convergence time of local averaging algorithms which scales
 quadratically in the number of agents. This lower bound holds even if all the communication graphs are equal to a
 fixed line graph.  Our work points to a number of open questions.

\begin{enumerate} \item Is it possible to loosen the definition of local averaging algorithms to encompass a wider class of
algorithms? In particular, is it possible to weaken the requirement that each $f_{i,G(t)}$ be smooth, perhaps only
to the requirement that it be piecewise-smooth or continuous, and
still obtain a $\Omega(n^2)$ lower bound?
\item Does the worst-case convergence time change if we introduce some memory and allow
$x_i(t+1)$ to depend on the last $k$ sets of messages received by
agent $i$? Alternatively, there is the broader question of how much is
there to be gained if every agent is allowed to keep track of extra
variables. Some positive results in this direction were obtained
in \cite{JSS09}.
\item What if each node maintains a small
number of update functions, and is allowed to choose which of them to apply? Our lower
bound does not apply to such schemes, so it is an open question whether its possible
to design practical algorithms along these lines with worst-case convergence
time scaling better than $n^2$.
\end{enumerate}

\aoc{In general, it would be nice to understand the relationship between the structure
of classes of averaging algorithms (e.g., how much memory they use, whether the updates are linear)
and the best achievable performance.}

\chapter{Quantized averaging \label{qanalysis}}

In this chapter, we consider a quantized version of the update rule in Eq.
(\ref{eq:basicupdate}). This model is a good approximation for a network of
nodes communicating through {finite bandwidth channels, so that} at
each time instant, only a finite number of bits can be transmitted.
We incorporate this constraint in our algorithm by assuming that
each node, upon receiving the values of its neighbors, computes the
convex combination $\sum_{j=1}^{{n}} a_{ij}(k) x_j(k)$ and quantizes
it. This update rule also captures {a} constraint that each node
can only store quantized values.

Unfortunately, {under Assumptions \ref{assumpt:weights}, \ref{assumpt:boundedintervals}, and \ref{assumpt:ds}, if the output of Eq.\ (\ref{eq:basicupdate}) is rounded
to the nearest integer, the sequence $x(k)$} is not guaranteed to
converge to consensus; see \cite{KBS06} for an example. We therefore choose a
quantization rule that rounds
the values down, 
according to
\begin{equation} \label{quantupdate}
x_i(k+1) = \left\lfloor \sum_{j=1}^{{n}} a_{ij}(k) x_j(k)
\right\rfloor,
\end{equation}
where $\lfloor \cdot \rfloor$ represents rounding {\it down} to the
nearest multiple of $1/Q$, and where $Q$ is some positive
integer.

We adopt the natural assumption that the initial values are
{already} quantized. \begin{assumption}
For all $i$, $x_i(0)$ is a multiple of
$1/Q$.\label{quantizedinitials}
\end{assumption} \aoc{We will next demonstrate that starting from multiples of $1/Q$, arbitrarily accurate quantized
computation of the average is possible provided the number of bits used to quantize is
at least on the order of $\log n$.  Moreover, the time scaling with $n$ required to do this is still on the order of $n^2$,
as in the previous chapters. Thus
the above Eq. (\ref{quantupdate}), despite its simplicity, turns out to have excellent performance in the quantized
setting.}

\aoc{Our exposition in this chapter will follow the paper
\cite{NOOT07}, where the results described here have previously
appeared.}

\section{A quantization level dependent bound}

For convenience we define \[ U = \max_i x_i(0), \qquad L = \min_i
x_i(0). \] We use $K$ to denote the total number of {relevant}
quantization levels, i.e., \[ K = (U-L)Q, \] which is an integer by
Assumption \ref{quantizedinitials}.

We first present a {convergence time bound that depends on the
quantization level $Q$.}

\begin{proposition} \label{simpleq}
Let Assumptions \ref{assumpt:weights} (non-vanishing weights),
\ref{assumpt:ds} (double stochasticity),
\ref{assumpt:boundedintervals} ($B$-connectivity), and
\ref{quantizedinitials} (quantized initial values) hold. Let
$\{{x(k)}\}$ be generated by the update rule (\ref{quantupdate}). If
$k \geq nBK$, {then} all {components} of $x(k)$ are equal.
\end{proposition}

\begin{proof} Consider the nodes {whose initial
value is $U$.} There are at most $n$ of them. {As long as} not all
entries of ${x(k)}$ are equal, then every $B$ iterations, at least
one {node} must use a value strictly less than $U$ {in} an update;
\an{such a} node will have its value decreased to $U-1/Q$ or less. It
follows that after $nB$ iterations, the largest node {value will be}
at most $U-1/Q$. Repeating this argument, we {see} that at most
$nBK$ iterations are possible before {all the nodes have} the same
value.
\end{proof}

Although the above bound gives informative results for small $K$, it
becomes weaker as {$Q$ (and, therefore, $K$)} increases.
On the other hand, as $Q$ approaches infinity,
the quantized system approaches the unquantized system; the availability
of convergence time bounds for the unquantized system suggests that
similar bounds should be possible for the quantized one. Indeed,  in
the next section, \ao{ we adopt a notion of convergence time
parallel to our notion of convergence time for the unquantized
algorithm; as a result, we} obtain a bound on the convergence time
{which is} independent of the total number of quantization levels.

\section{A quantization level independent bound}

We adopt a slightly different measure of convergence for the
analysis of the quantized consensus algorithm. For any $x\in\R^n$,
we define $m(x)=\min_i x_i$ and
\[\underline{V}(x)=\sum_{i=1}^n (x_i - m(x))^2.\]
We will also use the simpler notation $m(k)$ and $\underline{V}(k)$
to denote $m(x(k))$ and $\underline{V}(x(k))$, respectively, where
it is more convenient to do so. The function $\underline{V}$ will be
our Lyapunov function for the analysis of the quantized consensus
algorithm. The reason for not using our earlier Lyapunov function,
$V$, is that for the quantized algorithm, $V$ is not guaranteed to
be monotonically nonincreasing in time. On the other hand, we have
that $V(x)\leq \underline{V}(x) \leq \alexo{4} \ao{n}V(x)$ for
any\footnote{\alexo{The first inequality follows \jnt{because}
$\sum_{i} (x_i - z)^2$ is minimized when $z$ is the mean of the
vector $x$; to establish the second inequality, observe that it
suffices to consider the case when the mean of $x$ is $0$ and
$V(x)=1$. In that case,  the largest distance between $m$ and any
$x_i$ is $2$ by the triangle inequality, so $\underline{V}(x) \leq
4n$.}}  $x \in \R^n$. As a consequence, any convergence time bounds
expressed in terms of $\underline{V}$ translate to essentially the
same bounds expressed in terms of $V$, \ao{up to a logarithmic
factor}.

Before proceeding, we {record an elementary fact which will allow us
to} relate the variance decrease $V(x)-V(y)$ to the decrease,
$\underline{V}(x)- \underline{V}(y)$, of our new Lyapunov function.
The proof involves {simple algebra,} and is therefore omitted.

\begin{lemma} \label{sumlemma} Let $u_1,\ldots,u_{n}$ and
$w_1,\ldots,w_n$ be real numbers satisfying \[ \sum_{i=1}^n u_i =
\sum_{i=1}^n w_i.
\] {Then, the expression}
\[ f(z) = \sum_{i=1}^n (u_i - z)^2 - \sum_{i=1}^n (w_i - z)^2 \]
is {a constant,}
independent of the scalar $z$.
\end{lemma}

Our next lemma places a bound on {the decrease of} the Lyapunov
function $\underline{V}(t)$ between times $kB$ and $(k+1)B-1$.

\noindent \begin{lemma} \label{quantdiff}
Let Assumptions
\ref{assumpt:weights}, \ref{assumpt:ds}, \ref{weakconnect}, and \ref{quantizedinitials} hold.
Let $\{x{(k)}\}$ be generated by the update rule
(\ref{quantupdate}). Suppose that the components $x_i(kB)$ of
the vector $x(kB)$ have been ordered from largest to smallest,
with ties broken arbitrarily. Then, we have
\[ \underline{V}(kB) - \underline{V}((k+1)B)
\geq \frac{\eta}{2}\, \sum_{i=1}^{n-1} (x_{i}(kB)
- x_{i+1}(kB))^2.\]
\end{lemma}

\begin{proof} For all $k$, we view Eq.\ (\ref{quantupdate}) as the
composition of two operators: \[ y(k) = A(k) x(k), \] where $A(k)$
is {a} doubly stochastic matrix, and \[ x(k+1) = \lfloor y(k)
\rfloor,
\] where the quantization $\lfloor \cdot \rfloor$ is {carried out}
componentwise.

We apply Lemma \ref{sumlemma} with the identification
$u_i=x_{{i}}(k)$, $w_i=y_{{i}}(k)$. Since multiplication by a doubly
stochastic matrix preserves the mean, the condition $\sum_i u_i =
\sum_i w_i$ is satisfied. By considering two different choices for
the scalar $z$, namely, $z_1=\bar{x}(k)=\bar{y}(k)$ and $z_2=m(k)$, we obtain
\begin{equation} \label{firststep} V(x(k)) - V(y(k)) =
\underline{V}(\ao{x(k)}) - \sum_{i=1}^n (y_i(k)-m(k))^2.
\end{equation}
Note that $x_i(k+1)-m(k)\leq y_i(k)-m(k)$. Therefore,
\begin{equation} \label{secondstep} \underline{V}(\ao{x(k)}) -
\sum_{i=1}^n (y_i(k)-m(k))^2 \leq \underline{V}(\ao{x(k)}) -
\sum_{i=1}^n (x_i(k+1)-m(k))^2.
\end{equation}
Furthermore, note that \alexo{since $x_i(k+1) \geq m(k+1) \geq m(k)$
for all $i$, we have that } $x_i(k+1)-m(k+1)\leq x_i(k+1)-m(k)$.
Therefore,
\begin{equation} \label{thirdstep} \underline{V}(\ao{x(k)}) -
\sum_{i=1}^n (x_i(k+1)-m(k))^2 \leq \underline{V}(\ao{x(k)}) -
\underline{V}(\ao{x(k+1)}).
\end{equation}
By combining
Eqs.\ (\ref{firststep}), (\ref{secondstep}), and (\ref{thirdstep}),
we obtain
\[ V(x(t)) - V(y(t)) \leq \underline{V}(\ao{x(t)}) -
\underline{V}(\ao{x(t+1)})\qquad \hbox{for all }t. \]
Summing the
preceding relations over $t=kB,\ldots, (k+1)B-1$, we further obtain
\[
\sum_{t=kB}^{(k+1)B-1} \Big(V(x(t))-V(y(t))\Big) \leq
\underline{V}(\ao{x(kB)}) - \underline{V}(\ao{x((k+1)B))}.
\]

To complete the proof, we provide a lower bound on the expression
\[
\sum_{t=kB}^{(k+1)B-1} \Big(V(x(t)) - V(y(t))\Big). \] Since
$y(t)=A(t)x(t)$ for all $t$, it follows from Lemma \ref{vl} that for
any $t$, \[ V(x(t))-V(y(t)) = \sum_{i<j} w_{ij}(t)
(x_i(t)-x_j(t))^2,\] where $w_{ij}(t)$ is the $(i,j)$-th entry of
$A(t)^T A(t)$. Using this relation and following the same line of
analysis used in the proof of Lemma \ref{vardiff} [where the
relation $\alexo{x_i(t) \geq y_{d_1}}$
holds in view of the assumption that $x_i(kB)$ is a multiple of
$1/Q$ for all $k\ge 0$, cf.\ Assumption \ref{quantizedinitials}] ,
we obtain the desired result.
\end{proof}

\vspace{1pc} The next theorem contains our main result on the
convergence time of the quantized algorithm.

\noindent \begin{theorem} \label{qbound} Let Assumptions
\ref{assumpt:weights} (non-vanishing weights), \ref{assumpt:ds}
(double stochasticity), \ref{weakconnect} (connectivity relaxation),
and \ref{quantizedinitials} hold. Let $\{x{(k)}\}$ be generated by
the update rule (\ref{quantupdate}). Then, there exists an absolute
constant $c$ such that we have
$$\underline{V}(k) \leq \epsilon \underline{V}(0)\qquad \hbox{for all }
k\ge c\, (n^2/\eta) B
\log(1/\epsilon).$$

\end{theorem}

\begin{proof}
Let us assume that $\underline{V}(kB)>0$. From Lemma
\ref{quantdiff}, we have
\[ \underline{V}(kB)- \underline{V}((k+1)B) \geq \frac{\eta}{2}
\sum_{i=1}^{n-1} (x_{i}(kB) - x_{i+1}(kB))^2,  \] where {the
components} $x_i(kB)$ are ordered from largest to smallest. Since
$\underline V(kB) = \sum_{i=1}^n (x_i(kB) - x_n(kB))^2$, we have \[
\frac{\underline{V}(kB)-\underline{V}((k+1)B)}{\underline{V}(kB)}
\geq \frac{\eta}{2} \frac{\sum_{i=1}^{n-1} (x_{i}(kB) -
x_{i+1}(kB))^2}{\sum_{i=1}^n (x_i(kB)-x_n(kB))^2}.
\] {Let} $y_i = x_i(kB)-x_n(kB)$. Clearly, $y_i \geq 0$ for all $i$,
and $y_n=0$. Moreover, the monotonicity of $x_i(kB)$ implies the
monotonicity of $y_i$: \[ y_1 \geq y_2 \geq \cdots \geq y_n=0.
\] Thus,
\[ \frac{\underline{V}(kB)-\underline{V}((k+1)B)}{\underline{V}(kB)}
\geq \frac{\eta}{2} \min_{{y_1 \geq y_2\ge \cdots\ge y_n\atop
y_n=0}}\frac{\sum_{i=1}^{n-1} (y_i - y_{i+1})^2}{\sum_{i=1}^n
y_i^2}.
\] \alexo{Next, we simply repeat the steps of Lemma \ref{lboundvar}. We can
assume without loss of generality that $\sum_{i=1}^n y_i^2=1$.
\an{Define $z_i=y_{i}-y_{i+1}$ for $i=1,\ldots,n-1$ and $z_n=0$. We}
have that $z_i$ are all nonnegative
and $\sum_{i} z_i = y_1 - y_n \geq 1/\sqrt{n}$.} \an{Therefore,}
\[ \frac{\eta}{2} \min_{{y_1 \geq y_2\geq \cdots \geq y_n\atop
\sum_i y_i^2=1}} \sum_{i=1}^{n-1} (y_i - y_{i+1})^2 \geq
\frac{\eta}{2} \min_{z_i \geq 0, \sum_i z_i \geq 1/\sqrt{n}}
\sum_{i=1}^n z_i^2.
\]
The minimization problem on the right-hand side has an optimal value of at
least $1/n^2$, and the desired result follows.
\end{proof}

\section{Extensions and modifications}

In this section, we  comment briefly on some corollaries
of Theorem \ref{qbound}.

First, we note that the results of Section \ref{matrixpicking}
immediately carry over to the quantized case. Indeed, in Section
\ref{matrixpicking}, we showed how to pick the weights $a_{ij}(k)$
in a decentralized manner, based  only on local information, so that
Assumptions 1 and \ref{weakconnect} are satisfied, with $\eta \geq
1/3$. When using a quantized version of the load-balancing algorithm of Chapter \ref{nsquared},
we once again \alexo{ manage to remove the factor of $1/\eta$ from
our upper bound.}

\begin{proposition} \label{savingnq}
For the quantized version of the load-balancing algorithm of Chapter \ref{nsquared},
and under the same assumptions as in Theorem \ref{savingn}, if
$k\geq c\, n^2B \log (1/\epsilon))$, then $\underline{V}(k ) \leq
\epsilon \underline{V}(0)$, where $c$ is an absolute
constant.\end{proposition}

Second, we note that Theorem \ref{qbound} {can be used} to obtain a
bound on the time until {the values of all nodes are equal. Indeed,
we observe that in the presence of quantization, once the condition
$\underline{V}(k) <1/Q^2$ is satisfied, all components of $x(k)$
must be equal.}

\begin{proposition} \label{qboundequal}
Consider the quantized algorithm (\ref{quantupdate}), and assume
that Assumptions \ref{assumpt:weights} (non-vanishing weights),
\ref{assumpt:ds} (double stochasticity), \ref{weakconnect}
(connectivity relaxation), and \ref{quantizedinitials} (quantized
initial values) hold. If $k\geq c(n^2/\eta) B \big[\log Q + \log
\underline{V}(0)\big]$, then all components of $x(k)$ are equal,
where $c$ is an absolute constant.
\end{proposition}

\section{Quantization error}

Despite favorable convergence properties of our quantized
{averaging} algorithm (\ref{quantupdate}), the update rule does not
preserve the average of the values at each iteration. Therefore, the
{common limit of the sequences $x_i(k)$, denoted by $x_f$,}
need not be equal to the exact average of the initial values. We
next provide an upper bound on the error between {$x_f$} and the
initial average, as a function of the number of quantization levels.

\begin{proposition} \label{qerror}
There is an absolute constant $c$ such that
for the common limit $x_f$ of the
values $x_i(k)$ generated by the
quantized algorithm (\ref{quantupdate}), we have
\[ \left|x_f - \frac{1}{n} \sum_{i=1}^n x_i(0)\right| \le {c\over Q}
\ {n^2\over\eta}\, B\log (Qn(U-L)).\]
\end{proposition}

\begin{proof} By \alexo{Proposition} \ref{qboundequal}, after $O\Big((n^2/\eta) B \log (Q
\underline{V}(x(0)))\Big)$ iterations, all nodes will have the same
value. Since $\underline{V}(x(0)))\le n(U-L)^2$ and the average
decreases by at most $1/Q$ at each iteration, the result follows.
\end{proof}

\vspace{1pc}


Let us assume that the parameters 
$B$, $\eta$, and $U-L$ are fixed. \alexo{Proposition} \ref{qerror}
implies that as $n$ increases, the number of bits {used for each
communication,} which is proportional to $\log Q$, needs to grow
only as $O(\log n)$ to make the error negligible. Furthermore, this
is true even if the {parameters} $B$, ${1}/{\eta}$, and $U-L$ grow
polynomially in $n$.

\begin{center}
\begin{figure}
\hspace{2cm}
\includegraphics*[width=12cm]{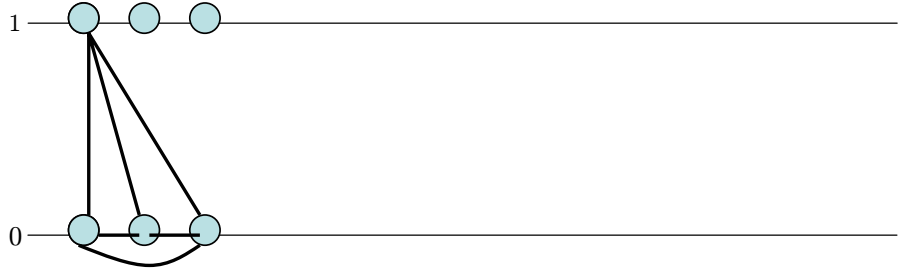}
\caption{\label{ic} Initial configuration. Each node takes the
average value of its neighbors. }
\end{figure} \end{center} {For a converse, it can} be seen that ${\Omega}(\log n)$ bits are
needed. Indeed, consider  $n$ nodes, with $n/2$ {nodes initialized
at} $0$, and $n/2$ {nodes initialized at} $1$. {Suppose} that ${Q} <
n/2$; we connect the nodes by {forming} a complete subgraph over all
the nodes with value $0$ and exactly {one} node with value $1$; see
Figure \ref{ic} for an example with $n=6$. Then, each node {forms
the} average {of} its neighbors. This brings one of the nodes with
{an initial value of} $1$ down to $0$, without raising the value of
any {other} nodes. We can repeat this {process,} to bring all of the
nodes with {an initial value of} $1$ down to $0$. Since the true
average is $1/2$, the final result is $1/2$ away from the true
average. {Note now that $Q$ can} grow linearly with $n$, and still
satisfy the inequality ${Q}<n/2$. {Thus,} the number of bits can
grow as $\Omega(\log n)$, and yet, independent of $n$, the error
remains $1/2$. 
\section{Concluding remarks}

\aoc{The high-level summary of this chapter is that $c \log n$ bits suffice 
for quantized averaging. The answers obtained will not be exact, but the error can be made arbitrarily small by
picking $c$ large, and the favorable convergence times from the earlier chapters retained.}

\aoc{An interesting direction is whether its possible to reduce the number of bits
even further, to a constant number per each link maintained by a node (at least one bit
per link is necessary merely to store incoming messages). A positive answer in the case of fixed graphs is provided
by the following chapter. On dynamic graphs, the number of bits required by
deterministic averaging algorithms is still not understood.}
\chapter{Averaging with a constant number of bits per link \label{ch:constantstorage}}

\aoc{The previous chapter analyzed the performance of a quantized update rule. The ``punchline'' was that
if the number of bits involved in the quantization is on the order of $\log n$,
then arbitrarily accurate averaging is possible.}

\aoc{This chapter analyzes what happens if we try to push down the
number of bits stored at each node to a constant. \aoc{Our
exposition in this chapter will follow the preprint \cite{HOT10},
where the results described here have previously appeared.}}

\aoc{We will assume that at any time step, nodes can exchange binary messages
with their neighbors in an interconnection graph which is undirected
and unchanging with time. Naturally, a node needs to store at least one
bit for each of the links just to store the message arriving on that link.
We will allow the nodes to maintain a constant number of bits for each of their links. Thus in
a constant-degree network, this translates to a constant number of bits at each node.}

\aoc{Supposing that the nodes begins with numbers $x_i \in \{0,1,\ldots,K\}$, we will say that a function
of the initial values is {\em computable with constant storage} if it computable subject to the
restrictions in the pevious paragraph. The exact average is not computable with constant storage because
just storing the average takes $\Omega (\log n)$ bits. Thus we will turn to the related, but weaker, question of deciding
whether the $0$'s or $1$'s are in the majority.}

\aoc{In fact, we will study a more general problem called ``interval averaging,'' which asks to return
the set among $ \{ 0 \}, (0,1), \{1\}, (1,2), \ldots, (K-1,K), \{K\}$
within which the average of the initial numbers lies. If interval averging is possible, then majority computation
with binary initial conditions is possible as well. Indeed, the argument for this is simple: to do majority computation, every node beginning with $x_i=1$ instead sets $x_i=2$ and runs interval averaging with $K=2$. Depending on whether the average is in $\{0,\}(0,1), \{1\}, (1,2), \{2\}$ each node knows which initial condition (if any) had
the majority. }

\aoc{The interval averaging problem has trivial solutions with randomized algorithms (see Chapter \ref{why}).
These algorithms ``centralize'' the problem by electing a leader and streaming information
towards the leader. Decentralized randomized algorithms are also possible \cite{KBS06, FCFZ08}.
In this chapter, we address the question of whether it is possible to solve this problem
with deterministic algorithms.}

\aoc{The main result of this chapter is the following theorem.}

\begin{theorem} \label{thm:average_computation} Interval averaging is computable with constant storage. Moreover, there exists
a (constant storage) algorithm for it under which every node has the correct answer after
$O(n^2 K^2 \log K)$ rounds of communication.
\end{theorem}

\aoc{The remainder of this chapter will be devoted to the proof of this fact. Let us briefly summarize the
main idea behind the computation of interval averages. Imagine the integer \jnt{input} value $x_i$
\green{as represented by a number of $x_i$} pebbles \jblue{at node
$i$.} \jnt{The} algorithm attempts to exchange pebbles between
nodes with unequal numbers so that the overall distribution
\green{becomes} more even. Eventually, either all nodes will have
the same number of pebbles, or some will have a certain number and
others just one more. The nodes will try to detect which of these two possibilites have occured, and will estimate
the interval average accordingly.}

\aoc{The remainder of this chapter is structured as follows. First, we discuss the problem of
tracking, in a distributed way, the maximum of time-varying values at each node. This is a seemingly unrelated problem, but
it will be useful in the proof of Theorem \ref{thm:average_computation}. After describing a solution to this problem, we give
a formal algorithm which implements the pebble matching idea above. This algorithm will use
the maximum tracking algorithm we developed as a subroutine to match up nodes with few pebbles to
nodes with a large number of pebbles.}

\section{Related literature}

A number of papers explored various tradeoffs associated with quantization of consensus schemes. We will not attempt
a survey of the entire literature, but only mention papers that are most closely relevant to the algorithms 
presented here. The paper \cite{KBS06} proposed randomized {\it
gossip-type} quantized averaging algorithms under the assumption
that each agent value is an integer. They showed that these
algorithms preserve the average of the values at each iteration and
converge to approximate consensus. They also provided bounds on the
convergence time of these algorithms for specific static topologies
(fully connected and linear networks). We refer the reader also to the
later papers \cite{ZM09} and \cite{FCFZ08}. A dynamic scheme which allows us to approximately compute the average   as the nodes communicate more and more bits with each other can be found in \cite{CBZ10}. In the recent work
\cite{CFFTZ07}, Carli {\it et al.} proposed a distributed algorithm
that uses quantized values and preserves the average at each
iteration. They showed favorable convergence properties using
simulations on some static topologies, and provided performance
bounds for the limit points of the generated iterates. 

\section{Computing and tracking maximal values\label{sec:max_tracking}}
\rt{We now describe an algorithm that tracks the maximum (over all
nodes) of time-varying inputs at each node. It will be used as a
subroutine \aoc{later}}. \jnt{The basic idea is simple: every node keeps track of the largest value
it has heard so far, and forwards this ``intermediate result'' to
its neighbors. However, when an input value changes, the existing
intermediate results \aoc{need to be} invalidated, and this is done by sending
``restart'' messages. A complication arises because invalidated
intermediate results might keep circulating in the network, always one
step ahead of the restart messages. We deal with this difficulty by
``slowing down'' the intermediate results, so that they travel at
half the speed of the restart messages.} \jblue{In this manner,
restart messages are guaranteed to eventually catch up with and
remove invalidated intermediate results.}

\jnt{We start by giving the specifications of the algorithm.}
Suppose that each node $i$ has \green{a time-varying input} $u_i(t)$
stored in memory at time $t$, \green{belonging to a finite} set of
numbers ${\cal U}$. \green{We assume that, for each $i$, the
sequence $u_i(t)$} must eventually stop changing, i.e., \green{that
there exists} some $T'$ \green{such that}
\[ u_i(t) = u_i(T'), ~~~~~~~~\mbox{ for all  } \jnt{i \mbox{ and }}t \geq T'.\]
(However, \jnt{node $i$} need not \jnt{ever be} aware that $u_i(t)$
has reached its final value.) Our goal is to develop a distributed
algorithm \green{whose output eventually settles on the value
$\max_i u_i(T')$. More precisely, each} node $i$ is to maintain a
number $M_i(t)$ which must satisfy the following \jblue{condition:}
\red{for every \jnt{network}} \green{and any allowed sequences
$u_i(t)$,} there exists some $T''$ with
\[ M_i(t) = \max_{\jnt{j}=1,\ldots,n} u_j(t), ~~~~~\mbox{  for all  }
\jnt{i \mbox{ and }}t \geq T''. \]

Moreover, \jnt{each} node $i$ must also maintain a pointer $P_i(t)$
to a neighbor or to itself. We will use the notation
$P^2_i(t)=P_{P_i(t)}(t)$, $P^3_i(t)=P_{P_i^2(t)}(t)$, etc. \green{We
require} the following \green{additional} property, for all $t$
larger than $T''$: for each node $i$ there exists a node $\green{j}$
and a power $K$ such that for all $k \geq K$ we have $P^k_i\aoc{(t)} =  j$ and $M_i(t) = u_j(t)$. In  words, by successively following
the pointers $P_i(t)$, one \green{can} arrive at a node with a
maximal value.

 \jmj{We} next describe the algorithm. \jnt{We will use the term {\em slot $t$} to refer, loosely
speaking, to the interval between times $t$ and $t+1$. More
precisely, during slot $t$} each node processes the messages that
have arrived at time $t$ and computes the state at time $t+1$ as
well as the messages it will send at time $t+1$.

The variables $M_i(t)$ and $P_i(t)$ are a complete description of
the state of node $i$ at time $t$. Our algorithm has only two types
of messages \jnt{that a node} can send to \jnt{its} neighbors. Both
are broadcasts, in the sense that the node sends them to every
neighbor:
\begin{enumerate}
\item[1.]
``Restart!'' \item[2.] ``My \jnt{estimate of the maximum} is $y$,''
where $y$ is some number in \jnt{${\cal U}$} \jnt{chosen} by the
node.
\end{enumerate}

\newpage

\begin{figure}[h] \hspace{0cm}
 \epsfig{file=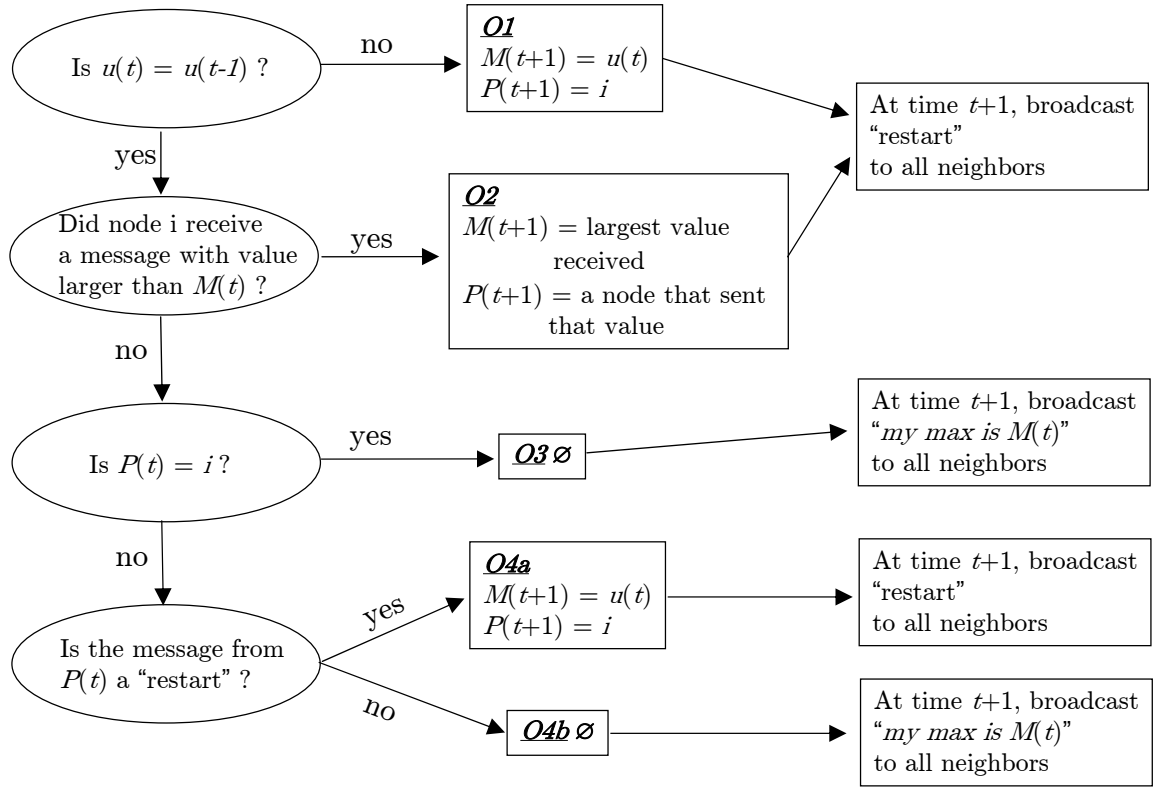, scale=0.65}
\caption{\jnt{Flowchart of the procedure used by node $i$ during
slot $t$ in the maximum tracking algorithm. The subscript $i$ is
omitted, but $u(t)$, $M(t)$, and $P(t)$ should be understood as
$u_i(t)$, $M_i(t)$, and $P_i(t)$. In those cases where an updated
value of $M$ or $P$ is not indicated, it is assumed that
$M(t+1)=M(t)$ and $P(t+1)=P(t)$. \jblue{The symbol $\emptyset$ is
used to indicate no action. \jred{Note that the various actions
indicated are taken during slot $t$, but the messages determined by
these actions are sent (and instantaneously received) at time
$t+1$.}
 }}\jmj{Finally, observe that every node sends an \aoc{identical}  message to all its
neighbors at every time $t>0$.} \aoc{We note that the apparent non-determinism in instruction
O2 can be removed by picking a node with, say, the smallest port label.} }\label{fig:maxtracking}
\end{figure}

\newpage

Initially, each node sets $M_i(0)=\rt{u}_i(0)$ and $P_i(0)=i$. At
time $t=0,1,\ldots$, nodes exchange messages, which \jnt{then}
determine their state at time $t+1$, i.e., the pair  $M_i(t+1),
P_i(t+1)$, \jnt{as well as the messages to be sent at time $t+1$.}
The procedure node $i$ uses to do this is \jnt{described} in
Figure~\ref{fig:maxtracking}. One can verify that \jnt{a memory size
of $C\log |{\cal U}|+C\log d(i)$ at each node $i$ suffices,
\jblue{where $C$ is an absolute constant.} (This is because $M_i$
and $P_i$ can take one of $|{\cal U}|$ and $d(i)$ possible values,
respectively.)}

\aoc{Note that while we assumed that agents can only transmit binary messages
to each other, the above algorithm requires transmissions of the values
in $\{0,\ldots,K\}$. This means that each step of the above algorithm
takes $\log K$ real-time steps to implement.}

\aoe{The result that follows asserts the correctness of the
algorithm. The idea of the proof is quite simple. Nodes maintain
estimates $M_i(t)$ which track the largest among all the $u_i(t)$ in
the graph; these estimates are ``slowly" forwarded by the nodes to
their neighbors, with many artificial delays along the way. Should
some value $u_j(t)$ change, restart messages traveling without
artificial delays  are forwarded to any node which thought $j$ had a
maximal value, causing those nodes to start over. The possibility of
cycling between restarts and forwards is avoided because restarts
travel faster. Eventually, the variables $u_i(t)$ stop changing,
and the algorithm settles on the correct answer.}

\begin{theorem}[Correctness of \jnt{the} maximum tracking algorithm]\label{thm:maxtracking}
\jnt{Suppose that the $u_i(t)$ stop changing after some finite time.
Then, for every network,} there is a time after which \jnt{the
variables $P_i(t)$ and $M_i(t)$ stop changing and satisfy
$M_i(t) = \max_i \rt{u}_i(t)$; furthermore, after that time, and for
every $i$, the node $j=P_i^n(t)$ satisfies
$M_i(t)=u_j(t)$.}\end{theorem}

\def\aoe{}
\def\aon#1{{#1}}

\aoc{We now turn to the proof of Therem \ref{thm:maxtracking}.}

\aon{In the following, we will occasionally} use the following convenient shorthand:
we will say that a statement $S(t)$ {\em holds eventually} if
there exists some $T$ so that \jnt{$S(t)$ is} true for all $t \geq
T$.

The analysis is made easier by introducing the time-varying directed
graph $G(t) = (\{1,\ldots,n\}, E(t))$ where $(i,j) \in E(t)$ if
\jnt{$i\neq j$, } $P_i(t)=j$, and there is no ``restart" message
sent \jnt{by $j$ (and therefore received by $i$) during \aoe{time}
$t$.} \aon{We will abuse notation by writing $(i,j) \in G(t)$ to mean that
the edge $(i,j)$ belongs to the set of edges $E(t)$ at time $n$.}

\begin{lemma} \label{newedge} \jnt{Suppose that $(i,j) \notin G(t-1)$ and $(i,j) \in G(t)$.}
Then, $i$ executes O2 during slot $t-1$.
\end{lemma}

\begin{proof}
\jnt{Suppose that $(i,j) \notin G(t-1)$ and $(i,j) \in G(t)$.} If
$P_i(t-1)=j$, \jnt{the definition of $G(t-1)$ implies}  that $j$
sent a restart message to $i$ at time $t-1$. \jmj{Moreover, $i$
cannot execute O3 during the time slot $t-1$ as \aon{this} would
require $P_i(t-1) = i$.}\jmj{Therefore, during this time slot,}
\jnt{node $i$ either executes O2 (and we are
done\footnote{\jred{In fact, it can be shown that this case never
occurs.}}) or it executes one of O1 and O4a. For both of the
latter two cases, we will have} $P_i(t)=i$, so that $(i,j)$ will
not be in $G(t)$, \jnt{a contradiction.} Thus, it must be that
$P_i(t-1) \neq j$. We now observe that the only place in Figure
\ref{fig:maxtracking} \jnt{that can change $P_i(t-1) \aon{\neq j}$
to} $P_i(t)=j$ is O2.
\end{proof}

\begin{lemma}\label{lem:speed1/2}
In \jnt{each of the following three} cases, node $i$ has no incoming
edges in \aon{either graph $G(t)$ or $G(t+1)$}:\\
(a) Node $i$ executes O1, O2, or O4a during time slot $t-1$;\\
(b) $M_i(t) \not = M_i(t-1)$;\\
(c) For some $j$, $(i,j) \in G(t)$ but $(i,j) \notin G(t-1).$
\end{lemma}
\begin{proof} \ao{(a) If $i$ executes O1, O2, or O4a during slot $t-1$, then it sends
a restart \jblue{message} to each of its neighbors at time $t$.
Then, \jblue{for any neighbor $j$} of $i$, \jnt{the definition of
$G(t)$ implies that}  $(j,i)$ is not in  $G(t)$. Moreover, by Lemma
\ref{newedge}, in order for $(j,i)$ to be in $G(t+1)$, \jblue{node
$j$ must execute} $O2$ during slot $t$. But the execution of O2
during slot $t$ cannot result in the addition of the edge $(j,i)$
\jblue{at time $t+1$,} \jnt{because} \jmj{the message broadcast by
$i$ \aon{at time $t$} to its neighbors was a restart.} So, $(j,i) \notin G(t+1)$. }

\ao{(b)} \jnt{If} $M_i(t) \not = M_i(t-1)$, \jnt{then} $i$ executes
O1, O2, or O4a during slot $t-1$, so the \jnt{claim} follows from
part (a).

\ao{(c)} By Lemma \ref{newedge}, \jnt{it must be the case that
\jblue{node} $i$ executes} O2 during slot $t-1$, and part (a)
implies the \jnt{result}.
\end{proof}

\begin{lemma}\label{lem:acyclic}
The graph $G(t)$ \ao{is acyclic}, for \jnt{all} $t$.
\end{lemma}
\begin{proof}
The initial graph $G(0)$ does not contain a cycle. \ao{ Let $t$ be
the first time a cycle \jblue{is present}, and let $(i,j)$ be an
edge \jnt{in} \jblue{a} cycle that \jnt{is added} at time $t$,
\jnt{i.e., $(i,j)$ belongs to $G(t)$ but not $G(t-1)$.} Lemma
\ref{lem:speed1/2}\jmj{(c)} implies that $i$ has no incoming edges
\jnt{in $G(t)$,} so $(i,j)$ cannot be an edge of the cycle---a
contradiction. }
\end{proof}

\jblue{Note that every node has out-degree at most one, because
$P_i(t)$ is a single-valued variable. Thus, the acyclic graph $G(t)$
must be a forest, specifically,} \jnt{a collection of disjoint
trees, with all arcs of a tree directed so that they point towards a
root of the tree (i.e., a node with zero out-degree).} The next
lemma establishes that $M_i$ is constant on any path of $G(t)$.

\begin{lemma}\label{lem:edge->mi=mj}
If \jnt{$(i,j)\in G(t)$,} then  $M_i(t) = M_j(t)$.
\end{lemma}

\begin{proof} \ao{Let $t'$ be \jnt{a time when $(i,j)$ is added to the graph, or more precisely,
a time such that $(i,j)\in G(t')$ but $(i,j)\notin G(t'-1)$.} First,
we argue that the statement we want to prove holds at time $t'$.
Indeed, Lemma \ref{newedge} implies that during slot $t'-1$, node
$i$ executed $O2$, so that $M_i(t') = M_j(t'-1)$. Moreover,
$M_j(t'-1)=M_j(t')$, \jnt{because otherwise case (b) in} Lemma
\ref{lem:speed1/2} would imply that} $j$ has no incoming edges at
time \jnt{$t'$, contradicting our assumption that $(i,j) \in
G(t')$.}

\ao{Next, we argue that \jnt{the property $M_i(t)=M_j(t)$ continues
to hold,} \jblue{starting from time $t'$ and for} as long as $(i,j)
\in G(t)$. Indeed, \jblue{as long as} $(i,j) \in G(t)$, then
$M_j(t)$ \jblue{remains} unchanged, \jnt{by case (b) of} Lemma
\ref{lem:speed1/2}. To argue that $M_i(t)$ \jblue{also remains}
unchanged, we simply observe that in Figure \ref{fig:maxtracking}, every box
which leads to a change in $M_i$ also \jnt{sets $P_i$ either to $i$
or to the sender of a message with value strictly larger than}
$M_i$; this latter \jblue{message} cannot come from $j$ because as
we just argued, increases in $M_j$ lead to \jnt{removal} of the edge
$(i,j)$ from $G(t)$. So, changes in $M_i$ are also accompanied by
\jnt{removal} of the edge $(i,j)$ from $G(t)$.}
\end{proof}

\jnt{For the purposes of the next lemma, we use the convention
$u_i(-1)=u_i(0)$.}

\begin{lemma}\label{lem:m_geq_y}
If $P_i(t) = i$, then $M_i(t) = \rt{u}_i(t-1)$; \jnt{if $P_i(t) \neq
i$, then} $M_i(t) > \rt{u}_i(t\ao{-}1)$.
\end{lemma}

\begin{proof}
We prove this result by induction. \jnt{Because of} the convention
$\rt{u}_i(-1) = \rt{u}_i(0)$, \jblue{and the initialization
$P_i(0)=i$, $M_i(0)=u_i(0)$,} the result trivially holds at time
$t=0$. Suppose now that the result holds at time $t$.
\ao{\jnt{During time slot $t$, we have three possibilities}
\jblue{for node $i$:}
\begin{itemize}
\item[(i)] Node $i$ executes O1 or O4a. In this case,
$M_i(t+1)=\rt{u}_i(t), P_i(t+1)=i$, so the result holds at time
$t+1$. \item[(ii)] Node $i$ executes O2. In this case $P_i(t) \neq
i$ and $M_i(t+1) > M_i(t) \geq \rt{u}_i(t-1)=\rt{u}_i(t)$. \jnt{The
first inequality follows from the condition for entering step O2.
The second follows from the induction hypothesis. The} last equality
follows because if $\rt{u}_i(t) \neq \rt{u}_i(t-1)$, node $i$ would
have executed O1 rather than O2. So, once again, the result holds at
time $t+1$. \item[(iii)] Node $i$ executes O3 or O4b.  The result
holds at time $t+1$ because neither $\rt{u}_i$ nor $M_i$ changes.
\end{itemize}}
\end{proof}

In the sequel, we will use $T'$ to refer to a time after which all
the $\rt{u}_i$ are constant. The following lemma shows that, after
$T'$, the largest estimate does not increase.

\begin{lemma}\label{lem:nondecreasing}\aoe{
Suppose that at some time $t' > T'$ we have $\widehat{M}> \max_i
M_i(t')$. Then \begin{equation} \label{eq:nondecreasing}
\widehat{M}> \max_i M_i(t) \end{equation} for all $t \geq t'$.}
\end{lemma}

\begin{proof}
{ \aoe We prove Eq. (\ref{eq:nondecreasing}) by induction. By
assumption it holds at time $t=t'$. Suppose now Eq.
(\ref{eq:nondecreasing}) holds at time $t$; we will show \aon{that} it holds at
time $t+1$.}

Consider a node $i$. If it executes O2 during the slot $t$, it sets
$M_i(t+1)$ to the value contained in a message sent at time $t$ by
some node $j$. It follows from the rules of our algorithm that the
value in this message is $M_j(t)$ and therefore, $M_i(t+1) = M_j(t)
< \widehat{M}$.

Any operation other than O2 that modifies $M_i$ sets $M_i(t+1) =
u_i(t)$, {\aoe and since $u_i(t)$ does not change after time $T'$,
we have $M_i(t+1) = u_i(t-1)$. By Lemma \ref{lem:m_geq_y},
 $M_i(t) \geq u_i(t-1)$, so that $M_i(t+1) \leq
M_i(t)$. We conclude that $M_i(t+1) < \widehat{M}$ holds for this
case as well.}
\end{proof}

We now introduce some terminology used to specify whether the
estimate $M_i(t)$ held by a node has been invalidated or not.
Formally, we say that node $i$ has a \emph{\jnt{valid} estimate}
\jnt{at time $t$} if by following the path in $G(t)$ \jnt{that
starts} at $i$, \jnt{we}  eventually arrive at a node $r$ with
$P_r(t) = r$ and \jblue{$M_i(t)=u_r(t-1)$.} In any other case, we
say that a node has an \emph{\jnt{invalid} estimate} at time $t$.

\smallskip

{\bf Remark:} Because of the acyclicity property, a path
in $G(t)$, starting from a node $i$, eventually leads to a node $r$
with out-degree 0; it follows from Lemma \ref{lem:edge->mi=mj} that
$M_r(t) = M_i(t)$. Moreover, Lemma \ref{lem:m_geq_y} implies that if
$P_r(t) = r$, then \jnt{$M_i(t) =M_r(t)=\rt{u}_r(\jblue{t-1})$,} so
that the estimate is \jnt{valid.} Therefore, if $i$ has an
\jnt{invalid} estimate, the corresponding node $r$ must have $P_r(t)
\not = r$; \jnt{on the other hand,} since $r$ has out-degree 0 in
$G(t)$, \jnt{the definition of $G(t)$} implies that there is a
``restart" message from $P_r(t)$ to $r$ sent at time $t$.

\smallskip

{\aoe The following lemma gives some conditions which allow us to
conclude that a given node has reached a \aon{final} state.}

\begin{lemma}\label{lem:valid+max-->final} {\aoe
Fix some $t'>T'$ and let $M^*$ be the largest estimate at time $t'$,
i.e., $M^*=\max_i M_i(t')$. If $M_i(t')
= M^*$, and this estimate is valid, then  for all $t \geq t'$: \\
\noindent (a) $M_i(t) = M^*$, $P_i(t) = P_i(t')$, and node $i$ has a valid estimate at time $t$. \\
(b) Node $i$ executes either O3 or O4b at time $t$.}
\end{lemma}
\begin{proof} {\aoe We will prove this result by induction on $t$.
Fix some node $i$. By assumption, part (a) holds at time $t=t'$.
To show part (b) at time $t=t'$, we first argue that $i$ does not
execute O2 during the time slot $t$. Indeed, this would require
$i$ to have received a message with an estimate strictly larger
than $M^*$, sent by some node $j$ who executed O3 or O4b during
the slot $t-1$. In either case, $M^* < M_j(t-1) = M_j(t)$,
contradicting the definition of $M^*$. \aon{Because of the
definition of $T'$, $u_i(t)=u_i(t-1)$ for $t>T'$, so that $i$ does
not execute O1.} This concludes the proof of the base case. }

{\aoe Next, we suppose that our result holds at time $t$, and we will
argue that it holds at time $t+1$. If } $P_i(t) = i$, then $i$
executes O3 {\aoe during slot $t$}, so that $M_i(t+1) = M_i(t)$ and
$P_i(t+1) = P_i(t)$, {\aoe completes the induction step for this case. }

{\aoe It remains to consider the case where $P_i(t) = j \neq i$. }
It follows from the definition of a valid estimate that $(i,j) \in
E(t)$. \aon{Using the definition of $E(t)$, we conclude that}
there is no restart message sent from $j$ to $i$ at time $t$.
{\aoe By the induction hypothesis,}
during the slot $t-1$, $j$ has thus executed O3 or O4b, so that
$M_j(t-1) = M_j(t)$; {\aoe in fact, Lemma \ref{lem:edge->mi=mj}
gives that $M_j(t)=M_i(t)=M^*$}. Thus during slot $t$, $i$ reads a
message from $j=P_i(t)$ with the estimate $M^*$, {\aoe and
executes O4b, consequently leaving its $M_i$ or $P_i$ unchanged.}

{\aoe We finally argue that node $i$'s estimate remains valid.
This is the case because since we can apply the arguments of the
previous paragraph to every node $j$ on the path from $i$ to a node
with out-degree $0$; we obtain that all of these nodes both (i)
keep $P_j(t+1)=P_j(t)$ and (ii) execute O3 or O4b, and consequently
do not send out any restart messages. }
\end{proof}

\smallskip

\aoc{Recall (see the comments following Lemma \ref{lem:acyclic})
that $G(t)$ consists of a  collection of disjoint in-trees (trees
in which all edges are oriented  towards a root node).
Furthermore, by Lemma \ref{lem:edge->mi=mj}, the value of $M_i(t)$
is constant on each  of these trees. Finally, all nodes on a
particular tree have either a valid or invalid estimate (the
estimate being valid if and only if $P_r(t)=r$ and
$M_r(t)=u_r(t-1)$ at the root node $r$ of the tree.) For any $z\in
{\cal U}$, we let $G_z(t)$ be the subgraph of $G(t)$ consisting of
those trees at which all nodes have $M_i(t)=z$ and for which the
estimate $z$ on that tree  is invalid. We refer to $G_z(t)$ as the
\emph{invalidity graph} of $z$  at time $t$. In the sequel we will
say that $i$ is in $G_z(t)$, and abuse notation  by writing $i\in
G_z(t)$, to mean that $i$ belongs to the set of nodes  of
$G_z(t)$. The lemmas that follow aim at showing} that the
invalidity graph of the largest estimate eventually becomes empty.
\jred{Loosely speaking,} \jmj{the first \aoc{lemma}} asserts that
\jred{after the $u_i(t)$ have stopped changing,} it takes
essentially two time steps for a maximal estimate $M^*$ to
propagate to a neighboring node.

\begin{lemma} \label{lemma:addition}
\mod{\jnt{Fix some time $t>T'$.} Let $M^*$ be the largest estimate
at \jnt{that time,}  i.e., $M^* = \max_{\jnt{i}} M_i(t)$.
Suppose that $i$ is in $G_{M^*}(t+2)$ but not in $G_{M^*}(t)$. Then
$P_i(t+2)\in G_{M^*}(t)$.}
\end{lemma}
\begin{proof}
The fact that $i\not\in G_{M^*}(t)$ implies that either \aoc{(i)} $M_i(t) \neq
M^*$ or \aoc{(ii)} $M_i(t) = M^*$ and $i$ has a valid estimate at time $t$. In
the latter case, it follows from Lemma
\ref{lem:valid+max-->final} that $i$ also has a valid estimate at
time $t+2$, contradicting the assumption $i\in G_{M^*}(t+2)$. Therefore, \aoc{we can and
will assume that} $M_i(t) <
M^*$. Since $t>T'$, no node ever executes O1. The difference between
$M_i(t)$ and $M_i(t+2)= M^*$ can only result from the execution of
O2 or O4a by $i$ during time slot $t$ or $t+1$.

Node $i$ cannot have executed O4a during slot $t+1$, \aoc{because this would
result in} $P_i(t+2) = i$, and $i$ would have a
valid estimate at time $t+2$, contradicting \aoc{the assumption} $i\in G_M^*(t+2)$.
Similarly, if $i$ executes O4a during slot $t$ it sets $P_i(t+1)
= i$. Unless it executes O2 \aoc{during slot} $t+1$, we have \aoc{again} $P_i(t+2) = i$ contradicting \aoc{
the assumption} $i\in
G_M^*(t+2)$. Therefore, $i$ must have executed O2 during \aoc{either} slot
$t+1$ or slot $t$, and in the latter case it must not
have executed O4a during slot $t+1$.

Let us suppose that $i$ executes O2 during slot $t+1$, and
sets thus $P_i(t+2) =j$ for some $j$ that sent at time $t+1$ a
message with the estimate $M^*=M_i(t+2)$. The rules of the
algorithm imply that {\aoe $M^* = M_j(t+1)$. We can also conclude
that $M_j(t+1)=M_j(t)$, since if this \aoc{were} not true, node $j$ would
have sent out a restart at time $t+1$. Thus $M_j(t) = M^*$. It }
remains to prove that the estimate $M^*$ of $j$ at time $t$ is not
valid. Suppose, to obtain a contradiction, that it is valid. Then it
follows from Lemma \ref{lem:valid+max-->final} that $j$ also has a
valid estimate at time $t+2$, and from the definition of validity
that the estimate of $i$ is also valid at $t+2$, in contradiction
with \aoc{the assumption} $i\in G_{M^*}(t+2)$. {\aoe Thus we have \aoc{established} that }
$P_i(t+2) = j \in G_{M^*}(t)$ if $i$ executes O2 during slot
$t+1$. The same argument applies if $i$ executes O2 during slot
$t$, without executing O4a or O2 during the slot $t+1$, \aoc{using the fact that
in this case} $P_i(t+2) = P_i(t+1)$.
\end{proof}

\jblue{\jred{Loosely speaking,} the \aoc{next} lemma asserts that the
removal of an invalid maximal estimate $M^*$, through the
propagation of restarts, takes place at unit speed.}

\begin{lemma} \label{lemma:roots}
\jblue{Fix some time $t>T'$,} and let $M^*$ be the largest
estimate at \jblue{that time.} Suppose that $i$ is a root
\jnt{(i.e., has zero out-degree)} in the forest $G_{M^*}(t+2)$.
Then, either (i) $i$ is the root of \jnt{an one-element} tree in
\mod{$G_{M^*}(t+2)$} \aoc{consisting} only of $i$, or (ii) $i$ is
\jmj{at least} ``two levels down in $G_{M^*}(t)$, i.e., there
\jnt{exist} nodes $i',i''$ with $(i,i'), (i',i'') \in G_{M^*}(t)$.
\end{lemma}
\begin{proof}
Consider such a node $i$ and assume that (i) does not hold. Then,
\begin{eqnarray} M_i(t) & = & M_i(t+1) = M_i(t+2) = M^* \nonumber \\
P_i(t) & = & P_i(t+1) = P_i(t+2) \label{equality} \end{eqnarray}
\aoc{This is because} otherwise, cases (a) and (b) of Lemma
\ref{lem:speed1/2} imply that $i$ has \aoc{zero in-degree} in
$G_{M^*}(t+2)\subseteq G(t+2)$, in addition to having a zero
out-degree, contradicting our assumption that (i) does not hold.
Moreover, the estimate of $i$ is not valid at $t$, because it
would then also be valid at $t+2$ by Lemma
\ref{lem:valid+max-->final}, in contradiction with $i\in
G_{M^*}(t+2)$. Therefore, $i$ belongs to the forest $G_{M^*}(t)$.
Let $r$ be the root of the connected component to which \aoc{$i$}
belongs. We \aoc{will} prove that $i\neq r$ and $P_i(t) \neq r$,
and thus that (ii) holds.

Since $r\in G_{M^*}(t)$, \aoc{we have} $M_r(t) = M^*$ and
\aoc{thus} $r$ does thus not execute O2 during slot $t$. Moreover,
$r$ is a root and has an invalid estimate, so $P_r(t) \neq r$ and
there is a ``restart" message from $P_r(t)$ to $r$ at time $t$.
Therefore, $r$ executes O4a during slot $t$, setting $P_r(t+1) =r$
and sending ``restart" messages to all its neighbors at time
$t+1$. This implies that $i\neq r$, as we have seen that $P_i(t) =
P_i(t+1)$. Let us now assume, to obtain a contradiction, that
$P_i(t) =r$ and thus \aoc{by Eq. (\ref{equality})}, $P_i(t+2) =
P_i(t+1) =r$. In that case, we have just seen that there is at
time $t+1$ a ``restart" message from $r=P_i(t+1)\neq i$ to $i$, so
$i$ executes O4a during slot $t+1$ and sets $P_i(t+2) =i$. This
however contradicts the fact $P_i(t+2) = P_i(t+1)$. Therefore,
{\aoe $r \neq i$ and $r \neq P_i(t)$,} i.e.,  $i$ is ``at least
two levels down" in $G_{M^*}(t)$.
\end{proof}

\ao{Let the {\em depth} of a tree be the largest distance between a leaf
of the tree and the root; the depth of a forest is the largest depth
of any tree in the forest. We will use $g(\cdot)$ to denote depth. }
\jblue{The following lemma uses the previous two lemmas to
assert that a {\aoe forest }carrying an invalid maximal estimate
has its depth decrease by at least one over a time interval
of length two.}

\begin{lemma} \jblue{Fix some
time $t > T'$,} and let $M^*$ be the largest estimate value at
\jblue{that time.} If $g(G_{M^*}(t+2))
> 0$, then \aoc{$g( G_{M^*}(t+2)) \leq g(G_{M^*}(t))-1.$}
\label{lemma:depth_decrease}
\end{lemma}

\begin{proof}
\jblue{Suppose that $g(G_{M^*}(t+2)) > 0$. Let us fix a leaf $i$
and a root $j$ in the forest $G_{M^*}(t+2)$ such that the length
of the path from $i$ to $j$ is equal to the depth of
$G_{M^*}(t+2)$.} \mod{Let $i'$ be the \jblue{{\aoe \aoc{single
neighbor}} of node} $i$ in $G_{M^*}(t+2)$. \jblue{We \jmj{first}
claim} that every edge $(k,k')$ on the path from $i'$ to $j$ in
$G_{M^*}(t+2)$ was also present in $G_{M^*}(t)$.} { \aoe Indeed,
by Lemma \ref{lem:speed1/2}, the appearance of a new edge $(k,k')$
at time $t+1$ or $t+2$ implies that node $k$ has in-degree $0$ in
$G(t+2)$, which contradicts $k$ being an intermediate node on the
path from $i$ to $j$ in $G_{M^*}(t+2)$. The same argument
establishes that $M_k(t)=M_k(t+1)=M^*$. Finally, the estimate of
$k$ at time $t$ is \aoc{invalid}, for if it were \aoc{valid}, it
would still be \aoc{valid} at time $t+2$ by Lemma
\ref{lem:valid+max-->final}, so $i$ \aoc{would also} have a
\aoc{valid} estimate at time $t+2$, which is false by assumption.
Thus we have just established \aoc{that both} the node $k$ and its
edge $(k,k')$ at time $t+2$ belong to $G_{M^*}(t)$. }

\aoc{Thus the graph $G_{M^*}(t)$ includes a path from $i'$ to $j$ of length
$g(G_{M^*}(t+2))-1$.} Moreover, by Lemma \ref{lemma:roots}, we know
that at time $t$ \jblue{some} edges $(j,j')$ and $(j',j'')$ were
present in $G_{M^*}(t)$, so the \jred{path length} from $i'$ to
\jmj{$j''$ is at least $g( G_{M^*}(t+2)) + 1$}. This proves \aoc{
that $g( G_{M^*}(t)) \geq g(G_{M^*}(t+2))+1$ and the
lemma.}
\end{proof}

The following lemma {\aoe analyzes the remaining case of invalidity
graphs with zero depth.} \aoc{It shows that the invalidity graph will be empty
two steps after its depth reaches zero.}

\begin{lemma}\label{lem:invalid_length0}
Fix some time $t > T'$, and let $M^*$ be the largest estimate value
at \jblue{that time.}  If $G_{M^*}(t+2)$ is not empty, then $g(
G_{M^*}(t+1))
>0 $ or $g( G_{M^*}(t))  >0 $.
\end{lemma}

\begin{proof}
Let us take a node $i\in G_{M^*}(t+2)$ and let $j = P_i(t+2)$. It follows from the
definition of a valid estimate that $j \neq i$. This
implies that $i$ did not execute O4a (or O1) during slot $t+1$.
We treat two cases separately:

(i) \emph{Node $i$ did not execute O2 during slot $t+1$.}
In this case, $P_i(t+1) = P_i(t+2) = j$ and $M_j(t+1) = M_j(t+2) =
M^*$. Besides, there is no ``restart" message from $j=P_i(t+1)$ to
$i$ at time $t+1$, for otherwise $i$ would have executed O4a
during slot $t+1$, which we know it did not. Therefore,  $(i,j)\in
E(t {\aoe + 1})$ by definition of $G(t)$, and $M_j(t+1) = M_i(t+1)
= M^*$ by Lemma \ref{lem:edge->mi=mj}. Moreover, neither $i$ nor
$j$ have a valid estimate, for otherwise Lemma
\ref{lem:valid+max-->final} would imply that they both hold the
same valid estimate at $t+2$, in contradiction with $i\in
G_{M^*}(t+2)$. So the edge $(i,j)$ is present in $G_{M^*}(t+1)$,
which has thus a positive minimal depth.\\
(ii) \emph{Node $i$ did execute O2 during slot $t+1$:}
In that case, there was a message \aoc{with} the value $M_i(t+2) = M^*$
from $j$ to $i$ at time $t+1$, which implies that $M_j(t+1) =
M_j(t+2) = M^*$. This implies that $j$ did not execute operation O2
during slot $t$. Moreover, node $j$ did not have a valid
estimate at time $t+1$. Otherwise, {\aoe part (a) }of Lemma
\ref{lem:valid+max-->final} implies that \aoc{$j$ has} a valid
estimate at time $t+2$, {\aoe and part (b) of the same lemma
implies} there \aoc{was not} a ``restart" message {\aoe from} $j$ at
$t+2$, so that $(i,j)\in E(t+2)$. This would in turn imply that $i$
has a valid estimate at time $t+2$, contradicting $i\in
G_{M^*}(t+2)$. {\aoe To summarize,} $j$ has an invalid estimate
$M^*$ at time $t+1$ and did not execute O2 during slot $t$.
{\aoe We now simply observe that the argument of case (i) applies}
to $j$ at time $t+1$.
\end{proof}

\jmj{The next lemma asserts that the largest invalid estimates are
eventually purged, and thus that eventually, all remaining largest
estimates are valid.}

\begin{lemma}\label{prop:no_unreliable}
Fix some time $t > T'$, and let $M^*$ be the largest estimate value
at \jblue{that time.} Eventually $G_{M^*}(t)$ is empty.
\end{lemma}
\begin{proof}
{\aoe Lemma \ref{lemma:depth_decrease} implies there is a time
$t'>T'$ after which $g(G_{M^*}(t'')) =0$ for all $t''>t'$. Lemma
\ref{lem:invalid_length0} then implies} that $G_{M^*}(t)$ is empty
for all $t>t'+2$.
\end{proof}

\jnt{We are now ready for the {\aoe proof of the main theorem}.}

\begin{proof}[Proof of Theorem \ref{thm:maxtracking}.]
\jblue{Let $\oM=\max_i u_i(T')$.} \jmj{It follows from the
definition of a valid estimate that any node holding an estimate
$M_i(t)> \oM$ \aoc{at time $t \geq T'$} has an invalid estimate.} Applying Lemma
\ref{prop:no_unreliable} repeatedly shows the existence of some time
$\oT\jblue{\geq T'}$ such that when $t \geq \oT$, \jmj{no node has
an estimate larger than $\oM$, and every node having an estimate
$\oM$ has a valid estimate.}

{\aoe We will assume that the time $t$ in
every statement we make below satisfies $t \geq \overline{T}$.  Define} $Z(t)$ as the set of nodes
having the estimate $\oM$ at time $t$. Every node in $Z(t)$ holds a
valid estimate, and $Z(t)$ is never empty because Lemma
\ref{lem:m_geq_y} implies that $M_i(t)\geq \oM$ for every $i$ with
$u_i = \oM$. Moreover, it follows from Lemma
\ref{lem:valid+max-->final} and the definition of validity  that any
node belonging to some $Z(t)$ will forever afterwards \aoc{maintain $M_i(t) = \oM$ and will} satisfy
 the conclusion of  Theorem
\ref{thm:maxtracking}.

{\aoe We conclude the proof by arguing that eventually every node is
in $Z(t)$. In particular, we will argue that}  a node $i$ adjacent
to a node $j\in Z(t)$ necessarily belongs to $Z(t+2)$. {\aoe Indeed,
} it follows from Lemma \ref{lem:valid+max-->final} that node $j$
sends to $i$ a message \aoc{with} the estimate $\oM$ at time $t+1$. If
$i\in Z(t+1)$, then $i \in Z(t+2)$; else $M_i(t+1)< \oM$, $i$
executes O2 during slot $t+1$, and sets $M_i(t+2) = \oM$, so indeed
$i\in Z(t+2)$. \end{proof}

\section{\jnt{Interval-averaging} \label{sec:comput_average}}

\jnt{In this section, we present an interval-averaging algorithm
and prove its correctness. We start by repeating the informal discussion of
its main idea from the beginning of the chapter.} Imagine the integer \jnt{input} value $x_i$
\green{as represented by a number of $x_i$} pebbles \jblue{at node
$i$.} \jnt{The} algorithm attempts to exchange pebbles between
nodes with unequal numbers so that the overall distribution
\green{becomes} more even. Eventually, either all nodes will have
the same number of pebbles, or some will have a certain number and
others just one more. \red{We let $u_i(t)$ be the current number
of pebbles at node $i$; in particular, $u_i(0)=x_i$. An important
property of the algorithm \ora{will be} that the total number of
pebbles is conserved.}

To match nodes with unequal number of pebbles, we use the maximum
tracking algorithm of Section \ref{sec:max_tracking}. Recall that
the algorithm provides nodes with pointers which attempt to track
the \green{location of the} maximal values. When a node with
$\red{u}_i$ pebbles comes to believe in this way that a node with
at least $u_i+2$ pebbles exists, it sends a request in the
direction \green{of the latter node} to obtain one or more
pebbles. This request follows \blue{a} path to \blue{a} node with
\blue{a} maximal number of pebbles until \blue{the request}
either gets denied, or gets accepted by a node with at least
$u_i+2$ pebbles.

\subsection{\jnt{The algorithm} \label{se:alg}}
\jnt{The} algorithm uses two types of messages. \jblue{Each type
of message}  \jnt{can be either \emph{originated} at a node or
\emph{forwarded} by a node.}

\smallskip
(a) \green{(Request, $r$): This is a request for a
transfer of \jnt{pebbles.} Here, $r$ is an integer that} represents the
number of pebbles \ora{$u_i(t)$} at the
\jnt{node \ora{$i$} that first originated the request,}
\jblue{at the time \ora{$t$} that the request was originated.}
\ora{(Note, however, that this request is actually sent at time $t+1$.)}
\smallskip

(b)
\green{(Accept, $w$): This corresponds to} acceptance of
a request, and \jblue{a} transfer of  $w$ pebbles
\jnt{towards the node that originated the request.}
\jnt{An acceptance} with a value $w=0$ represents a request denial.
\smallskip

\blue{As} part of the algorithm, the nodes run the maximum
tracking algorithm of Section \ref{sec:max_tracking}, as well as a
minimum tracking counterpart. In particular, each node $i$
has access to the variables $M_i(t)$ and $P_i(t)$ of the maximum
tracking algorithm (recall that these are, respectively, the
\green{estimated maximum and a} pointer to a neighbor \aoc{or to itself}).
\green{Furthermore,} each node maintains three additional
variables.

\smallskip
(a)
``\jnt{M}ode($t$)"$\in$ \{\jnt{F}ree,\ \jnt{B}locked\}. Initially, the mode of every
node is free.
\jnt{A node is blocked if it has originated or forwarded a request,
\jblue{and is still waiting to hear whether the request is} accepted (or denied).}

\smallskip
(b) ``$\Rin_i(t)$'' and ``$\Rout_i(t)$" are pointers to
a neighbor of $i$, or to \jnt{$i$} itself. \jnt{The \ora{meaning}
of these pointers \ora{when in blocked mode} are as follows. If
$\Rout_i(t)=j$, then node $i$ has sent (either originated or
forwarded) a request to node $j$, \jblue{and is still in blocked
mode, waiting to hear whether the request} \ora{is accepted or
denied.} If $\Rin_i(t)=k$, and $k\neq i$, then node $i$ has
received a request from node $k$ \jblue{but has not yet responded
to node $k$.} If $\Rin_i(t)=i$, then node $i$ has originated a
request \jblue{and is still in blocked mode, waiting to hear
whether the request} is accepted or denied.}

\smallskip

\jnt{A precise description of the algorithm is given in Figure
\ref{fig:descr_avg_algo}. The proof of  correctness is given in the
next subsection, thus also establishing Theorem
\ref{thm:average_computation}. Furthermore,  we will show that the
time until the algorithm settles on the correct output is of order
$O(n^2K^2 \log K)$.}

\subsection{\jnt{Proof of correctness}}

\old{\jnt{To facilitate understanding, let us start with some observations.
A node alternates between the Free and Blocked modes. Whenever
it switches from Blocked to Free, $\Rin_i$ and $\Rout_i$ are set
to $\emptyset$; and whenever it switches from Free to Blocked,
$\Rin_i$ and $\Rout_i$ are set to values other than $\emptyset$.
Thus, Mode$_i(t)=$Free if and only if $\Rout_i(t)=\Rin_i(t)=\emptyset$.}}

We begin by arguing that the rules of the algorithm
preclude one potential obstacle; \ora{we will show that} nodes will not get stuck sending
requests to themselves.

\begin{lemma}\label{l:no-self} \jnt{A node never sends (originates or forwards) a request to itself. More precisely, $\Rout_i(t)\neq i$, for all $i$ and $t$.}
\end{lemma}
\begin{proof} \jnt{By inspecting} the first two \jnt{cases for the} free mode, we observe that \jnt{if} node $i$
\jnt{originates} a request \ora{during time slot $t$ (and sends a
request message at time $t+1$),} \aoe{ then $P_i(t) \neq i$. Indeed,
to send a message, it must be true $M_i(t) > u_i(t) = u_i(t-1)$.
However, any action of the maximum tracking algorithm that sets
$P_i(t)=i$ also sets $M_i(t)=u_i(t-1)$, and moreover, as long as
$P_i$ doesn't change neither does $M_i$.  So }
\jnt{the recipient $P_i(t)$ of the request originated by $i$ is
different than $i$, \ora{and accordingly,} $\Rout_i\jblue{(t+1)}$ is
set to a value different than $i$. \ora{We argue} that the same is
true for the case where $\Rout_i$ is set by the the ``Forward
request'' box of the free mode. Indeed, that box is enabled  only
when $u_i(t)=u_i(t-1)$ and $u_i(t)-1\leq r <M_i(t)-1$, so that
$u_i(t-1)<M_i(t)$. \ora{As in the previous case, this implies that}
$P_i(t)\neq i$ and that $\Rout_i(t+1)$ is again set to a value other
than $i$. We conclude that $\Rout_i(t)\neq i$ for all $i$ and $t$.}
\end{proof}

\ora{We will now analyze the evolution of the requests.
A request is originated at some time $\tau$ by some originator node $\ell$ who sets $\Rin_{\ell}(\tau+1)=\ell$ and sends the request to some node $i=\Rout_{\ell}(\tau+1)=P_{\ell}(\tau)$.
The recipient $i$ of the request either accepts/denies it, in which case $\Rin_i$ remains unchanged, or forwards it while also setting $\Rin_i(\tau+2)$ to $\ell$. The process then continues similarly. The end result is that at any given time $t$, a request initiated by node $\ell$ has resulted in a ``request path of node $\ell$ at time $t$,'' which
is a maximal sequence of nodes $\ell,i_1,\ldots,i_k$ with
$\Rin_{\ell}(t)=\ell$, $\Rin_{i_1}(t)=\ell$, and $\Rin_{i_m}(t)=i_{m-1}$ for
$m\leq k$.

\newpage
\begin{figure}
\centering

\epsfig{file=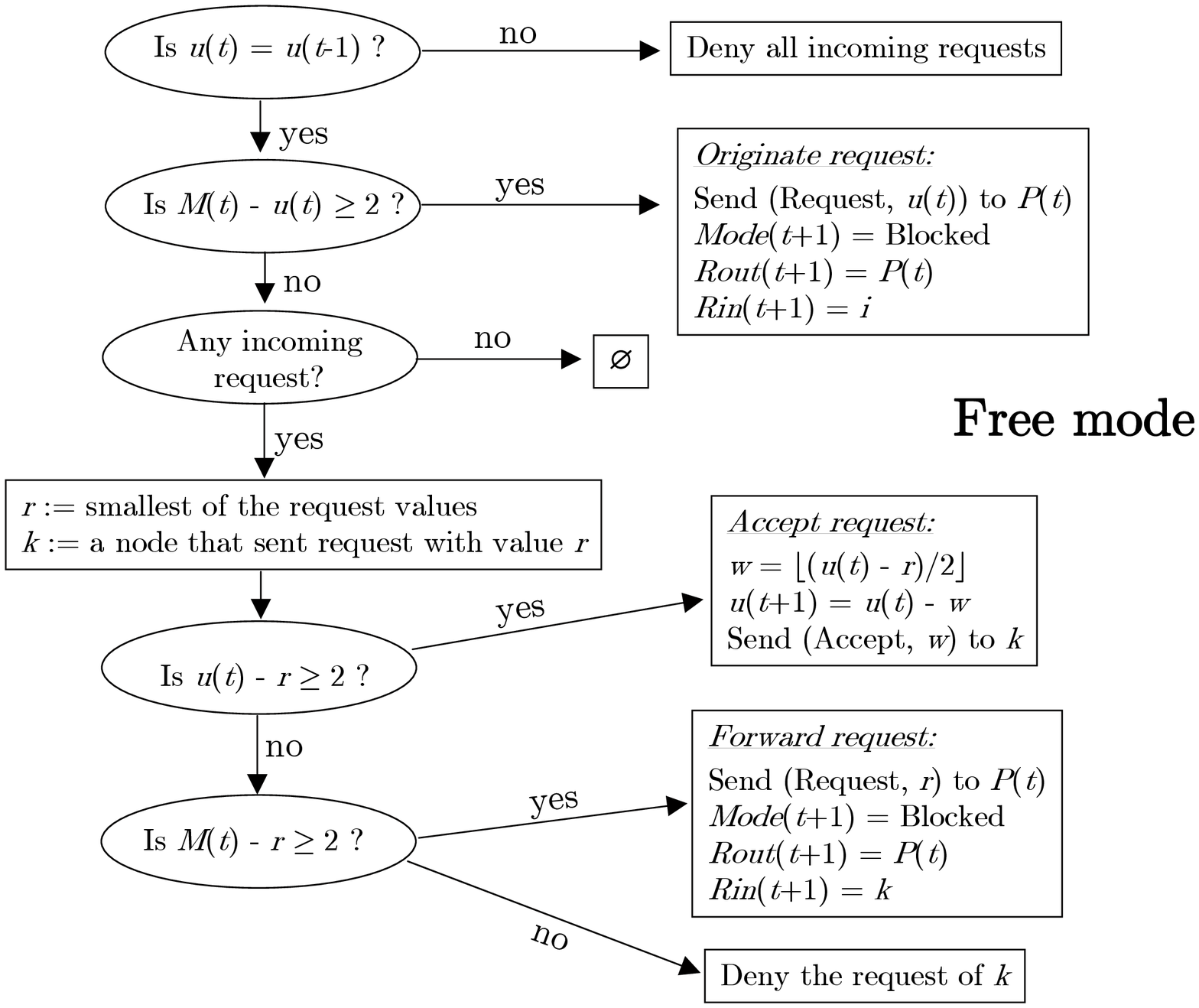, scale=0.5}


\epsfig{file=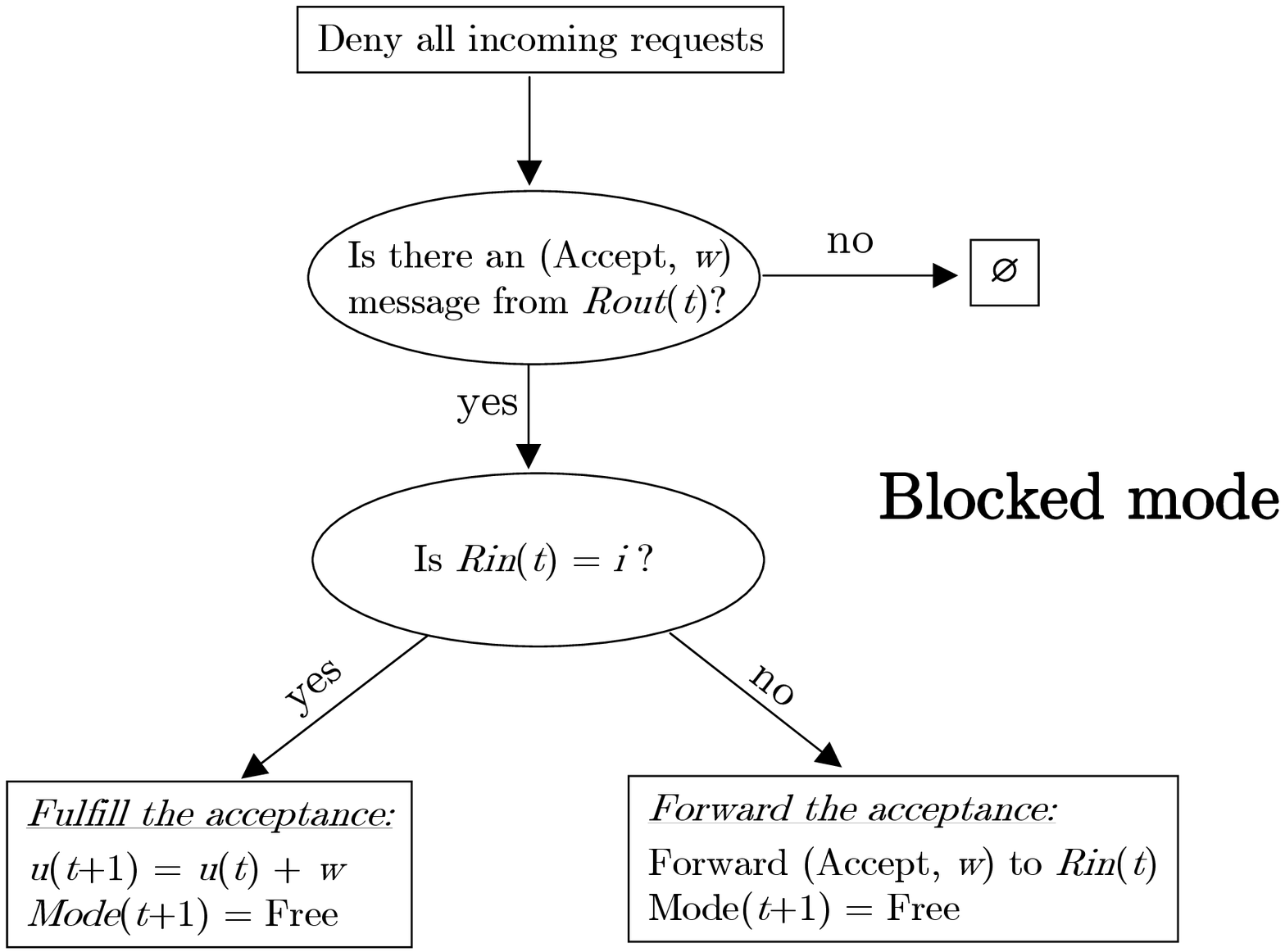, scale=0.55}
\caption{\jnt{Flowchart of the procedure used by node $i$ during slot $t$ in the interval-averaging algorithm.
The subscript $i$ is omitted from variables such as Mode$(t)$, $M(t)$, etc.
Variables for which an update is not explicitly indicated are assumed to remain unchanged.}
\jblue{``Denying a request'' is a shorthand for
$i$ sending a message of the form (Accept, 0) at time $t+1$ to a node from which $i$ received a request at time $t$. Note also that \ora{``forward the acceptance''} in the blocked mode includes the case where
the answer had $w=0$ (i.e., \ora{it was} a request denial), in which case the denial is forwarded.}
}\label{fig:descr_avg_algo}
\end{figure}
\newpage

\begin{lemma}\label{lem:descr_Gr}
At any given time, different request paths cannot intersect (they involve disjoint sets of nodes). Furthermore, at any given time, a request path cannot visit the same node more than once.
\end{lemma}
\begin{proof}
For any time $t$, we form a graph that consists of all edges that lie on some request path.
Once a node $i$ is added to some request path, and as long as that request path includes $i$, node $i$ remains in blocked mode and the value of $\Rin_i$ cannot change. This means that adding a new edge  that points into $i$ is impossible. This readily implies that cycles cannot be formed and also that two request paths cannot involve a
common node.
\end{proof}
}

\ora{We use $p_\ell(t)$ to denote the request path of node $\ell$ at time $t$, and
$s_\ell(t)$ to denote the last node
on this path.}
We will say that a request \jnt{originated} by node $\ell$ \emph{terminates} when
node $\ell$ receives an \jnt{(Accept, $w$)} message, with \jnt{any} value
$w$.

\begin{lemma}\label{lem:request_terminate}
Every request eventually terminates. Specifically, if node
\ora{$\ell$} \jnt{originates} a request at time $t'$ \ora{(and sends a request message at time $t'+1$),} then \jnt{there exists a}  later time
$t''\mod{\leq t'+n}$ \jnt{at which} node $s_r(t'')$ receives an ``accept
request'' message \jnt{(perhaps with $w=0$),} which is forwarded until it reaches $\ell$,
\mod{no later than \ora{time} $t''+n$}. 
\end{lemma}
\begin{proof} By the rules of our algorithm, node $\ell$ sends a request
message to node $P_{\ell}(t')$ at time \ora{$t'+1$.} If node $P_\ell(t')$ replies
at time \ora{$t'+2$}
with a ``deny request'' \ora{response}  to $\ell$'s request, then
\ora{the claim is true; otherwise,} observe that $p_\ell(t'+\ora{2})$ is nonempty and until
$s_\ell(t)$ receives an ``accept request'' message, the length of
$p_\ell(t)$ increases \jnt{at each time step.} Since
this length cannot be larger than $n\ora{-1}$, by Lemma \ref{lem:descr_Gr},
it follows \jnt{that} $s_\ell(t)$ receives an ``accept request'' message at most
$n$ steps after $\ell$ initiated the \ora{request.}
\jnt{One can then} easily show that this \jnt{acceptance} message is forwarded \jnt{backwards along the path} \ora{(and the request path keeps shrinking)}
until the acceptance message reaches $\ell$, at most $n$ steps later.
\end{proof}

\ora{The arguments so far had mostly to do with deadlock avoidance.
The next  lemma concerns the progress made by the algorithm. Recall
that a central idea of the algorithm is to conserve the total number
of ``pebbles,'' but this must include both pebbles possessed by
nodes and pebbles in transit. We capture the idea of ``pebbles in
transit'' by defining a new variable. If $i$ is the originator  of
some request path that is present at time $t$, and if the final node
$s_i(t)$ of that path} receives an \jnt{(Accept, $w$)} message at
time $t$, we let $w_i(t)$ be \jnt{the value $w$ in that message.
(This convention includes the special case where $w=0$,
corresponding to a denial of the request).} \ora{In all other cases,
we} set $w_i(t) =0$. Intuitively, $w_i(t)$ is the value that has
already been given away by a node who answered a request originated
by node $i$, and that will eventually be added to \ora{$u_i$}, once
the answer reaches \ora{$i$}.

We now define
$$\hat u_i(t)= u_i(t) + w_i(t).$$ By the rules of our algorithm,
if $w_i(t)\ora{=w}>0$, \ora{an amount $w$} will eventually be added to $u_i$, once
the \ora{acceptance message is forwarded back to $i$.}  The value $\hat u_i$ can thus be
seen as a future value of $u_i$, that includes its present value
and the \ora{value that has} been sent to $i$ but has not yet reached it.

\ora{The rules of our algorithm imply
that the sum of the $\hat u_i$ remains constant. Let $\bar x$ be the average of the initial values $x_i$. Then,
\[ \frac{1}{n} \sum_{i=1}^n \hat{u}_i(t) =
\frac{1}{n} \sum_{i=1}^n x_i =\bar{x}. \] We define the \emph{variance} function $V$ as
\[ V(t) = \sum_{i=1}^n (\hat{u}_i(t) - \bar{x})^2.\]}

\begin{lemma}\label{prop:finite_request_acceptances}
The number of times \ora{that} a node \jnt{can send an \ora{acceptance} message (Accept, $w$) with $w\neq 0$,}
is finite.
\end{lemma}
\begin{proof} Let us first describe the idea behind the proof. Suppose that nodes
could instantaneously transfer value to each other. \ora{It is
easily checked that if} a node $i$ \ora{transfers} an amount $w^*$
to a node $j$ with $u_i-u_j\geq 2$ and $1\leq w^*\leq
\frac{1}{2}(u_i-u_j)$, the  variance $\sum_i (u_i- \ora{\bar x})^2$
decreases by at least \aoe{2}. \ora{Thus, there can only be} a
finite number of \ora{such transfers.} In our model, \ora{the
situation is more complicated because transfers}  are not immediate
\ora{and involve a process of requests and acceptances. A key
element of the argument is to realize that the algorithm can be
interpreted as if it only involved instantaneous exchanges involving
disjoint pairs of nodes.}

\ora{Let us consider the difference  $V(t+1)-V(t)$ at some typical time
$t$.
Changes in $V$ are solely due to changes in the $\hat u_i$.
Note that if a node $i$ executes the ``\aoc{fulfill} the acceptance'' instruction at time $t$, node $i$ was the
originator of the request
and the request path has length zero, so that it is also the final node on the path, and $s_i(t)=i$. According to our definition, $w_i(t)$ is the value $w$ in the message received by node $s_i(t)=i$.
At the next time step, we have $w_i(t+1)=0$ but $u_i(t+1)=u_i(t)+w$.
Thus, $\hat u_i$ does not change, and the function $V$ is unaffected.}

\ora{By inspecting the algorithm, we see that  a nonzero  difference
$V(t+1)-V(t)$ is possible only if some node $i$ executes the ``accept request'' instruction at slot $t$, with
some particular value $w^*\aoc{\neq 0}$,
in which case $u_i(t+1)=u_i(t)-w^*$. For this to happen, node $i$ received a message (Request, $r$) at time $t$ from a node $k$ for which $\Rout_k(t)=i$, and with $u_i(t) - r\geq 2$.
That node $k$ was the last node, $s_{\ell}(t)$, on the request path of some originator node $\ell$. Node $k$ receives an (Accept, $w^*$)
message at time $t+1$ and, therefore, according to our definition, this sets $w_\ell(t+1)=w^*$.}

It follows from the rules of our algorithm that
$\ell$ \ora{had} originated a request with value $r=u_{\ell}(t')$ at
some previous time $t'$. Subsequently, node $\ell$ entered the blocked
mode, preventing any modification of $u_{\ell}$, so that
$r=u_{\ell}(t) = u_{\ell}(t+1)$. Moreover, observe that $w_{\ell}(t)$ was 0 because \ora{by time $t$,} no node
had answered \ora{$\ell$'s} request.
\ora{Furthermore,} $w_i(t+1) = w_i(t) =0$ because
having a positive $w_i$ requires \ora{$i$} to be in blocked mode, preventing
the execution of ``accept request". It follows that
$$\hat
u_i(t+1) = u_i(t+1) = u_i(t) - w^* = \hat u_i(t) - w^*,$$ and
$$\hat
u_{\ell}(t+1) = r + w^* = \hat u_{\ell}(t) + w^*.$$
\ora{Using the update equation $w^*=\lfloor(u_i(t)-r)/2 \rfloor$, and the fact $u_i(t) - r\geq 2$,
we obtain} $$1\leq w^* \leq \frac{1}{2}(u_i(t) -
r) = \frac{1}{2}(\hat u_i(t) - \hat u_{\ell}(t)).$$ Combining
with the previous equalities, we have
$$
\hat u_{\ell}(t) +1 \leq  \hat u_{\ell}(t+1) \leq \hat u_i(t+1)
\leq \hat u_i(t) -1.
$$
\ora{Assume for a moment that node $i$ was the only one that
executed the ``accept request'' instruction at time $t$. Then, all
of the variables $\hat u_j$, for $j\neq i,\ell$, remain unchanged.}
Simple algebraic manipulations then show that $V$ decreases by at
least \aoe{2}. \ora{If there was another pair of nodes, say $j$ and
$k$, that were involved in a transfer of value at time $t$, it is
not hard to see that the transfer of value was related to a
different request, involving a separate request path. In particular,
the pairs $\ell, i$ and $j,k$ do not overlap. This implies that the
cumulative effect of multiple transfers on the difference
$V(t+1)-V(t)$ is the sum of the effects of individual transfers.
Thus, at every time for which at least one ``accept request'' step
is executed, $V$ decreases by at least 2. We also see that no
operation can ever result in an increase of $V$. It follows that}
the instruction ``accept request" can be executed only a finite
number of times.
\end{proof}

\begin{proposition}\label{prop:final_proof_average} There is a time $t'$ such that $\rt{u}_i(t) =
\rt{u}_i(t')$,  for \jnt{all $i$ and} all $t \geq t'$. Moreover,
\begin{eqnarray*}
\sum_\jnt{i} \rt{u}_i(t') & = & \sum_\jnt{i} x_i, \\
\jnt{\max_i \rt{u}_i(t')-\min_i \rt{u}_i(t')} & \leq &   1.
\end{eqnarray*}
\end{proposition}

\begin{proof}
\noindent It follows from \jnt{Lemma}
\ref{prop:finite_request_acceptances} that there is a time \aoc{$t'$}after
which no more requests are accepted \jnt{with $w\neq 0$.} By Lemma \ref{lem:request_terminate}, this
implies that after at most $n$ additional time steps, the system
will never \ora{again} contain any  ``accept request'' messages \jnt{with $w\neq 0$,} so no node
will change its value \ora{$u_i(t)$ thereafter.}

We have already argued that the \ora{sum (and therefore the
average) of the variables $\hat u_i(t)$} does not change.
\jnt{Once there are} no more ``accept request'' messages in the
system \jnt{with $w\neq 0$, we must have  $w_i(t)=0$, for all $i$.
Thus,} at this stage the \jnt{average of the $\rt{u}_i(t)$} is the
same as the \jnt{average} of \jnt{the $x_i$.}

It remains to show that once \jnt{the $\rt{u}_i(t)$ stop} changing, the
maximum and minimum \jnt{$u_i(t)$} differ by \jnt{at most} $1$. \jnt{Recall (cf.\ Theorem that
\ref{thm:maxtracking})} \aoc{that} at some time after \jnt{the $\rt{u}_i(t)$
stop} changing, \ora{all
estimates $M_i(t)$ of the maximum will be} equal to $M(t)$, the true maximum of \jnt{the
$\rt{u}_i(t)$;} moreover, \jnt{starting at
any node and} following the pointers $P_i(t)$  \jnt{leads to a node $j$} whose value \jnt{$u_j(t)$} is the true \jnt{maximum,} $M(t)$.
Now let \ora{$A$} be the set of nodes whose value at this stage is at
most $\max_i \rt{u}_i(t)-2$. To \ora{derive} a contradiction, \jnt{let us} suppose
that \ora{$A$} is nonempty.

\jnt{Because} only nodes in \ora{$A$} will \jnt{originate} requests, and \jnt{because} every
request eventually terminates \jnt{(cf.\ Lemma \ref{lem:request_terminate})}, if we wait
some finite amount of time, we will have
the additional property that all requests in the system originated
from \ora{$A$}. Moreover,  nodes in \ora{$A$}  \jnt{originate} requests every time they \aoc{are in} \ora{the
free mode,}
which is infinitely
often.

\ora{Consider now a request originating at a node in the set
\ora{$A$}. The value $r$ of such a request satisfies $M(t)-r\geq 2$,
which implies that every node that receives \aoc{it} either accepts it
(contradicting the fact that no more requests are accepted after
time \aoc{$t'$}), or forwards it, or denies it. But a node $i$ will deny a
request only if it is in blocked mode, that is, if it has already
forwarded some other request to node $P_i(t)$. This shows that
requests will keep propagating along links of the form $(i,P_i(t))$,
and therefore will eventually reach a node at which $u_i(t)=M(t)\geq
r+2$, at which point they will be accepted---a contradiction.}
\end{proof}

We are now ready to conclude.

\begin{proof}[Proof of Theorem \ref{thm:average_computation}]
\jnt{Let $u_i^*$ be the value that $u_i(t)$ eventually settles on.
Proposition \ref{prop:final_proof_average} \ora{readily} implies} that if the average $\bar x$ of
the $x_i$ is an integer, then $\rt{u}_i(t)=\rt{u}_i^*=\bar x$ will
eventually hold for every $i$. If \jnt{$\bar x$} is not an integer, then some
nodes will eventually have $\rt{u}_i(t)=\rt{u}_i^* = \lfloor \bar x
\rfloor$ and \ora{some other nodes} $\rt{u}_i(t) = \rt{u}_i^*=\lceil \bar x
\rceil$. Besides, using the maximum and minimum computation
algorithm, nodes will eventually have a correct estimate of $\max
\rt{u}_i^*$ and $\min \rt{u}_i^*$, since all
$\rt{u}_i(t)$ \jnt{settle on the fixed values $\rt{u}_i^*$.} This allows \jnt{the nodes} to \jnt{determine whether}
the average is exactly $\rt{u}^*_i$ (integer average), or \jnt{whether}  it
lies in $(\rt{u}_i^*,\rt{u}_i^*+1)$ or $( \rt{u}_i^*-1,
\rt{u}^*_i)$ (fractional average). Thus, \jnt{with some simple post-processing at each node (which can be done using finite automata), the nodes can produce the correct output for the interval-averaging problem. The proof of
Theorem \ref{thm:average_computation} is complete.}
\end{proof}

\aoe{Next, we give a convergence time bound for the algorithms we
have just described.}


\smallskip

\aoe{The general idea
is} quite simple. We have just argued that the
nonnegative function $V(t)$ decreases by at least $2$ each time a
request is accepted. \aoc{It also satisfies $V(0)=O(n K^2)$}. Thus there are at
most \aoc{$O(n K^2)$} acceptances. To complete the argument, one needs to
argue that if the algorithm has not terminated, there will be an
acceptance within $O(n)$ iterations (which corresponds to $O(n \log K)$ real-time
rounds of communication due to the $O(\log K)$ slowdown from transmitting elements in
$\{0,\ldots,K\}$). This should be \aoc{fairly clear} from the
proof \aoc{of Theorem \ref{thm:average_computation}. A formal argument is
given in the next section. 
It is also shown there that the running time of our algorithm, for  many graphs,
satisfies a $\Omega(n^2)$ lower bound, in the worst case  over all initial conditions.  }

\section{The Time to Termination}\label{sec:complexity}

The aim of this section is to prove an $O(n^2 K^2 \log K)$ upper bound
on the time to termination of the interval-averaging algorithm
from Section \ref{sec:comput_average}.

We will use the notation $M'(t)$ to denote the largest estimate
\jh{held by any node} at time $t$ or in the $n$ time steps
preceding it: \[ M'(t) = \max_{
\begin{array}{c}
                                    i=1,\ldots,n \\
                                    k=t,t-1,\ldots,t-n \\
                                  \end{array}
} ~~~M_i(k).\] For $M'(t)$ to be well defined, we will adopt the
convention that for all negative times $k$, $M_i(k)=u_i(0)$.

\aoc{\begin{lemma}\label{lem:M'monotonous} In the course of the
execution of the interval-averaging algorithm, $M'(t)$ never
increases.
\end{lemma} }
\begin{proof}
Fix a time $t$. We will argue that \begin{equation} \label{eq:it}
M_i(t+1) \leq M'(t) \end{equation} for each $i$. This clearly
implies $M'(t+1) \leq M'(t)$.

If $M_i(t+1) \leq M_i(t)$, then Eq. (\ref{eq:it}) is obvious. We
can thus suppose that $M_i(t+1) > M_i(t)$. There are only three
boxes in Figure \ref{fig:maxtracking} which result in a change
between $M_i(t)$ and $M_i(t+1)$. These are $O2$, $O1$, and $O4a$.
We can rule out the possibility that node $i$ executes $O2$, since
that merely sets $M_i(t+1)$ to some $M_j(t)$, and thus cannot
result in $M_i(t+1) > M'(t)$.

Consider, then, the possibility that node $i$ executes $O1$ or
$O4a$, and as a consequence $M_i(t+1)=u_i(t)$. If $u_i(t) \leq
u_i(t-1)$, then we are finished because \[ M_i(t+1) = u_i(t) \leq
u_i(t-1) \leq M_i(t), \] which contradicts the assumption
$M_i(t+1)>M_i(t)$. Note that the last step of the above chain of
inequalities used Lemma \ref{lem:m_geq_y}.

Thus we can assume that $u_i(t)>u_i(t-1)$. In this case, $i$ must
have fulfilled acceptance from some node $j$ during slot $t-1$.
\jh{Let $\hat{t}$ the time when node $j$ received the
corresponding request message from node $i$. The rules of our
algorithm imply that $u_i(\hat t) = u_i(t-1)$, and that the
quantity $w$ sent by $j$ to $i$ in response to the request is no
greater than $\frac{1}{2}(u_j(\hat t)-u_i(\hat t))$. This implies
that $u_i(t) = u_i(t-1) +w < u_j(\hat t)$.}

Crucially,  we have that $\hat{t} \in [t-1-n,t-1]$, since at most
$n+1$ time steps pass between the time node $j$ receives the
request message it will accept and the time when node $i$ fulfills
$j$'s acceptance. So \[ M_i(t+1) = u_i(t) < u_j(\hat{t}) \leq
M_j(\hat{t}+1) \leq M'(t).\] We have thus showed that $M_i(t+1)
\leq M'(t)$ in every possible case, \jh{which implies that
$M'(t+1)\leq M'(t)$.}
\end{proof}

\begin{lemma}\label{lem:conv_time_minmax}
\ora{Consider the maximum tracking algorithm.} If each
\aoc{$u_i(t)$} is constant for $\aoc{t \in } [t_0,t_0+4n]$, then
\aoc{at least one of the following two statements is true}:
(a) $M'(t_0+3n) \ora{<} M'(t_0)$.
(b) $M_i(t_0+4n) = \max_j u_j(t_0)$ for every $i$.
\end{lemma}
\begin{proof}
\jh{Suppose first that no node holds an estimate equal to
$M'(t_0)$ at some time between $t_0+2n$ and $t_0+3n$. Then it
follows from the definition of $M'(t)$ and its monotonicity
(\ref{lem:M'monotonous}) that condition (a) holds. Suppose now
that some node holds an estimate equal to $M'(t_0)$ at some time
between $t_0+2n$ and $t_0+3n$. The definition of $M'(t)$ and the
monotonicity of $\max_i M_i(t)$ when all $u_i$ are constant (Lemma
\ref{lem:nondecreasing}) imply that $M'(t_0) = \max_i
M_i(t)$ for all $t\in [t_0,t_0+3n]$. It follows from repeated
application of Lemmas \ref{lemma:depth_decrease} and
\ref{lem:invalid_length0} (similarly to what is done in the proof
of Proposition \ref{prop:no_unreliable}) that every estimate
$M'(t_0)$ at time $t_0 + 2n$ is valid, which by definition implies
the existence of at least one node $i$ with $u_i(t_0) = M'(t_0)$.
Besides, since $M_i(t)\geq u_i(t)$ holds for all $i$ and $t$ by
Lemma \ref{lem:m_geq_y} and since we know that $M'(t_0) = \max_i
M_i(t)$ for all $t\in [t_0,t_0+3n]$, we have $M'(t_0) = \max_i
u_i(t_0)$. As described \ora{in the} proof of Theorem
\ref{thm:maxtracking}, this implies that after at most $2n$ more
time steps, $M_i(t)=M'(t_0)$ holds for every $i$, \aoc{and so (b)
holds.}}
\end{proof}

The next lemma \aoc{upper bounds} the largest time before
\aoc{some} request is accepted or \aoc{some} outdated estimate is
purged from the system. Recall that $\bar x = (\sum_i^nx_i)/n$.

\begin{lemma}\label{lem:bound_interoperation_time}
\ora{Consider the interval-averaging} algorithm described in
Sections \ref{sec:max_tracking} and \ref{sec:comput_average}. For
any $t_0$, at least one of the following is true: (a)
Some node accepts a request at \ora{some slot $t\in
[t_0,t_0+8n-1]$.} (b) We have $M'(t+1) < M'(t)$ for some
$t\in [t_0+1,t_0+3n]$. (c) All $u_i$ remain forever
\aoc{constant} after time $t_0+n$, with $u_i \in \{\floor{\bar x}
, \ceil{\bar x}\}$, and all $M_i(t)$ remain \aoc{forever} constant
after time $t_0 + 5n$, with $M_i = \ceil{\bar x}$.
\end{lemma}
\begin{proof}
Suppose that condition (a) does not hold, i.e., that no node
accepts a request between $t_0$ and $t_0 + 8 n$. \ora{Since an
acceptance message needs to travel through} at most $n-2$
intermediate nodes before reaching the \ora{originator} of the
request (cf.\ Lemma \ref{lem:request_terminate}), we conclude that the system does
not contain any \ora{acceptance messages}  after \ora{time}
$t_0+n$. As a result, no node modifies its \ora{value} $u_i(t)$
between \ora{times} $t_0+n$ and $t_0 + 8 n$. Abusing notation \aoc{slightly}, we
\aoc{will} call these values $u_i$.

It follows from Lemma \ref{lem:conv_time_minmax} that either
condition (b) holds, or that there is a time $\tilde{t} \leq
t_0+5n$ at which $M_i(\tilde{t}) = \max_j u_j$ for every $i$. Some
requests may have been emitted in the interval $[t_0,t_0+5n]$.
\aoc{Since we are assuming that} condition (a) does not hold,
\ora{these requests must have all been} denied. It follows from
Lemma \ref{lem:request_terminate} that none of {these requests} is
present by time $t_0+ 7n$. Moreover, by the rules of our
algorithm, once $M_i$ \ora{becomes equal to} $\max_j u_j$ for
every node $i$, every node with $u_i \leq \max_j u_j -2$ \ora{will
keep} emitting requests. \ora{Using an argument similar to the one
at the end of the proof of}  Proposition
\ref{prop:final_proof_average}, if such requests are sent, at
least one must be accepted within $n$ time steps, that is, no
later than \ora{time} $t_0 +8n$. \aoc{Since by assumption this has
not happened, no such requests could have been sent, implying
that} $u_i \geq \max_j u_j -1$ for every $i$. \aoc{Moreover, this
implies that no request messages/acceptance are ever sent after
time $t_0 + 7n$, so that $u_i$ never change. It is easy to see
that the $M_i$ never change as well, so that condition condition
(c) is satisfied.}
\end{proof}

We can now give an upper bound on the \ora{time until our
algorithm terminates.}

\begin{theorem}
The interval-averaging algorithm described in \aoc{Section
\ref{sec:comput_average}} \ora{terminates after} at most
\aoc{O($n^2K^2 \aoc{\log K})$} time steps.
\end{theorem}
\begin{proof}
Consider the function $V(t) = \sum_{i=1}^n \left(\hat u_i(t) -
\bar u\right)^2$, where $\bar u$ \ora{is the average of the $x_i$,
which is also the average of the $\hat u_i$, and where the $\hat
u_i(t)$ are} as defined before Lemma
\ref{prop:finite_request_acceptances}. Since $\hat u_i(t)\in
\{0,1,\dots,K\}$ for all $i$, one can verify that $V(0) \leq
\frac{1}{4}n K^2$.  Moreover, as explained in the proof of Lemma
\ref{prop:finite_request_acceptances}, $V(t)$ is non-increasing,
and decreases by at least 2 with every request acceptance.
Therefore, \ora{a total of} at most $\frac{1}{8}nK^2$ requests can
be accepted. Furthermore, \aoc{we showed that $M'(t)$ is
non-increasing, and since $M'(t)$ always belongs to
$\{0,1,\dots,K\}$, it can strictly decrease at most $K$ times.} It
follows then from Lemma \ref{lem:bound_interoperation_time} that
condition (c) must hold after at most $\frac{1}{8}nK^2\cdot 8n +
K\cdot 3n \jh{+5}$ time steps.

\ora{Recall that} in parallel with the maximum-tracking and
averaging algorithm, we \ora{also} run a minimum tracking
algorithm. In the previous paragraph, we demonstrated that
condition (c) of Lemma \ref{lem:bound_interoperation_time} holds,
i.e. $u_i$ remain fixed forever, after $n^2 K^2 + K\cdot 3n$ time
steps. A similar argument to Lemma \ref{lem:conv_time_minmax}
implies that the minimum algorithm will reach a fixed point after
an additional \jh{$(3K+4)n$} steps.

Putting it all together, the algorithm reaches a fixed point after
\jh{$n^2 K^2 + (6K+4)\cdot n$} steps. \aoc{Accounting in addition for the $\log K$ slowdown
from transmitting values in $\{0,\ldots,K\}$, we obtain a convergence time of
$O(n^2 K^2 \log K)$.}
\end{proof}

\bigskip

\aoc{We note that there are cases where the running time of
interval averaging is quadratic in $n$. For example,} consider the
network in Figure \ref{fig:lower_complexity_bound}, consisting of
two arbitrary \aoc{connected} graphs $G_1,G_2$ with $n/3$ nodes
each, connected by a \aoc{line} graph of $n/3$ nodes. Suppose that
$K=2$, and that $x_i=0$ if $i\in G_1$, $x_i=2$ if $i\in G_2$, and
$x_i =1$ otherwise. \aoc{The algorithm will have} the nodes of
$G_1$ with $u_i = 0$ send requests to nodes $j$ in $G_2$ with
$u_j=2$, and each successful request will result in the pair of
nodes changing their values, $u_i$ and $u_j$, to 1. The system
will reach its final state after $n/3$ such successful requests.
Observe now that each successful request must cross the line
graph, which takes $n/3$ time steps in each direction. Moreover,
since nodes cannot simultaneously treat multiple requests, once a
request begins crossing the line graph, all other requests are
denied until the response to the first request reaches $G_1$,
which takes at least $2n/3$ time steps. Therefore, in this
example, it takes at least $2n^2/\aoc{9}$ time steps until the
algorithm terminates.


\bc \begin{figure}[h] \bc \label{fig:lower_complexity_bound}
        \epsfig{file=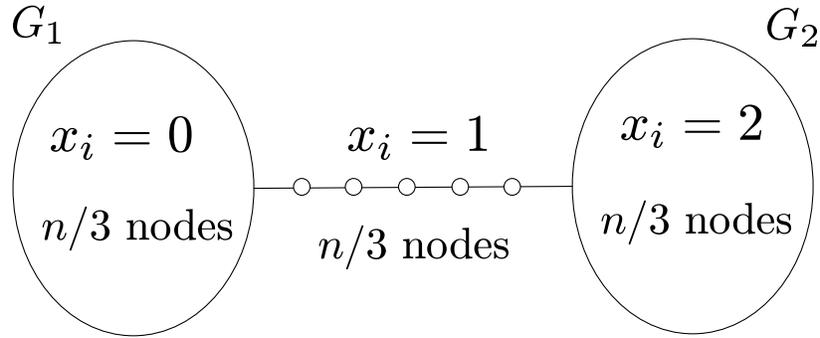,width=11cm} \caption{A class of networks and initial
conditions for which our algorithm takes $\aoc{\Theta}(n^2)$ time steps to
reach its final state. } \ec
  \end{figure} \ec

\section{Simulations \label{se:simul}}
\begin{figure}[h] \hspace{0cm}\centering
 \epsfig{file=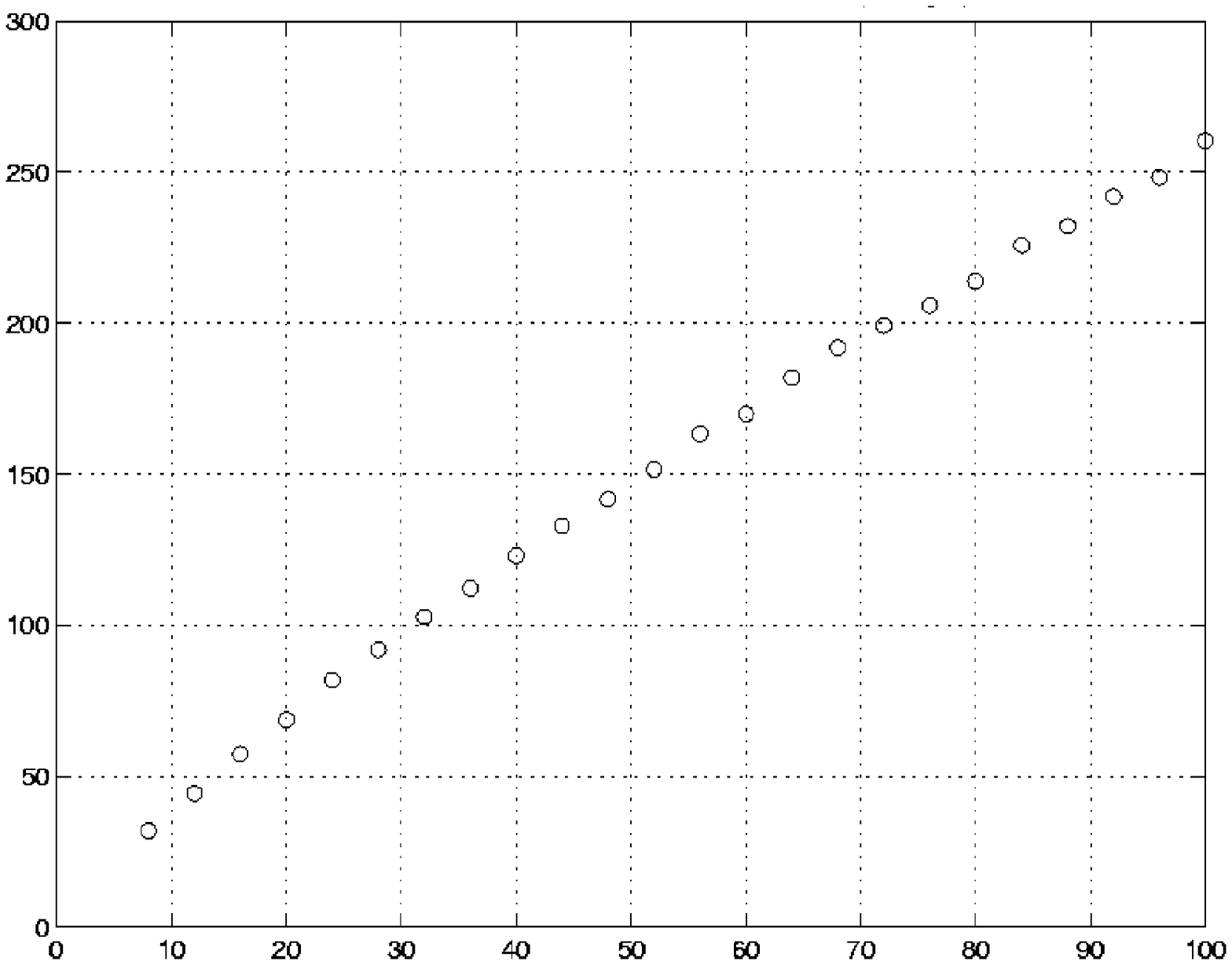, scale=0.72} \caption{\aoc{The number of iterations as a function
of the number of nodes for a complete graph.}} \label{simulc}
\end{figure}

\begin{figure}[h] \hspace{0cm}\centering
 \epsfig{file=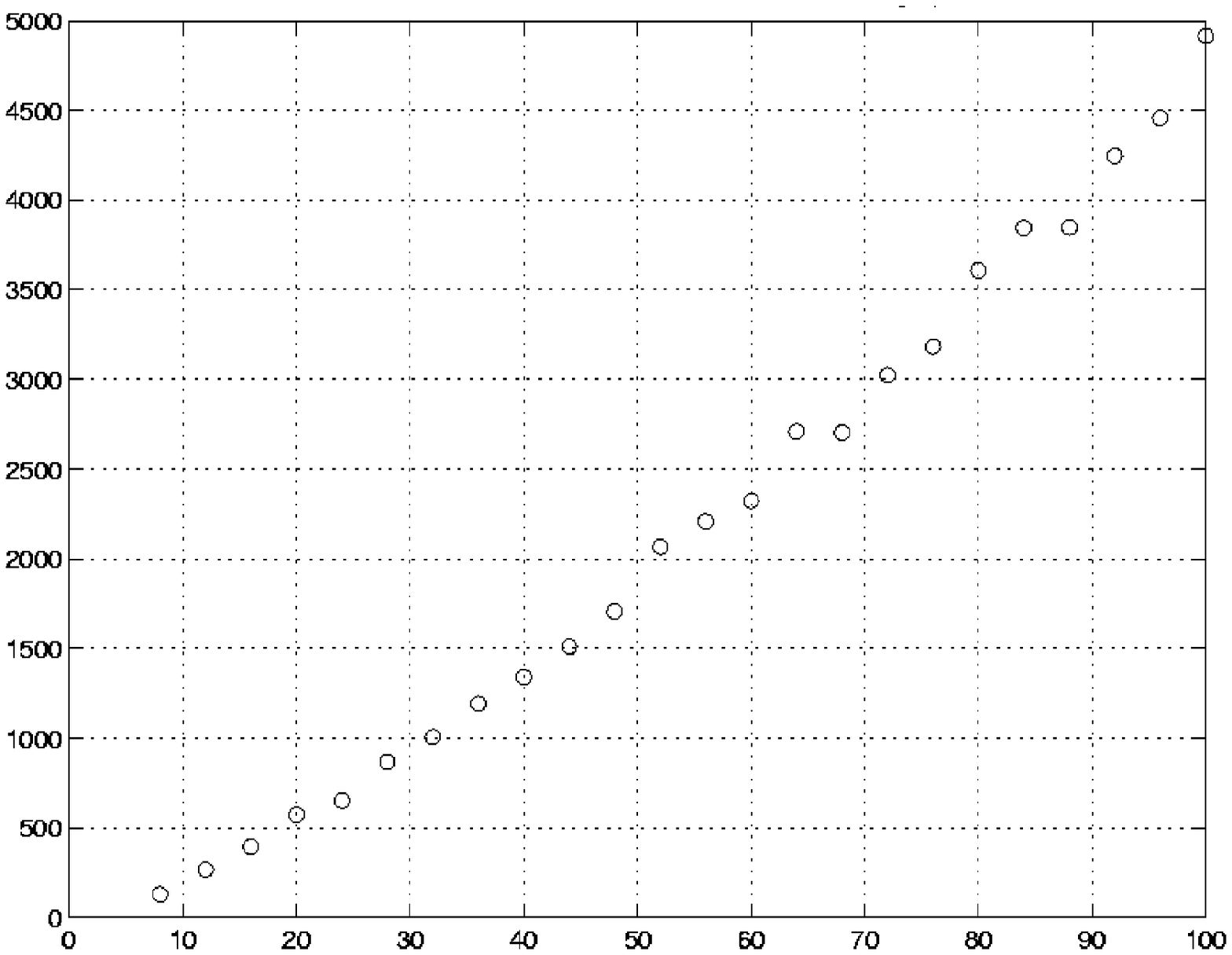, scale=0.6} \caption{\aoc{The number of iterations as a function
of the number of nodes for a line graph.}} \label{simull}
\end{figure}

\aoe{We report here on simulations involving our
algorithm on several natural graphs. \aoc{Figures \ref{simulc} and
\ref{simull} describe the results for a complete graph and a line. Initial conditions were
random integers between $1$ and $30$, and each data point represents
the average of two hundred runs.} As expected, convergence is faster on the complete graph. Moreover,
\aoc{convergence time in both simulations appears to be approximately linear.}

}

\section{Concluding remarks}

\aoc{In this chapter, we have given a deterministic algorithm for averaging which stores a constant
number of bits per each link. Unfortunately, this algorithm works only for static graphs; whether
such an algorithm exists for time-varying graph sequences is still an open question.}

\aoc{As discussed in chapter \ref{why}, low storage is one of the main reasons to pick an averaging 
algorithm over competitors such as flooding. It is therefore interesting to consider just how low
the storage requirements of averaging algorithms are. In the next chapter, we will ask the
analogous question for the problem of computing arbitrary functions.}

\chapter{A framework for distributed function computation \label{ch:function}}

\aoc{We now wish to take a broader perspective and consider general
distributed function computation problems modeled on consensus.}
Indeed, the goal of many multi-agent systems, \green{distributed}
computation \green{algorithms,} and decentralized data fusion
methods is to have a set of \green{nodes compute} \blue{a} common
value based on \green{initial values or observations at each node.
Towards this purpose, the nodes,} which we will \green{sometimes
refer to as agents,} perform some internal computations and
repeatedly communicate with each other. \jnt{The objective of this
chapter is to understand the fundamental limitations and
capabilities of such systems and algorithms when the available
information and computational resources at each node are limited.}

\aoc{Our exposition in this chapter will follow the preprint
\cite{HOT10}, where the results described here have previously
appeared.}

\section{\jnt{Motivation}}\label{se:motiv}

\noindent {\bf (a) Quantized consensus:} Suppose
\green{that} each \jnt{node} begins with an integer value $x_i(0)
\green{\in \{0,\ldots,K\}}$. We would like \green{the \jnt{nodes} to
end up, at some later time, with values $y_i$ that are almost
equal, i.e., $|y_i - y_j| \leq 1$, for all $i,j$, while preserving
the sum of the values, i.e., $\sum_{i=1}^n x_i(0) =
\sum_{\blue{i=1}}^n y_i$.} This is the so-called quantized
\red{averaging} problem, which has received considerable attention
recently; see, \jnt{e.g.,} \cite{KBS07,FCFZ08,ACR07, KM08, BTV08}, \aoc{and we have
talked about it at some length in the previous Chapters \ref{qanalysis} and \ref{ch:constantstorage}.}
It may be viewed as
the problem of computing the function $(1/n) \sum_{i=1}^n x_i$,
rounded to \jnt{an integer value}. 

\smallskip
\noindent {\bf (b) Distributed hypothesis testing} and \green{\bf majority voting:} Consider $n$
sensors \green{interested in deciding} between two hypotheses,
$H_0$ and $H_1$. Each sensor collects measurements and makes a
preliminary decision \red{$x_i\in\{0,1\}$} in favor of one of the
hypotheses. The sensors would like to \green{make a final decision
by majority vote, in which case they need to compute the indicator
function of the event $\sum_{i=1}^n x_i \geq n/2$, in a
distributed way. Alternatively, in a weighted majority vote, they
may be interested in computing the indicator function of \jnt{an event such as}
$\sum_{i=1}^n x_i \geq 3n/4$.} \rt{A variation of this problem
involves the possibility that some sensors \jnt{abstain from the vote, perhaps} due to their inability to gather
sufficiently reliable information.}

\smallskip
\noindent {\bf (c) \aoc{Direction coordination on a ring}:}
\rt{Consider $n$ vehicles placed on a ring, each with some
arbitrarily chosen 
direction of motion (clockwise or counterclockwise). We would
like the $n$ vehicles to agree on a single direction of motion.}
\rt{A variation of this problem was considered in \cite{MBCF07},
where, however, additional requirements on the vehicles were
imposed which we do not consider here. The solution \jnt{provided} in
\cite{MBCF07} was semi-centralized in the sense that vehicles had
unique \jnt{numerical} identifiers, and the final direction of most \jnt{vehicles} was \jnt{set to the
direction of the vehicle} with the \jnt{largest identifier.}  We wonder \jnt{whether} the
direction coordination problem \jnt{can} be solved in a completely
decentralized way. Furthermore, we would like the final direction
of motion to correspond to the initial direction of the majority
of the vehicles: if, say, 90\% of the vehicles are moving
counterclockwise, \jnt{we would like} the other 10\%  to turn around.}
\rt{If we define $x_i$ to be $1$ \jnt{when} the $i$th vehicle is
initially oriented clockwise, and $0$ if it is oriented
counterclockwise, then, coordinating on a direction involves the
distributed computation of the \jnt{indicator function of the event}  $\sum_{i=1}^n x_i \geq
n/2$}.

\smallskip
\noindent {\bf (d) Solitude verification:} \green{This is the
problem of} \jnt{checking whether exactly} \green{one} node has a given state. \green{This problem is of interest if we want
to avoid simultaneous transmissions over a common channel
\cite{GFL83}, or if we want to maintain a single leader (as in
motion coordination  --- see for example \cite{jad}) Given
\aoc{states} $x_i \in
\{\aoc{0},1,\ldots,\blue{K}\}$, \green{solitude verification} is
equivalent to \green{the problem of} computing the binary function
which is equal to 1 if and only if $|\{i: x_i=\aoc{0}\}| =1$.}
\smallskip

There are numerous methods that have been proposed for solving
problems such as the above. Oftentimes, different algorithms involve
different computational capabilities on the part of the \jnt{nodes},
which makes it hard to talk about \jnt{a} ``best'' algorithm. At the
same time, simple algorithms (such as setting up a spanning tree and
\jnt{aggregating} information by progressive summations over the
tree, as in Chapter \ref{why}) are often \jnt{considered
undesirable} \blue{because they require} too much coordination or
global information. \jnt{It should be clear} that a sound discussion
of such issues requires the specification of a precise model of
computation, followed by a systematic analysis of fundamental
limitations under  a given model. This is precisely the objective of
this chapter: \blue{to} propose a particular model, and to
characterize the class of \jnt{functions computable} under this
model.

\subsection{\ora{The features of our model}}\label{se:features}

Our model provides an abstraction for \aoc{common}
requirements for distributed algorithms in the \ora{wireless} sensor network
literature.
\jnt{We
model the nodes as interacting
deterministic finite automata that
exchange messages on a fixed undirected network,
with no time delays or unreliable transmissions.
Some important qualitative features of our model are the following.}

\smallskip

\noindent \textbf{\red{Identical} \jnt{nodes}:} Any two \jnt{nodes} with the
same number of neighbors must run the same algorithm.

\smallskip

\noindent\textbf{Anonymity:} \green{\jnt{A node} can distinguish its
neighbors using its own, private, local identifiers. However,
\jnt{nodes} do not have global identifiers.}

\smallskip

\noindent\textbf{\jnt{Determinism:}}
\jnt{Randomization is not allowed. This restriction is imposed in order to preclude essentially centralized solutions that rely on randomly generated distinct identifiers and thus bypass the anonymity requirement. Clearly, developing algorithms is much harder, and sometimes impossible, when randomization is disallowed.}

\smallskip
\jnt{
\noindent\textbf{Limited memory:} We focus on the case where the nodes can be described by finite automata, and pay special attention to the required memory size. Ideally, the number of memory bits required at each node should be bounded above by a slowly growing function of the degree of a node.}

\smallskip

\noindent\textbf{Absence of global information:} \jnt{Nodes} have no
\green{global information, and do not even have an upper bound on
the total number of nodes. Accordingly, the algorithm that each
\jnt{node} is running is independent of the network size and
topology.}\smallskip

\noindent\textbf{Convergence \jnt{requirements:}} \jnt{Nodes hold an estimated output that
must converge to a desired value which is a function of all nodes' initial
observations or values.}
 \green{In particular, for the case
of discrete outputs,} all \jnt{nodes} must eventually \green{settle on}
the desired value. \green{On the other hand, the \jnt{nodes} do not
need to \jnt{become} aware of such termination, which is anyway impossible
in} the absence of \green{any} global information \cite{ASW88}.

\smallskip

\jnt{\jblue{In this chapter, we
only consider the special case of} \textbf{fixed graph topologies,} where the underlying (and unknown) interconnection graph does not change with time. Developing a meaningful model for the time-varying case and extending our algorithms to that case is \jblue{an interesting topic, but} outside the scope of this thesis.}

\subsection{\jnt{Literature review}}

There is a \jnt{very} large literature on distributed function
computation in related models of computation \cite{BT89, L96}.
\jnt{This literature can be broadly divided into two strands,
although the separation is not sharp: works that address general
computability issues for various  models, and works that focus on
the computation of specific functions, such as the majority function
or the average. We start by discussing the first strand.}

A common model in the distributed computing literature involves the
requirement that \green{all} processes terminate \jnt{once} the
\green{desired output is produced} \jblue{and \aoc{that} nodes become
aware
that termination has occurred.} A consequence of the termination
requirement is that nodes typically need to know the network size
$n$ (or an upper bound on $n$) to compute non-\aoc{trivial} functions. We
refer the reader to \cite{A80, ASW88, YK88, KKB90, MW93} for
\green{some} fundamental results in this setting, \green{and to
\cite{FR03} for} a \jnt{comprehensive} summary of known results.
\ora{Closest to our work is the reference \cite{CS}  which provides
an impossibility result very similar to our Theorem
\ref{unbounded-bound}, for a closely related model computation.}

\ora{The} biologically-inspired ``population algorithm" model
of distributed computation has some features in common with our
model, namely, \jblue{anonymous,} \jnt{bounded-resource} \jnt{nodes,} and
\jblue{no requirement of termination awareness;}
\green{see \cite{AR07} for \jnt{an overview of available}
results.} However, \green{\jnt{this} model involves a
different type of node interactions from the ones we consider;
in particular, \jblue{nodes interact pairwise at times that may be chosen adversarially.}
}

\jnt{Regarding the computation of specific functions, \cite{LB95}
shows} the impossibility of \jnt{majority voting} \jnt{if the nodes}
\jblue{are limited to a binary state.} Some experimental memoryless
algorithms \jnt{(which are not guaranteed to always converge to the
correct answer) have been} proposed \green{in the physics
literature} \cite{GKL78}.  Several papers have quantified the
performance of simple heuristics for computing specific functions,
typically in randomized settings. We refer the reader \blue{to}
\cite{HP01}, \aoc{which} studied simple heuristics \green{for computing}
the
majority function, \jblue{and to \cite{PVV08}, which provides a
heuristic that has guarantees only for the case of complete graphs.}

The large literature on quantized averaging often tends to involve
themes \jnt{similar to those addressed in this chapter}
\cite{FCFZ08, KM08, ACR07, CB08, KBS07}.  \jnt{However, the
underlying models of computation
 are typically more powerful than ours, as they allow for randomization and unbounded memory.
  Closer to the current chapter, \cite{NOOT07}  develops \ora{an algorithm with $O(n^2)$ convergence time}
  for a variant of the quantized averaging problem, but requires unbounded memory.
  Reference \cite{BTV08} provides an algorithm for the particular quantized averaging problem that
   we consider in Section \ref{sec:reduction_to_avg} \jblue{(called in \cite{BTV08} the ``interval consensus problem''),} which uses
\ora{randomization but only} bounded memory (a total of two bits at
each node). \ora{Its convergence time is addressed in \cite{DV10},
but a precise convergence time bound, as a function of $n$, is not
available. Nevertheless, \aoc{it appears} to be significantly higher
than $O(n^2)$.} Similarly, the algorithm in \cite{ZM09} runs in
$O(n^5)$ time for the case of fixed graphs. (However, we note that
\cite{ZM09} also addresses the case of time-varying graphs.) Roughly
speaking, the algorithms in \cite{BTV08, ZM09} work by having
positive and negative ``load tokens'' circulate randomly in the
network until they meet and annihilate each other. Our algorithm
involves a similar idea. However, at the cost of some algorithmic
complexity, our algorithm is deterministic. This allows for fast
progress, in contrast to the slow progress of algorithms that need
to wait until the coalescence time of two independent random walks.}
Finally, a deterministic algorithm for computing the majority
function (and some more general functions) was proposed in
\cite{LBRA04}. However, the algorithm appears to rely on
\aoc{the computation of shortest path lengths}, and thus \aoc{requires unbounded
memory at each node.}

Semi-centralized versions of the problem, in which the nodes
ultimately transmit to a fusion center, have often \jblue{been} considered
in the literature, e.g., for distributed statistical inference
\cite{MK08} \green{or detection} \cite{KLM08}. The papers
\cite{GK05},
 \cite{KKK05}, and \cite{YSG07} consider the complexity of computing a function and
communicating \green{its value} to a sink node. We refer the reader
to the references therein for an overview of \green{existing}
results in such semi-centralized settings. \jblue{However, the
underlying model is fundamentally different from ours, because the
presence of a fusion center violates our anonymity assumption.}

\jblue{Broadly speaking,} our results differ from previous works in
several key respects: (i) Our model, which involves totally
decentralized computation, deterministic algorithms, and constraints
on memory and computation resources at the nodes, but does not
require the nodes to know when the computation is over, \ora{is
different from that considered in \aoc{almost all} of the relevant
literature.}
 (ii) Our focus is on identifying computable
\green{and} non-computable functions \blue{under our model,} and
we achieve a nearly tight separation, \rt{as evidenced by Theorem \ref{unbounded-bound} and Corollary
\ref{cor:approx_set_epsilon}.}

\subsection{\jnt{Summary and Contributions}}

We \green{provide}  a general model of decentralized
anonymous computation \jblue{on fixed graphs,} \jnt{with the features described in Section \ref{se:features},}
and characterize the type of functions of the initial
\green{values} that can be computed.

We prove that if a function is computable under \green{our model,}
then its value \jnt{can only depend} on the \green{frequencies of the
different possible initial values.} For example, if the initial
\green{values} $x_i$ \jnt{are binary,} a computable
function \jnt{can} only depend on $p_0:= | \{i:x_i=0\}|/n$ and
$p_1:=|\{i:x_i=1\}|/n$. In particular, determining the number of
nodes, or \green{whether} at least two nodes have an initial
\green{value} \blue{of} $1$, is \green{impossible.}

Conversely, we prove that if a function only depends on the
\green{frequencies of the different possible initial values (and
is measurable)}, then the function can be approximated
\green{with any given} precision, except possibly on a set of
frequency vectors of  arbitrarily small volume. Moreover, if the
dependence on \green{these frequencies} can be expressed \jblue{through} a
combination of linear inequalities with rational coefficients,
then the function is computable exactly. In particular, the
functions involved in the quantized consensus, distributed
hypothesis testing, \rt{and direction coordination} examples are
computable, whereas the function involved in solitude verification
is not. Similarly, statistical measures such as the standard
deviation \jnt{of the distribution of the initial values} can be approximated with arbitrary
precision.
Finally, we show that with infinite memory, \red{the frequencies
of the different \jnt{initial} values} (i.e., $p_0$, $ p_1$ in the binary case) are
computable \jnt{exactly}, \aoc{thus obtaining a precise characterization of the
computable functions in this case.}

\jnt{The key to our positive results is the algorithm for calculating the (\jblue{suitably} quantized) average of the initial values \aoc{described in the previous Chapter \ref{ch:constantstorage}.}}

\subsection{\rt{Outline}}

In
Section \ref{sec:formal_descr}, we describe formally our model of computation. In Section
\ref{sec:nec_conditions}, we establish necessary conditions for a function
to be computable. In Section \ref{sec:reduction_to_avg}, \jnt{we provide}
sufficient conditions for a function to be
computable or approximable.
\jnt{Our positive results}
rely on an algorithm that keeps track of nodes
with maximal values, and an algorithm that calculates
\jnt{a suitably rounded average of the nodes' initial values; these} \aoc{were described
in the previous Chapter \ref{ch:constantstorage}}.
\aoc{We we end with some concluding remarks, in
Section \ref{sec:conclusions}}.

\section{Formal description of the model}\label{sec:formal_descr}

\jnt{Under our model, a distributed computing system consists of three elements:

\smallskip
\noindent {\bf (a) A network:} A network is a triple $(n,G,{\cal L})$,
where $n$ is the number of nodes, and $G=(V,E)$ is a \aoc{{\it connected undirected}} graph $G=(V,E)$ with $n$
nodes (and no self-arcs). 
We define $d(i)$ as the of node $i$. Finally, $\cal L$ is a \emph{port labeling} which assigns a
\emph{port number} \ora{(a distinct integer in the set $\{\aoc{0,1,}\ldots,d(i)\}$)}  to each outgoing edge of
\jblue{any} node $i$.}

\ora{We are interested in algorithms that work for arbitrary port labelings. However, in the case of wireless networks, port labelings can be assumed to have some additional structure. For example, if two neighboring nodes $i$ and $j$ have coordinated so as to communicate over a distinct frequency, they should be able to coordinate their port numbers as well, so that the port number assigned by any node $i$ to an edge $(i,j)$ is the same as the port number assigned by $j$ to edge $(j,i)$. \aoc{We will say that a network is {\em edge-labeled} if port numbers have this property.} In Section~\ref{sec:nec_conditions} we will note that our negative results also apply to \aoc{edge-labeled networks}.
}

\jnt{
\smallskip\noindent {\bf (b) Input and output sets:}
The input set is a finite set $X=\{0,1,\ldots,K\}$ to which the initial value of each node belongs. The output set is a finite set $Y$ to which the output of each node belongs.}

\jnt{
\smallskip\noindent {\bf (c) An algorithm:} An algorithm is defined as a family of finite
automata $(A_d)_{d = 1,2,\ldots}$, where the automaton $A_d$
describes the behavior of a node with degree $d$.}
The state of the automaton $A_d$ is \green{a} tuple $\jnt{[}x,z,y; (m_1,
\ldots, m_d)]$;  we will call $x \in X$ the
{\it initial value}, $z\in \red{Z_d}$ the {\it internal memory state},
$y\in Y$ the {\it output} or {\it estimated answer}, and $m_1, \ldots, m_d \in
M$ the \jnt{outgoing} {\it messages.} The sets $Z_d$ and $M$ are assumed finite.
\jnt{We allow the cardinality of $Z_d$ to increase with $d$. Clearly, this would be necessary for any algorithm
that needs to store the messages received in \aoc{the} previous time step.}
\jnt{Each automaton $A_d$ is identified with a transition law
from $X \times \blue{Z_d}
\times Y \times M^d$ into itself, which maps each}
$\left[x,z,y;(m_1,\dots,m_d)\right]$ to some
$\left[x, z', y';(m'_1,\dots, m'_d)\right]. $ In words, \jblue{at each iteration, the
automaton takes $x$, $z$, $y$, and incoming messages into account, to create} a new memory state, output, and
\green{(outgoing)} messages, but does not change
the initial \blue{value.}

\smallskip
\jnt{Given the above elements of a distributed computing system, an algorithm proceeds as follows.
For convenience, we assume that the above defined sets $Y$, $Z_d$, and
$M$ contain a special element, denoted by $\emptyset$.
Each} node $i$ begins with an initial value
$x_i\in X$ and implements the automaton $A_{d(i)}$, initialized with $x=x_i$ and
$z=y=m_1=\cdots=m_d = \emptyset$.
\green{We use $S_i(t)=
[ \blue{x_i},y_i(t),z_i(t),m_{i,1}(t), \ldots, m_{i, d(i)}(t)]$
to denote the}
state
of \jblue{node $i$'s} automaton at \jblue{time} $t$.
\jnt{Consider a particular node $i$. Let}
$j_1, \ldots,
j_{d(i)}$ be an enumeration of \jnt{its} neighbors, according to the port numbers. (Thus, $j_k$ is the \jblue{node at the other end of} the $k$th outgoing edge at node $i$.) Let $p_k$
be the port number \jnt{assigned to link $(j_k,i)$ according to the port labeling at node $j_k$. At each time step, node $i$ carries out the following update:}

\begin{small}
\begin{eqnarray*}\left[x_i,
z_i(t+1),y_i(t+1); m_{i,1}(t+1), \ldots, m_{i, d(i)}(\red{t+1}) \right]  & 
 & \\  = A_{d(i) }\left[\jnt{x_i}, z_i(t), y_i(t) ; m_{\blue{j_1, p_1}}(t),
\ldots, m_{j_{d\jnt{(i)}}, p_{d{\jnt{(i)}}}}(t) \right]. &&
\end{eqnarray*}
\end{small}

\vspace{-8pt}
\noindent
In words, the
messages \jnt{$m_{j_{k}, p_{k}}(t)$, $k=1,\ldots,d(i)$,} ``sent'' by the  neighbors of $i$ into the ports
leading to $i$ are used to transition to a new state and
create new messages \jnt{$m_{i,k}(t+1)$, $k=1,\ldots,d(i)$,} that  $i$ ``sends'' to its neighbors at \jblue{time $t+1$.} We say that
\jnt{the algorithm \emph{terminates} if there exists some
$y^*\in Y$ (called the \emph{\green{final output}} of
the algorithm) and} a time $t'$ such that $y_i(t) =
y^*$ for every $i$ and $t\geq t'$.

Consider now a family of functions $(f_n)_{n = 1,2,\ldots}$, \jnt{where
$f_n:X^n\to Y$.} We say that such a
family is \emph{computable} if there exists a family of automata
$(A_d)_{d=1,2,\ldots}$ such that \blue{for any $n$,} for any
\jnt{network $(n,G,{\cal L})$},
and any set of initial conditions
$x_1,\dots,x_{\blue{n}}$, the \jnt{resulting algorithm terminates and the final output} is
$f_{\blue{n}}(x_1,\dots,x_n)$.

\jnt{As an exception to the above definitions, we note that although we primarily focus on the finite case, we will} \jblue{briefly consider
in Section \ref{se:inf}}
function families $(f_n)_{n = 1,2,\ldots}$ {\em computable with infinite memory},  by which
we mean that the internal memory \blue{sets $Z_d$} and \jnt{the} output set $Y$ are
\jnt{countably infinite,} the rest of the model \jnt{remaining the same.}

\jnt{The rest of the chapter focuses on the following general
question:} what \green{families} of functions are computable, and
how can we design a \jnt{corresponding algorithm
$(A_d)_{d=1,2,\ldots}$? To illustrate the nature of our model and
the type of algorithms that it supports, we provide a simple
example.}

\smallskip

\rt{\noindent \textbf{Detection problem:} In this \jnt{problem}, all
nodes \jnt{start with a binary initial value $x_i\in \{0,1\}=X$. We wish to
detect whether at least one node has an initial value equal to $1$.
We are thus dealing with the function family $(f_n)_{n=1,2,\ldots}$, where
$f_n(x_1,\ldots,x_n)=\max\{x_1,\ldots,x_n\}$.
This function family is computable by a family of
automata with binary messages, \jblue{binary} internal state, and
with the following transition rule:}}

\algsetup{indent=2em}
\begin{algorithmic}
\medskip
\IF {$x_i=1$ or $z_i(t) = 1$ or $\jnt{\max}_{j:(i,j)\in E}m_{ji}(t) =
1$}
 \STATE
set $z_i(t+1) = y_i(t+1) =1$ \STATE send $m_{ij}(t+1) = 1$ to
every neighbor $j$ of $i$

\ELSE \STATE set $z_i(t+1) = y_i(t+1) =0$ \STATE send $m_{ij}(t+1)
= 0$ to every neighbor $j$ of $i$

\ENDIF
\end{algorithmic}

\smallskip

\jnt{In the above algorithm, we initialize by setting $m_{ij}(0)$, $y_i(0)$, and $z_i(0)$ to zero instead of the special symbol $\emptyset$.}
One can easily verify that if $x_i=0$ for every $i$,
\jnt{then $y_i(t)=0$ for all $i$ and $t$.}
If on the other hand $x_k=1$
for some $k$, then at each time step $t$, those nodes \jnt{$i$} at distance
less than $t$ from \jnt{$k$} will have \jnt{$y_i(t)=1$.}
\jnt{Thus, for connected graphs, the algorithm will terminate within $n$ steps, with the correct output.}
It is important to note, \jnt{however, that because $n$ is unknown, a node $i$
can never know whether
its current output $y_i(t)$} is the final one. In particular, if
$y_i(t)=0$, node $i$
\jnt{cannot exclude the possibility that $x_k=1$ for some node whose distance from $i$ is larger than $t$.}

\section{Necessary condition for computability}\label{sec:nec_conditions}

\jnt{In this section we establish our main negative result, namely,
that if a function family is computable, then the final output can
only depend on the frequencies of the different possible initial
values. Furthermore, this remains true even if we allow for infinite
memory,} \ora{or restrict to networks in which neighboring nodes
share a common label for the edges that join them. This result is
quite similar to Theorem 3 of \cite{CS}, and so is the proof.
Nevertheless, we provide a proof in order to keep the chapter
self-contained.}

\def\pb{frequency-based}

\jnt{We first need some definitions. Recall that $X=\{0,1,\ldots,K\}$. We let $D$ be the \emph{unit simplex,} that is,} $D =\{(\jnt{p_0},\dots,p_K)\in [0,1]^{\jnt{K+1}} : \sum_{k=\jnt{0}}^Kp_k =1 \}$. We say that a function $h:D\to Y$
\emph{corresponds} to a \jnt{function family $(f_n)_{n=1,2,\ldots}$ if for every $n$ and} every $x\in \blue{X}^n$, we have
$$
f(x_1,\dots,x_n)=
h\left(\aoc{p_0}(x_1,\ldots,x_n), p_1(x_1,\ldots,x_n),\dots,p_K(x_1,\ldots,x_n)\right),
$$
where
\[ p_k(x_1,\ldots,x_n) = |\{ \blue{i} ~|~ x_{\blue{i}} = \blue{k} \}|/n,\] so that
$p_k(x_1,\ldots,x_n)$ is the frequency of occurrence of the initial value $k$.
In this case, we say that the family $(f_n)$ is \emph{\pb.}

\begin{theorem} \label{unbounded-bound}
Suppose that the family $(f_n)$ is computable with infinite memory,
Then, this family is \pb. \ora{The result remains true even if we
only require computability over edge-labeled networks.}
\end{theorem}

The following \jnt{are some applications of Theorem \ref{unbounded-bound}.}

\def\texitem#1{\par
\noindent\hangindent 25pt
\hbox to 25pt {\hss #1 ~}\ignorespaces}

\texitem{(a)}
The parity function $\sum_{i=1}^n x_i ~ ({\rm mod} \ k)$ is not
computable, for any $k>\aoc{1}$.
\texitem{(b)}
In a binary setting \blue{$(X=\{0,1\})$}, checking whether the number of nodes with
$x_i=1$ is \jnt{larger than or equal to} the number of nodes with $x_i=0$ \jnt{plus 10} is
not computable.
\texitem{(c)}
Solitude verification, i.e.,
checking whether $|i: \{x_i=0\}|=1$, is not computable.
\texitem{(d)}
An aggregate difference function such as
$\sum_{i<j} |x_i - x_j|$ is not computable, even \jnt{if it is to be} \green{calculated modulo $k$.}

\subsection{Proof of Theorem \ref{unbounded-bound}}

\ora{The proof of Theorem \ref{unbounded-bound} involves a particular degree-two network (a ring), in which all port numbers take values in the set $\{\aoc{0},1,2\}$, and in which any two
edges $(i,j)$ and $(j,i)$ have the same port number.
The proof proceeds through}
\jnt{a sequence of intermediate results, starting
with} the following lemma, which can be easily proved
by induction on time; \jnt{its proof is omitted.}

\begin{lemma} \label{switching}
Suppose that $G=(\{1,\ldots,n\},E)$ and $G'=(\{1,\ldots,n\},E')$
are isomorphic; that is, there exists a permutation $\pi$ such
that $(i,j) \in E$ if and only if $(\pi(i),\pi(j)) \in E'$.
Furthermore, suppose that the \green{port label at node $i$ for the edge leading to} $j$ in $G$ is the
same as the \green{port label at node $\pi(i)$ for the edge leading to} $\pi(j)$ in $G'$.  Then, \green{the state
$S_i(t)$ resulting} from \blue{the initial values $x_1,\ldots,x_n$} \jnt{on}
the graph $G$ is the same as \blue{the state} $S_{\pi(i)}(t)$
\blue{resulting from the initial values $x_{\jnt{\pi^{-1}}(1)},\ldots,x_{\jnt{\pi^{-1}}(n)}$}
 \jnt{on} the graph $G'$.
\end{lemma}

\begin{lemma} \label{symmetry}
Suppose that the family \jnt{$(f_n)_{n=1,2,\ldots}$} is computable with infinite memory
\aoc{on edge-labeled networks}. Then, each $f_i$
is invariant under permutations of its arguments.
\end{lemma}
\begin{proof}
\blue{Let ${\pi_{ij}}$ be \jnt{the} permutation that swaps $i$ with $j$ \jnt{(leaving the other nodes intact)}; with a slight abuse of notation, we also denote by $\pi_{ij}$} the mapping from \blue{$X^n$ to $X^n$} that swaps the
$i$th and $j$th elements of a vector.
\jnt{(Note that $\pi_{ij}^{-1}=\pi_{ij}$.)}
We show that for all $x \in
\blue{X^n}$, $f_{\blue{n}}(x) = f_{\blue{n}}( \pi_{ij}( x))$.

We run our \green{distributed} algorithm \green{on} the $\blue{n}$-node
\aoc{complete graph with an edge labeling.} \aoc{Note that at least one
edge labeling for the complete graph exists: for example, nodes $i$ and $j$ can use port number $(i+j) ~{\rm mod }~ n$ for the edge connecting them.}. Consider two different \jnt{sets of} initial
\jnt{values, namely the vectors  (i) $x$,  and (ii) $\pi_{ij}(x)$.
Let the port labeling in} case (i) be arbitrary;
in case (ii), let the \jnt{port labeling be}
\blue{such that the conditions in Lemma \ref{switching} are satisfied  (which is easily accomplished).}
Since the \jnt{final value} is $f(x)$ \jnt{in case (i)} and
$f(\pi_{ij}(x))$ \jnt{in case (ii)}, we obtain $f(x)=f(\pi_{ij}(x))$.
\blue{Since the permutations $\pi_{ij}$ generate the group of permutations, permutation invariance follows.}
\end{proof}

Let \blue{$x\in X^n$.}  We will \green{denote by} $x^2$ the
concatenation of $x$ with itself, and, generally, \jnt{by} $x^k$ the
concatenation of $k$ copies of $x$. We now prove that
self-concatenation does not affect the value of a computable
family of functions.

\begin{lemma} \label{repetition} Suppose that the family
\jnt{$(f_n)_{n=1,2,\ldots}$} is computable with
infinite memory \aoc{on edge-labeled networks}. Then, \red{for every $n\geq 2$}, every sequence
\blue{$x\in X^n$}, and \green{every positive} integer~$m$, \[ f_n(x) = f_{mn} (x^m).\]
\end{lemma}

\begin{proof}
Consider a ring of \jnt{$n$ nodes,} where the $i$th \jnt{node}
\aoc{clockwise} begins with the $i$th element of $x$; and consider
a ring of \jnt{$mn$ nodes}, where the \jnt{nodes}
$i,i+n,i+2n,\ldots$ (\aoc{clockwise}) begin with the $i$th element
of $x$. \aoc{Suppose that the labels in the first ring are $0,1,2,1,2,\ldots$.
That is, the label of the edge $(1,2)$ is $0$ and the labels of the subsequent edges
alternate between $1$ and $2$. In the second ring, we simply repeat $m$ times the labels in
the first ring. See Figure \ref{rings} for an example with
$n=5$, $m=2$.}

\begin{center}
\begin{figure}[h] \hspace{1cm}
\begin{minipage}[t]{.45\textwidth}
  \begin{center}        \epsfig{file=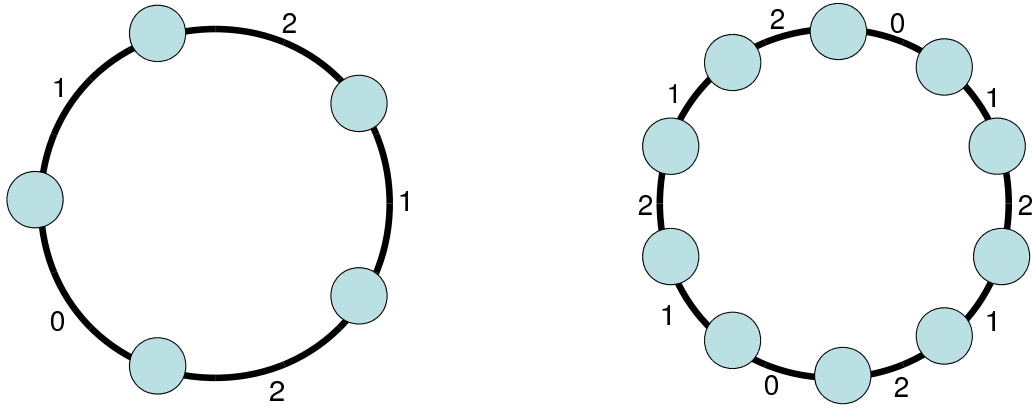, scale=0.55}
  \end{center}
\end{minipage}
\caption{Example of two situations that are \jnt{algorithmically} \blue{indistinguishable}. \ora{The numbers next to each edge are the edge labels.}
} \label{rings}
\end{figure}   \end{center}

Initially, the state \jnt{$S_i(t)=
[ \blue{x_i},y_i(t),z_i(t),m_{i,1}(t), m_{i, 2}(t)]$, with $t=0$,} of node $i$ in the first ring is exactly the
same as the state of the nodes \blue{$j=i,i+n,i+2n,\ldots$} in the second ring. We show by induction \green{that this property} must hold at all
times $t$. \blue{(To keep notation simple, we assume, without any real loss of generality,
 that $i\neq 1$ and $i\neq \aoc{n}$.)}

Indeed, suppose this \green{property} holds up to time $t$. At time $t$, node $i$ in
the first ring receives a message from node $i-1$ and a message
from node $i+1$; and in the second ring, node \red{$j$} satisfying $\red{j \ (\mbox{mod } n)  } =
i$ receives one message from $j-1$ and $j+1$.
Since $\red{j-1 \ (\mbox{mod } n)  } =
i-1 \aoe{\ (\mbox{mod } n) }$ and $\red{j+1 \ (\mbox{mod } n)  } =
i+1 \aoe{\ (\mbox{mod } n)}$,
the states of \green{$j-1$ and $i-1$} are identical at time $t$,
\green{and similarly for}
$j+1$ and $i+1$. Thus, \aoc{because of periodicity of the edge labels, nodes} $i$ (in the first ring) and $j$ (in the second ring) \aoe{receive
identical messages
\aoc{through identically labeled ports}} at time $t$. Since $i$ and $j$ were in the same state at time $t$, they
must be in the same state at time $t+1$. This proves that they are \jnt{always}
in the same state.
It follows that $y_i(t) = y_j(t)$ for all $t$, whenever
$\red{j \ (\mbox{mod } n)  } =
i$, and therefore $f_n(x)=f_{mn}(x^m)$.
\end{proof}

\begin{proof}[Proof of Theorem \ref{unbounded-bound}]
Let $x$ and $y$ be two sequences of $n$ and $m$
elements, respectively,  such that $p_k(x_1,\dots,x_n)$ \jnt{and} $p_k(y_1,\dots,y_m)$ \jnt{are equal to a common value} $\hat
p_k$, for $\jnt{k\in X}$; \jnt{thus,} the number of occurrences
of $k$ in $x$ and $y$ are $n\hat p_k$ and $m\hat
p_k$, respectively. Observe that \blue{for any $k\in X$, the vectors}  $x^m$ and $y^n$ have the same number $mn$ of
elements, and both contain $mn\hat p_k$ occurrences of \green{$k$.}
The sequences $y^n$ and $x^m$ can thus be obtained
from each other by a permutation, which by \jnt{Lemma}
\ref{symmetry} implies that $f_{nm}(x^m) = f_{nm}(y^n)$. From \jnt{Lemma} \ref{repetition}, \jnt{we have} that $f_{nm}(x^m) = f_n(x)$
and $f_{mn}(y^n)=f_m(y)$. \jnt{Therefore,} $f_n(x) = f_m(y)$. This proves
that the value of $f_n(x)$ is determined \aoc{by the values of $p_k(x_1,\ldots,x_n), k = 0,1, \ldots, K$.}
\end{proof}


\section{Reduction of generic functions to the computation of averages}\label{sec:reduction_to_avg}

\jnt{In this section, we turn to positive results, aiming at a converse of
Theorem
\ref{unbounded-bound}.
The centerpiece of our development is Theorem \ref{thm:average_computation2}, \aoc{proved in the previous chapter},  which states that a certain average-like function is computable. Theorem \ref{thm:average_computation2} then implies the computability of a large class of functions, yielding an approximate converse to Theorem
\ref{unbounded-bound}.}

\jnt{The average-like functions that we consider correspond to the ``interval consensus'' problem studied in \cite{BTV08}. They are defined as follows.
Let $X = \{0,\ldots,K\}$. Let $Y$ be the following set of single-point sets and intervals:
\[ Y = \{ \{0\}, (0,1), \{1\}, (1,2),  \ldots,
\{K-1\}, (K-1,K), \{K\} \}\]
(or equivalently, an indexing of this finite collection of intervals).
For any $n$, let $f_n$ be the function that maps $(x_1,x_2,\ldots,x_n)$ to the element
of $Y$ which contains the average $\sum_i x_i/n$. We refer to the function family $(f_n)_{n=1,2,\ldots}$ as the \emph{interval-averaging} family.}

\jnt{The motivation for this function family comes from the fact that the exact average $\sum_i x_i/n$ takes values in a countably infinite set, and cannot be computed when the set $Y$ is finite. In the quantized averaging problem considered in the literature, one settles for an approximation of the average. However, such approximations do not necessarily define a single-valued function from $X^n$ into $Y$. In contrast, the above defined function $f_n$ is both single-valued and delivers an approximation with an error of size at most one. Note also that once the interval-average is computed, we can readily determine the value of the average rounded down to an integer.}

\begin{theorem}\label{thm:average_computation2} \jnt{The interval-averaging function family}  is computable.
\end{theorem}


\jnt{The proof of Theorem \ref{thm:average_computation2} (and the
corresponding algorithm) was given in Chapter \ref{ch:constantstorage}} \aoc{In this section, we show that the
computation of a broad class of functions can be reduced to
interval-averaging.}

\jnt{Since only \pb\ function families can be computable (Theorem \ref{unbounded-bound}), we can restrict attention to the corresponding
functions $h$.}
\green{We
will say that a function $h$ on the \jnt{unit simplex $D$} is
\emph{computable} if it it corresponds to a \pb\
computable family $(f_n)$.  The level sets of $h$ are defined as
the sets $L(y) = \{ \red{p} \in D ~|~ h(p) = y\}$, for $y \in Y$.}

\begin{theorem}[Sufficient condition for computability] \label{bounded-alg}
Let $h$ be a function from the \jnt{unit simplex} $D$ to $Y$. \green{Suppose that} every level set $L(y)$ can be written as a
finite union,
\[ L(y) =  \bigcup_k C_{i,k},\] where each $C_{i,k}$ can in turn be
written as a finite intersection of \green{linear} inequalities \green{of the form}
\[ \aoc{\alpha_0 p_0} + \alpha_1 p_1 + \alpha_2 p_2 + \cdots + \alpha_K p_K \leq \alpha,
\] or
\[ \aoc{\alpha_0 p_0} + \alpha_1 p_1 + \alpha_2 p_2 + \cdots + \alpha_K p_K < \alpha,
\]
with rational \green{coefficients}  $\alpha,\aoc{\alpha_0}, \alpha_1,\ldots,\alpha_K$. Then, $h$
is computable.
\end{theorem}

\begin{proof}
Consider one such linear inequality, \jnt{which we assume, for \aoc{concreteness}, to be of the first type.} Let $P$ be the set of indices
\jnt{$i$} for which $\alpha_k \geq 0$. Since all coefficients are rational, we
can clear \aoc{their} denominators and rewrite \green{the inequality} as
\begin{equation}\label{eq:integ_equiv_density_ineq}
\sum_{k \in P} \beta_k p_k - \sum_{\jnt{k} \in P^c} \beta_k p_k \leq
\beta , \end{equation} for \red{nonnegative} integers $\beta_k$ and
$\beta$. Let $\chi_k$ be the indicator function \green{associated with initial
value} $k$, i.e., $\chi_k(i)=1$ if $x_i = k$, and $\chi_k(i) =0$
otherwise, so that $p_k=\frac{1}{n}\sum_{i}\chi_k(i)$.
\green{Then, (\ref{eq:integ_equiv_density_ineq}) becomes} $$\frac{1}{n} \sum_{i=1}^n \left(\sum_{k\in P}\beta_k
\chi_k(i) + \sum_{k\in P^c}\beta_k  (1-\chi_k(i))\right) \leq
\beta + \sum_{k\in \red{P^c}}\beta_k, $$ or \[ \frac{1}{n}\sum_{i\aoc{=1}}^{\aoc{n}} q_i
\leq  q^*, \] where $q_i = \sum_{k\in P}\beta_k \chi_k(i) +
\sum_{k\in P^c}\beta_k (1-\chi_k(i))$ and $q^* =\beta + \sum_{k\in
P^c}\beta_k$.

To determine \jnt{whether} the \aoc{last} inequality is satisfied, each node can
compute $q_i$ and $q^*$, and then apply a distributed algorithm that
computes \jnt{the integer part of} $\frac{1}{n}\sum_{i=1}^n q_i$;
\jnt{this is possible by virtue of Theorem
\ref{thm:average_computation2}, with $K$ set to $\sum_k\beta_k$ (the largest possible value of $q_i$).} To check any finite collection of
inequalities, the nodes can perform the computations for each
inequality in parallel.

To compute $h$, the nodes simply need to check which \green{set} $L(y)$ the
\green{frequencies} $\aoc{p_0}, p_1,\ldots,p_K$ lie in, and this can be done by
checking the inequalities defining each $L(y)$.
\green{All of these computations} can be accomplished
with finite automata: indeed, we do nothing more than run finitely
many copies of the automata provided by Theorem
\ref{thm:average_computation2}, one for each inequality.
\jnt{The total memory used by the automata depends on the number of sets $C_{i,k}$ and the magnitude of the coefficients $\beta_k$, but not on
$n$, as required.}
\end{proof}

Theorem \ref{bounded-alg} shows the computability of functions $h$
whose level-sets \green{can} be defined by linear inequalities
with rational coefficients. \red{On the other hand, it is clear
that not every function $h$ can be computable.} \blue{(This can be
shown by a counting argument: there are uncountably many possible
functions $h$ \jnt{on the rational elements of $D$,} but for the special case of bounded degree graphs,
only countably \aoc{many} possible algorithms.)} \green{Still, the next \jnt{result}
shows that the set of computable functions is rich enough, in the
sense that \jnt{computable}} functions can approximate any measurable
function, \jnt{everywhere except possibly on a low-volume set.}

We will call a set of the form $\prod_{\red{k}=\jnt{0}}^K (a_k,b_k)$,
\green{with every $a_k,b_k$ rational,} {\em a \red{rational open}
box}, where $\prod$ \green{stands for Cartesian} product. A
function that can be written as a finite sum $\sum_i a_i
\red{1_{B_i}}$, where \green{the $B_i$ are rational open boxes and
the $1_{B_i}$ are the associated indicator functions,} will be
referred to as a \emph{box function.} \green{Note that box
functions are computable by Theorem \ref{bounded-alg}.}

\begin{corollary}\label{cor:approx_set_epsilon}
If every level set of a function $h:D\to Y$ on the \jnt{unit simplex $D$}
is Lebesgue measurable, then, for every $\epsilon >0$, there exists a
computable box function $h_\epsilon:D\to Y$ such that the
set \aoc{$\{ p \in D ~|~ h(p) \neq h_{\epsilon}(p) \}$} has measure at most $\epsilon$.
\end{corollary}
\begin{proof} \jnt{(Outline)}
The proof relies on the following elementary result from measure
theory.
Given a Lebesgue measurable set $E \jnt{\subseteq D}$ and some $\epsilon
> 0$, there \green{exists a set $E'$ which is a finite union of
disjoint open boxes, and which satisfies}
\[ \mu( \aoc{E \Delta E'} ) < \epsilon,\] where $\mu$ is
the Lebesgue measure.
\red{By a routine argument, these boxes can be taken to be rational.
By applying this fact to the level sets of the function $h$ (assumed measurable), the function $h$ can be approximated by a box function}  \blue{$h_{\epsilon}$. Since box functions are computable, the result follows.}
\end{proof}

The following corollary \green{states} that  
continuous functions are approximable.

\begin{corollary}\label{cor:approx_continuous}
If a function $h:D\to [L,U]\subseteq \Re$ is continuous, then for
every $\epsilon>0$ there exists a computable function
$h_\epsilon:D \blue{\to  [L,U]}$ such that
$\|h-h_\epsilon\|_\infty < \epsilon$
\end{corollary}
\begin{proof}
Since $D$ is compact, \jnt{$h$} is uniformly continuous. One can
therefore \green{partition} $D$ into a finite number of subsets,
$A_1,A_2,\dots,A_q$, that can be \green{described} by linear inequalities
with rational coefficients, so that $\max_{p\in
A_j}h(p)-\min_{p\in A_j}h(p) <\epsilon$ holds for all $A_j$. The
function $h_\epsilon$ is then built by assigning to each $A_j$ an
\green{appropriate} value \green{in} $\{L,L+\epsilon, L+2\epsilon, \dots, U\}$.
\end{proof}

To illustrate \jnt{these} results, let us consider again
some examples.

\texitem{(a)} \green{Majority \jnt{voting} between two options is} equivalent
to checking whether $p_1 \leq 1/2$, with alphabet $\{0,1\}$, and is
therefore computable.
\texitem{(b)}
Majority \jnt{voting} when some nodes can ``abstain'' amounts to checking whether $p_1 -
p_{\aoc{0}} \geq 0$, with \jnt{input set} $\jnt{X=}\{0,1,{\rm abstain}\}$. This function family is
computable.
\texitem{(c)}
We can ask for the second
most popular value out of four, for example. In this case, the
sets $A_i$ can be decomposed into constituent sets defined by
inequalities such as $p_2 \leq p_3 \leq p_4 \leq p_1$, each of which
obviously has rational coefficients.
\texitem{(d)}
\green{For any subsets $I,I'$ of
$\{\aoc{0}, 1,\ldots,K\}$, the indicator function of the set
where} $\sum_{i \in I} p_i
\jnt{>} \sum_{i \in I'} p_i$  is computable. This is equivalent to checking
whether more nodes have a value in $I$ than do in $I'$.
\texitem{(e)}
The \green{indicator functions of the sets defined by} $p_1^2 \leq 1/2$ and
$p_1 \leq \pi/4$ are measurable, so they are approximable. We are
unable to say whether they are computable.
\texitem{(f)}
The \green{indicator function of the set defined by} $p_1 p_2 \leq 1/8$ is approximable, but we are unable
to say whether it is computable.

\subsection{\jnt{Computability with infinite memory}}\label{se:inf}
Finally, we show that with infinite memory, it is possible to
recover the exact \green{frequencies $p_k$.} \red{(Note that this
is impossible with finite memory, because $n$ is unbounded, and
the number of bits needed to represent $p_k$ is also unbounded.)}
\jnt{The main difficulty is that $p_k$ is a rational number whose denominator can be arbitrarily large, depending on the unknown value of $n$. The idea is to run separate algorithms for each possible value of the denominator (which is possible with infinite memory), and reconcile their results.}

\begin{theorem} \label{thm:infinite_mem_positive}
The vector $(\aoc{p_0}, p_1,\ldots,p_K)$ is computable with infinite memory.
\end{theorem}
\begin{proof}
We show that $p_1$ is computable exactly, which is sufficient to
prove the theorem. Consider the following algorithm, \jnt{to be referred to as $Q_m$,} parametrized
by a \green{positive integer} \red{$m$.} The \jnt{input} set $X_m$ \jnt{is} $\{0,1,\ldots,
m\}$ and the output set $Y_m$
\jnt{is the same as in the interval-averaging problem:}
$Y_m = \{ \{0\}, (0,1), \{1\}, (1,2), \{2\}, (2, 3),
\ldots, \{m-1\}, (m-1,m), \{m\} \}$. If $x_i=1$, then node sets
its initial value $x_{i,m}$ to $m$; else, the node sets its
initial value $x_{i,m}$ to $0$. \green{The algorithm computes} the
function family $(f_n)$ which maps $X_m^n$ to the element of $Y_m$
containing $(1/n) \sum_{i=1}^n x_{i,m}$, \green{which is possible, by} Theorem
\ref{thm:average_computation2}.

The nodes run \green{the algorithms} $Q_m$ for every
\green{positive integer value of $m$, in an interleaved manner.
Namely, at each time step, a node runs one step of a particular
algorithm $Q_m$, according to} the following order:
\[ Q_1, \hspace{0.2cm} Q_1, Q_2, \hspace{0.2cm} Q_1, Q_2, Q_3, \hspace{0.2cm} Q_1, Q_2, Q_3,
Q_4, \hspace{0.2cm} Q_1, Q_2, \ldots \]

At each time $t$, let \red{$m_i(t)$} be the smallest $m$
\green{(if it exists)} such that \green{the output
\red{$y_{i,m}(t)$} of $Q_m$ at node $i$}  is a \green{singleton}
(not an interval). \green{We identify this singleton with the
numerical value of its single element, and we} set $y_i(t) = y_{i,
m_i(t)}(t)/m_i(t)$. If \green{$m_i(t)$ is undefined},  then
$y_i(t)$ is set to some default value, \jnt{e.g., $\emptyset$.}

\jnt{Let us fix a value of $n$. For any $m\leq n$,}
the definition of $Q_m$ and Theorem
\ref{thm:average_computation2} \jnt{imply} that there exists a time after which
the outputs $\red{y_{i,m}}$ of \jnt{$Q_m$}
do not change, and are \jnt{equal to a common value, denoted $y_m$, for every $i$.} Moreover, at least one of \jnt{the
algorithms $Q_1,\ldots,Q_n$} has an integer output $y_m$. Indeed,  observe that $Q_n$
computes $(1/n) \sum_{i=1}^n n1_{x_i=1} = \sum_{i=1}^n 1_{x_i=1}$,
which is clearly an integer. \red{In particular, $m_i(t)$ is
eventually well-defined and bounded above by $n$.} We conclude
that there exists a time after which the output \jnt{$y_i(t)$} of our
\green{overall} algorithm is fixed, shared by all nodes, and
different from the default value \jnt{$\emptyset$}.

We now argue that this value is indeed $p_1$. \red{Let
$m^*$ be} the smallest $m$ for which the eventual output of $Q_m$
is a single integer $y_m$. Note that $y_{m^*}$  is the exact
average
of the $x_{i,m^*}$, i.e., \[ y_{m^*} = \frac{1}{n} \sum_{i=1}^n
m^* 1_{x_i=1} = m^* p_1.\] \red{For large $t$, we have
\jnt{$m_i(t)=m^*$ and therefore} $y_i(t) =
y_{i, m^*}(t)/m^* =p_1$, \jnt{as desired.}}

Finally, it remains to argue that the algorithm described here can
be implemented with a sequence of \jnt{infinite memory} automata. All the above
algorithm does is run a copy of all the automata implementing
$Q_1, Q_2, \ldots$ with time-dependent transitions.  This can be
accomplished with an automaton whose state space is \green{the
countable set} $\mathcal{N} \times \red{\cup_{m
=1}^{\infty}\prod_{i=1}^{m} {\cal Q}_i}$, where ${\cal Q}_i$ is
the state \green{space} of $Q_i$, and the set $\mathcal{N}$ of
integers is used to keep track of time.
\end{proof}

\section{\rt{Conclusions}}\label{sec:conclusions}

We have proposed a model of \ora{deterministic anonymous distributed computation,} inspired by the \ora{wireless sensor network and} multi-agent
control literature. We have given an almost tight characterization
of the functions that are computable in our model. We have shown
that computable functions must depend only on the \ora{the frequencies with which}
the different initial conditions \ora{appear,} and that if this dependence can
be expressed in term of linear inequalities with rational
coefficients, the function is indeed computable. Under weaker
conditions, the function can be approximated with arbitrary
precision. It remains open to exactly characterize  the class of
computable function {\ora families}.

Our positive results are proved constructively, by providing \aoc{a}
generic algorithm for computing the desired functions. Interestingly,
the finite memory requirement is not used in our negative
results, which remain thus valid in the infinite memory case. In
particular, we have no examples of functions that can be computed
with infinite memory but are provably not computable with
finite memory. We suspect though that simple examples
exist; a good candidate could be the indicator function
$1_{p_1<1/\pi}$, which checks whether the fraction of nodes with a
particular initial condition is smaller than
$1/\pi$.

\chapter{Concluding remarks and a list of open problems}

This thesis investigated several aspects of the convergence times of averaging algorithms and the effects of quantized
communication and storage on performance. Our goal has been to try to understand some aspects of the tradeoffs between
robustness, storage, and convergence speed in distributed systems. 

The main results of this thesis are:

\begin{enumerate} \item An $O \left( (n^2/\eta) B \log (1/\epsilon) \right)$ upper bound on the convergence time of products of doubly
stochastic matrices (which lead to a class of averaging algorithms) in Chapter \ref{chapter:poly}. This is
the first polynomial upper bound on the convergence time of averaging algorithms. 
\item An $O(n^2 B \log 1/\epsilon)$ averaging algorithm in Chapter \ref{nsquared}. This is currently the averaging
algorithm with the best theoretical guarantees.  
\item An $\Omega(n^2)$ lower bound in Chapter \ref{ch:optimality} on the convergence of any distributed averaging algorithm
 that uses a single scalar state variable at each agent and satisfies a natural ``smoothness'' condition. 
 \item The finding of Chapter \ref{qanalysis} that storing and transmitting $c \log n$ bits can lead to arbitrarily
 accurate computation of the average by simply quantizing any linear, doubly stochastic averaging scheme. 
 \item The deterministic algorithm in Chapter \ref{ch:constantstorage} which, given initial values of $0$ or $1$ at each node, can compute which value has the majority with only a constant number of bits stored per link at each node. 
 \item The computability and non-computability results of Chapter \ref{ch:function}, which tell us that with deterministic anonymous
 algorithms and a constant number of bits per link at each node we can compute only functions depending on proportions; and that all the ``nice'' functions of proportions are computable. 
\end{enumerate} 

Many questions remain unanswered. We begin by listing several that are central
to understanding the fundamentals of network information aggregation questions.

\begin{enumerate} \item Is it possible to come up with local averaging algorithms which are robust to link failures, smoothly update a collection of real numbers of fixed size, and average faster than $O(n^2 B \log 1/\epsilon)$ on $B$-connected graph sequences?
\item What if in addition to the above requirements we also ask that the convergence time be on the order of the diameter for a (time-invariant) geometric random graph?
\item What if in addition to the above requirements  we also ask that the convergence time be on the order of the diameter for any fixed graph?
\item Suppose every node starts out with a $0$ or $1$. Is it possible to deterministically compute which node has the majority initially
if the graph sequence $G(t)$ changes unpredictably? However, we do insist that each $G(t)$ be undirected, and we require that each node store only a constant number of bits per each link
it maintains.  Naturally, some additional connectivity assumptions on $G(t)$ will have to be imposed for any positive result. 
\item More generally, characterize exactly which functions of binary initial values can be computed with a deterministic
algorithm which maintains a constant number of bits per link at each node. The graph sequence $G(t)$ is undirected, but may
be either fixed or change unpredictably.
\end{enumerate}

We next list some other open questions motivated by the problems considered here.

\begin{enumerate} 
\item For which classes of algorithms can an $\Omega(n^2)$ lower bound on convergence time be proven? In particular, is
it possible to replace the assumption of Chapter \ref{ch:optimality} that the update function be differentiable by the weaker
assumption that the update function be piecewise differentiable?  What if we add some memory to the algorithm?
\item Here is a concrete instance of the previous question. Is it true that any doubly stochastic matric ``on the ring''
mixes as $\Omega(n^2)$? That is, let $P$ be a doubly stochastic matrix such that $P_{ij}=0$ if $|i-j| ~{\rm mod}~ n > 1 $. Is it true
that $\max_{\lambda(A) \neq 1} |\lambda(A)| \geq 1 - c/n^2$ for some constant $c$?

Major progress on this question was recently made in \cite{G10}.
\item Where is the dividing line in Chapter \ref{chapter:poly} between polyynomial and exponential convergence time? In particular, how far may we relax the double stochasticity Assumption \ref{assumpt:ds} and still maintain polynomial convergence 
time? For example, does polynomial time convergence still hold if we replace Assumption \ref{assumpt:ds} with the 
requirement that the matrices $A(t)$ be (row) stochastic and each column sum is in $[1-\epsilon, 1+\epsilon]$ for
some small $\epsilon$?
\item Let $A_k$ be an irreducible, aperiodic stochastic matrix in $R^{k \times k}$. Let $\pi_k$ be its 
stationary distribution. Can we identify interesting classes of sequences 
\[ A_1, A_2, A_3, \ldots,\] which satisfy $\lim_{k \rightarrow \infty} \pi_k(i) = 0$ for any $i$? In words, we are asking
for each agent to have a negligible influence on the final result in the limit. This is sometimes a useful property in 
estimation problems; see \cite{GJ10}.
\item Is there a decentralized way to pick a symmetric ($a_{ij}=a_{ji}$), stochastic linear update rule which minimizes
\[ \sum_{\lambda(A) \neq 1} \frac{1}{1-\lambda_i(A)^2} \] This corresponds to handling the effect of white noise disturbances optimally;
see \cite{XBK07} for a centralized algorithm based on convex optimization.
\item For which classes of correlated random variables can we design decentralized algorithms for maximum likelihood 
estimation as done in Chapter \ref{why}?
\item Suppose every node starts out with a $0$ or $1$. The communication graph is undirected, fixed, but unknown to the nodes. Each node can maintain a fixed number of bits per each link it maintains. 
It is possible to (deterministically) decide whether $(1/n) \sum_{i=0}^n x_i > 1/\sqrt{2}$? What about $(1/n) \sum_{i=0}^n x_i > 1/\pi$?
\item Is it possible to improve the running time of the interval-averaging algorithm of Chapter \ref{ch:constantstorage}  
to $O(n^2 K \log K)$ communication rounds?
\item Given a connected graph $G=(\{1,\ldots,n\}, E)$, suppose that every link $e$ is an erasure channel with erasure probability $p_e$.
That is, every node can send a message on each link at each time step, but that the message is lost with probability $p_e$. Nodes
do not know whether their messages are lost. What is the time complexity of average or sum computation in such a setting? How
does it relate to various known graph connectivity measures?
\item How should one design averaging algorithms which send as few messages as possible? We ask that 
these algorithms work on arbitrary $B$-connected undirected graph sequences.
\end{enumerate} 

\appendix
\chapter{On quadratic Lyapunov functions for averaging}

\aoc{This appendix focuses on a technical issue appearing in Chapter
\ref{basic}; its content has previous appeared in the M.S. thesis
\cite{O04} and the paper \cite{OT08}.}

\aoc{Specifically, in chapter
\ref{basic}, a number of theorems on the convergence of the process
\[ x(t+1) = A(t) x(t),\] were proved, e.g. Theorems \ref{thm:simplestconvergence}, \ref{thm:symmetric}, \ref{thm:boundedstuff}, \ref{thm:symmetric}. These theorems were proved by showing that the ``span norm''
$\max_{i} x_i(t) - \min_i x_i(t)$ is guaranteed to decrease after a certain number
of iterations. Unfortunately, this proof method usually gives an overly conservative bound on the convergence time of the algorithm. Tighter bounds on the convergence time would have to rely on alternative Lyapunov functions, such as quadratic ones, of the form $x^T M x$, if they exist. }

\aoc{Moreover, in Chapter \ref{chapter:poly}, we developed bounds for
convergence of averaging algorithms based on quadratic Lyapunov functions. Thus we are led to the question: is it possible
to find quadratic Lyapunov functions for the non-averaging convergence theorems of Chapter \ref{basic}? A positive
answer might lead to improved convergence times.}

\alexo{Although quadratic Lyapunov functions can always be found for
linear systems, they may fail to exist when the system is allowed to
switch between a fixed number of linear modes. On the other hand,
there are classes of such switched linear systems that do admit
quadratic Lyapunov functions. See \cite{LM99} for a broad overview
of the literature on this subject.}

\aoc{The simplest version of this question deals with the symmetric, equal-neighbor
model and was investigated in \cite{jad}.} The authors
write:

\begin{quote} ``...no such common Lyapunov matrix $M$ exists. While we have not been able to construct a simple analytical example which demonstrates
this, we have been able to determine, for example, that no common
quadratic Lyapunov function exists for the class of all [graphs
which have] $10$ vertices and are connected. One can verify that
this is so by using semidefinite programming...'' \end{quote}

The aim of this appendix is to provide an analytical
example that proves this fact.

\section{The Example}

Let us fix a positive integer $n$. We start by defining a class
$\alexo{\Q}$ of functions with some minimal desired properties of
quadratic Lyapunov functions. Let $\e$ be the vector in $\Re^n$ with
all components equal to 1. A square matrix is said to be {\it
stochastic} if it is nonnegative and the sum of the entries in each
row is equal to one. Let $\A \subset \Re^{n \times n}$ be the set of
stochastic matrices $A$ such that: (i) $a_{ii}>0$, for all $i$; (ii)
all positive entries on any given row of $A$ are equal; (iii)
$a_{ij} > 0$ if and only if $a_{ji}>0$; (iv) the graph associated
with the set of edges $\{ (i,j) ~|~ a_{ij}>0 \}$ is connected. These
are precisely the matrices that correspond to a single iteration of
the equal-neighbor algorithm on symmetric, connected graphs.

\begin{definition} \label{lyap-def}
\label{d:l} A function $\alexo{Q}:\Re^n\to\Re$ belongs to the class
$\alexo{\Q}$ if it is of the form $\alexo{Q}(x)=x^TMx$, where:
\begin{itemize}
\item[(a)]
The matrix $M \in \Re^{n \times n}$ is nonzero, symmetric, and
nonnegative definite.
\item[(b)] For every $A\in\A$, and $x\in\Re^n$, we have
$\alexo{Q}(Ax) \leq \alexo{Q}(x)$.
\item[(c)] We have $\alexo{Q}(\e)=0$.
\end{itemize}
\end{definition}

Note that condition (b) may be rewritten in matrix form as
\begin{equation} \label{cond:b} x^TA^TMAx \leq x^T M x, \qquad
\mbox{for all } A \in \A, \mbox{ and } x \in \Re^n.
\end{equation} The rationale behind condition (c) is as follows.
Let $S$ be the subspace spanned by the vector $\e$. Since we are
interested in convergence to the set $S$, and every element of $S$
is a fixed point of the algorithm, it is natural to require that
$\alexo{Q}(\e)=0$, or, equivalently,
$$M\e=0.$$
Of course, for a Lyapunov function to be useful, additional
properties would be desirable. For example we should require some
additional condition that guarantees that $\alexo{Q}(x(t))$
eventually decreases. However, according to Theorem~\ref{th:main},
even the minimal requirements in Definition \alexo{\ref{lyap-def}}
are sufficient to preclude the existence of a quadratic Lyapunov
function.

\begin{theorem} \label{th:main}
Suppose that $n \geq 8$. Then, the class $\alexo{\Q}$ (cf.\
Definition \ref{d:l}) is empty.
\end{theorem}

The idea of the proof is as follows. Using the fact the dynamics of
the system are essentially the same when we rename the components,
we show that if $x^TMx$ has the desired properties, so does $x^T
\alexo{Z} x$ for a matrix $\alexo{Z}$ that has certain
permutation-invariance properties. This leads us to the conclusion
that there is essentially a single candidate Lyapunov function, for
which a counterexample is easy to develop.

Recall that a permutation of $n$ elements is a bijective mapping
$\sigma:\{1,\ldots,n\}\to \{1,\ldots,n\}$. Let $\Pi$ be the set of
all permutations of $n$ elements. For any $\sigma\in\Pi$, we
define a corresponding permutation matrix $P_{\sigma}$ by letting
the $i$th component of $P_{\sigma}x$ be equal to $x_{\sigma(i)}$.
Note that $\Ps^{-1}=\Ps^T$, for all $\sigma\in \Pi$. Let $\Pm$ be the
set of all permutation matrices corresponding to permutations in
$\Pi$.
\begin{lemma} \label{sumpermutations} Let $M \in \alexo{\Q}$. Define $\alexo{Z}$ as
\[ \alexo{Z} = \sum_{P\in\Pm} P^T M P. \] Then, $\alexo{Z} \in \alexo{\Q}$.
\end{lemma}

\noindent {\bf Proof:} For every matrix $A\in\A$, and any $P\in
\Pm$, it is easily seen that $P A P^T \in\A$. This is because
the transformation $A\mapsto P A P^T$ amounts to permuting the
rows and columns of $A$, which is the same as permuting (renaming)
the nodes of the graph.

We claim that if $M\in \alexo{\Q}$ and $P\in\Pm$, then $P^T M P\in
\alexo{\Q}$. Indeed, if $M$ is nonzero, symmetric, and nonnegative
definite, so is $P^T M P$. Furthermore, since $P\e =\e$, if $M\e=0$,
then $P^T M P \e= 0$. To establish condition (b) in Definition
\ref{d:l}, let us introduce the notation $\alexo{Q}_P(x) = x^T (P^T
M P) x$. Fix a vector $x\in\Re^n$, and $A \in \A$; define $B=P A
P^T\in\A$. We have
\begin{eqnarray*} \alexo{Q}_P(Ax) & = & x^T A^T P^T M P A x \\
&=&x^T P^T P A^T P^T M P A P^T P x\\
&=&x^T P^T B^{T} M B P x\\
&\leq& x^T P^T M P x\\
& = & \alexo{Q}_P(x),
\end{eqnarray*}
where the inequality follows by applying Eq. (\ref{cond:b}), which
is satisfied by $M$, to the vector $Px$ and the matrix $B$. We
conclude that $\alexo{Q}_P \in \alexo{\Q}$.

Since the sum of matrices in $\alexo{\Q}$ remains in $\alexo{\Q}$,
it follows that $\alexo{Z}= \sum_{P\in\Pm} P^T M P$ belongs to
$\alexo{\Q}$. \rule{7pt}{7pt}

We define the ``sample variance'' $V(x)$ of the values
$x_1,\ldots,x_n$, by
$$V(x)= \sum_{i=1}^n (x_i - \bar{x})^2,$$ where
$\bar{x} = (1/n) \sum_{i=1}^n x_i$. This is a nonnegative quadratic
function  of $x$, and therefore, $V(x)=x^T C x$, for a suitable
nonnegative definite, nonzero symmetric matrix $C \in \Re^{n \times
n}$.

\begin{lemma} \label{qvariance} There exists some $\alpha>0$ such that
\[ x^T \alexo{Z} x = \alpha V(x), \qquad \mbox{for all } x\in \Re^n. \] \end{lemma}

\noindent {\bf Proof}: We observe that the matrix $\alexo{Z}$
satisfies
\begin{equation}
\alexo{R}^T \alexo{Z} \alexo{R} = \alexo{Z}, \qquad \mbox{for all }
R \in \Pm. \label{eq:p}
\end{equation}
To see this, fix $R$ and notice that the mapping $P\mapsto P R$ is a
bijection of $\Pm$ onto itself, and therefore,
$$R^T \alexo{Z} R= \sum_{P\in\Pm} (PR)^T M (P R)
=\sum_{P\in\Pm} P^T M P =\alexo{Z}.$$

We will now show that condition (\ref{eq:p}) determines $\alexo{Z}$,
up to a multiplicative factor. Let $\alexo{z}_{ij}$ be the $(i,j)$th
entry of $\alexo{Z}$. Let $\ei$ be the $i$th unit vector, so that
$\ei^T \alexo{Z} \ei = \alexo{z}_{ii}$. Let $P\in\Pm$ be a
permutation matrix that satisfies $P\ei=\ej$. Then,
$\alexo{z}_{ii}=\ei^T \alexo{Z} \ei =\ei^T P^T \alexo{Z} P \ei=\ej^T
 \alexo{Z} \ej =\alexo{z}_{jj}$. Therefore, all diagonal entries of $\alexo{Z}$ have a common
value, to be denoted by $\alexo{z}$.

Let us now fix three distinct indices $i,j,k$, and let $y=\ei+\ej$,
$\alexo{w}=\ei+\ek$.  Let $P\in\Pm$ be a permutation matrix such
that $P\ei=\ei$ and $P\ej=\ek$, so that $Py=\alexo{w}$. We have
$$
2 \alexo{z} +2 \alexo{z}_{ij} = y^T \alexo{Z} y= y^T P^T \alexo{Z} P
y = \alexo{ w^T Z w =2z + 2 z_{ik}}.$$ By repeating this argument
for different choices of $i,j,k$, it follows that all off-diagonal
entries of $\alexo{Z}$ have a common value to be denoted by $r$.
Using also the property that $\alexo{Z} \e=0$, we obtain that
$\alexo{z}+(n-1)r=0$. This shows that the matrix $\alexo{Z}$ is
uniquely determined, up to a multiplicative factor.

We now observe that permuting the components of a vector $x$ does
not change the value of $V(x)$. Therefore, $V(x)=V(Px)$ for every
$P\in\Pm$, which implies that $x^TP^T CPx=x^T C x$, for all $P\in
\Pm$ and $x\in\Re^n$. Thus, $C$ satisfies (\ref{eq:p}). Since all
matrices that satisfy (\ref{eq:p}) are scalar multiples of each
other, the desired result follows. \rule{7pt}{7pt}
\bc \begin{figure}[h] \bc
        \epsfig{file=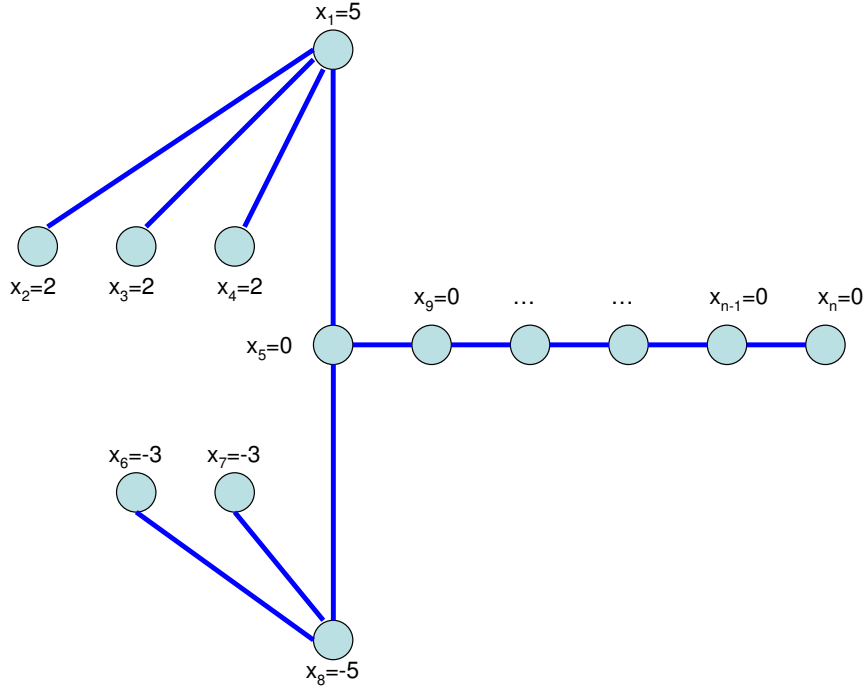,width=11.5cm} \caption{\label{f1}
 A connected graph
on $n$ nodes showing that $V(x)$ is not a Lyapunov function when $n
\geq 8$. All arcs of the form $(i,i)$ are also assumed to be
present, but are not shown. The nodes perform an iteration of the
symmetric, equal-neighbor model according to this graph. } \ec
  \end{figure} \ec

\noindent {\bf Proof of Theorem \ref{th:main}:} In view of Lemmas
\ref{sumpermutations} and \ref{qvariance}, if $\alexo{\Q}$ is
nonempty, then $V \in \alexo{\Q}$. Thus,  it suffices to show that
$V \notin \alexo{\Q}$. Suppose that $n \geq 8$, and consider the
vector $x$ with components $x_1=5$, $x_2=x_3=x_4=2$, $x_5=0$,
$x_6=x_7=-3$, $x_8=-5$, and $x_9 = \cdots = x_n = 0$. We then have
$V(x)=80$. Consider the outcome of one iteration of the symmetric,
equal-neighbor algorithm, if the graph has the form shown in Figure
\ref{f1}. After the iteration, we obtain the vector $y$ with
components  $y_1 = 11/5, y_2=y_3=y_4=7/2, y_5=0, y_6=y_7= -4, y_8 =
-11/4$, and $y_9=\cdots=y_n=0$. We have
\begin{eqnarray} V_n(y) & = & \sum_{i=1}^n (y_i - \bar{y})^2\nonumber \\
& \geq & \sum_{i=1}^8 (y_i - \bar{y})^2 \nonumber\\
& \geq & \sum_{i=1}^8 \Big(y_i - \frac{1}{8} \sum_{i=1}^8 y_i\Big)^2,
\label{eq:v8}
\end{eqnarray}
where we used that $\sum_{i=1}^k (y_i - z)^2$ is minimized
          when $z=(1/k) \sum_{i=1}^k y_i$. A
simple calculation shows that
the expression (\ref{eq:v8}) evaluates to $10246/127 \approx 80.68$,
which implies that $V(y)> V(x)$. Thus, if $n\geq8$, $V \notin
\alexo{\Q}$, and the set $\alexo{\Q}$ is empty. \rule{7pt}{7pt}

\section{Conditions for the Existence of a Quadratic Lyapunov Function}
Are there some additional conditions (e.g., restricting the matrices
$A$ to a set smaller than $\A$), under which a quadratic Lyapunov
function is guaranteed to exist? We start by showing that the answer
is positive for the case of a fixed matrix (that is, if the graph
$G(t)$ is the same for all $t$).

Let $A$ be a stochastic matrix, and suppose that there exists a
positive vector $\pi$ such that $\pi^T A = \pi^T$. Without loss of
generality, we can assume that $\pi^T \e = 1$. It is known that in
this case,
\begin{equation}\label{contr}
x^T A^T D A x \leq x^T D x,\qquad \forall\ x\in\Re^n,
\end{equation}
where $D$ is a diagonal matrix, whose $i$th diagonal entry is equal
to $\pi_i$ (cf.\ Lemma 6.4 in \cite{NDP}). However, $x^T D x$ cannot
be used as a Lyapunov function because $D\e \neq 0$ (cf.\ condition
(c) in Definition \ref{d:l}). To remedy this, we argue as in
\cite{BHOT05} and define the matrix $H=I-\e \pi^T$, and consider the
choice $M=H^T D H$. Note that $M$ has rank $n-1$.

We have $H\e= (I-\e \pi^T) \e =\e -\e (\pi^T \e) =\e-\e=0$, as
desired. Furthermore,
$$HA=A -\e \pi^T A =A-\e \pi^T = A - A \e \pi^T = AH.$$ Using this property, we obtain, for every $x\in\Re^n$,
$$x^T A^T M A x = x^T A^T H^T D H A x
=(x^T H^T) A^T D A (H x) \leq x^T H^T D H x = x^T M x,$$ where the
inequality was obtained from (\ref{contr}), applied to $Hx$. This
shows that $H^T D H$ has the desired properties (a)-(c) of
Definition \ref{d:l}, provided that $\A$ is replaced with $\{A\}$.

We have just shown that every stochastic matrix (with a positive
left eigenvector associated to the eigenvalue 1) is guaranteed to
admit a quadratic Lyapunov function, in the sense of Definition
\ref{d:l}. \alexo{Moreover, our discussion implies that there are
some classes of stochastic matrices $\B$ for which the same Lyapunov
function can be used for all matrices in the class.}
\begin{itemize}
\item[(a)]
Let $\B$ be a set of stochastic matrices. Suppose that there exists
a positive vector $\pi$ such that $\pi^T \e =1$, and $\pi^T A
=\alexo{\pi^T}$ for all $A\in\B$. Then, there exists a nonzero,
symmetric, nonnegative definite matrix $M$, of rank $n-1$, such that
$M\e=0$, and $x^T A^T M Ax \leq x^T M x$, for all $x$ and $A\in\B$.
\item[(b)]
The condition in (a) above is automatically true if all the matrices
in $\B$ are doubly stochastic (recall that a matrix $A$ is doubly
stochastic if both $A$ and $A^T$ are stochastic); in that case, we
can take $\pi=\e$.
\item[(c)] The condition in (a) above holds if and only if there exists a positive vector $\pi$, such that $\pi^T A x = \pi^T x$, for all $A\in\B$ and all $x$. In words, there must be a positive linear functional of the agents' opinions which is conserved at each iteration. For the case of doubly stochastic matrices, this linear functional is any positive multiple of the sum $\sum_{i=1}^n x_i$ of the agents' values (e.g., the average of
these values).
\end{itemize}

\chapter{Averaging with deadlock avoidance}

In this appendix, we describe a solution of a certain averaging
problem communicated to the author and J.N. Tsitsiklis by A.S. Morse. This is the
problem of averaging with ``deadlock avoidance,'' namely averaging with 
the requirement that each node participate in at most one pairwise average at every time step.

More concretely, we will describe a deterministic distributed algorithm for nodes $1, \ldots,n$,
with starting values $x_1(0), \ldots, x_n(0)$, to compute the average $(1/n) \sum_{i=1}^n x_i(0)$.
Our algorithm selects pairs of nodes to average at each time step. We will
allow the communication graphs $G(t) = (\{1,\ldots,n\}, E(t))$ to change with time, but we
will assume that each graph G(t) is undirected.  Moreover, our algorithm will involve three
rounds of message exchanges between nodes and their neighbors; we assume
these messages can be sent before the graph $G(t)$ changes.
We do not assume nodes have unique identifers; however, we do assume
each node has a ``port labeling'' which allows it to tell neighbors apart, that is,
if node $i$ has $d(i)$ neighbors in the graph $G(t)$, it assigns them labels $1,\ldots,d(i)$
in some arbitrary way.

\section{Algorithm description}

We will partition every time step into several synchronous
``periods.'' At the end of each period, nodes send each other
messages, which all of them synchronously read at the beginning of
the following period. During the first period, nodes exchange their
values with their neighbors; after that, nodes will send messages
from the binary alphabet $\{+,-\}$. Eventually, as a consequence of
these messages, nodes will "pair up" and each node will average its
value with that of a selected neighbor.

Intuitively, nodes will send ``$+$''s to neighbors they would like to
average with, and ``$-$''s to decline averaging requests. Their goal
will be to average with the neighbor whose value is the farthest
from their own.

The following is a description of the actions undertaken by node $i$ at
time $t$. At this point the nodes have values $x_1(t), \ldots, x_n(t)$.  Node $i$ will
keep track of a variable $N(i)$, representing the node it would like to average
with; and a variable $g_i)$, which will always be equal to $|x_i(t)-x_{N(i)}(t)|$.

We find it convenient to label the nodes $1,\ldots,n$ to make
statements like ``node $i$ sets $N(i) = j$." Of course, due to the
absence of unique identifiers, node $i$ will set its local variable
$N(i)$ to its port number for its neighbor $j$, and the above
statement should be understood as such.

\noindent {\bf The algorithm (at node $i$, time $t$):}

\begin{enumerate} \item Initialize $N(i) = i$ and $g_i = 0$. Node $i$ broadcasts its
value to all its neighbors.

\item Node $i$ reads the incoming messages. Next, node $i$ sends a ``$+$'' to
a neighbor $k$ with the smallest value among neighbors with value
smaller than $x_i(t)$. Ties can be broken
arbitrarily\footnote{...but deterministically to stay within our
assumption of deterministic algorithms. For example, node $i$ can
break ties in favor of lower port number.}. It sets $N(i) = k$ and
$g_i = |x_i(t)-x_k(t)|$.

If no such $k$ exists, node $i$ does nothing during this period.

\bibitem Node $i$ reads incoming messages. If node $i$ received at least
one ``$+$'' at the previous step, it will compute the gap
$|x_i(t)-x_j(t)|$ for every neighbor $j$ in the set $J$ of neighbors
that have sent it a ``$+$.''
\begin{itemize} \item If the node $m \in J$ with the largest\footnote{Ties can be broken arbitrarily, but
deterministically as before.} gap has $|x_m(t)-x_i(t)|
> g_i$, $i$ will update $N(i) = m$, $g_i = |x_m(t)-x_i(t)$; next, $i$ will
send a ``$+$'' to $m$ and ``$-$'' to all the other nodes in $J$.
Moreover, if $i$ sent a ``$+$'' to a node $k$ in step 2, it now sends
$k$ a ``$-$.''
\item On the other hand, if $|x_m(t)-x_i(t)| \leq  g_i$, node $i$ will send a
"$-$" to everyone in J.  \end{itemize} \item Node $i$ reads incoming
messages. If $i$ receives a ``$-$'' from node $N(i)$, it sets $N(i) =
i$, $g_i = 0$. \item Finally, $i$ sets \[ x_i(t+1) \leftarrow
\frac{x_i(t) + x_{N(i)}(t)}{2}. \] Observe that if $N(i) = i$ at the
execution of step 5, then the value of node $i$ is left unchanged.
\end{enumerate}

\section{Sketch of proof}

It is not hard to see that this algorithm results in convergence to
the average which is geometric with rate $1-c/n^3$, for some
constant $c$. To put it another way, if initial values $x_i(t)$  are
in $[0,1]$ then everyone is within $\epsilon$ of the average in
$O(n^3 \log(n \epsilon))$ time. An informal outline of the proof of
this statement follows.

\bigskip

\noindent {\bf Sketch of proof.}

\bigskip

First, we informally describe the main idea. First, one needs to argue that the
algorithm we just described 
results a collection of pairwise averages, i.e. if node $i$ sets $x_i(t+1) = (x_i(t)+x_j(t))/2$ 
in the final step, then node $j$ sets $x_j(t+1) = (x_j(t)+x_i(t))/2$. Moreover, 
we will argue that at least one of these pairwise averages occurs across an edge with a
``large'' gap; more precisely, we will argue that among the edges $(i,j)$ maximizing $|x_i(t)-x_j(t)|$, at least one
pair $i,j$ match up. This fact implies, after some simple analysis, that the ``sample variance'' defined as
$\sum_{i=1}^n (x_i(t) - (1/n) \sum_j x_j(t))^2$ shrinks by at least $1-1/2n^3$ from time $t$ to
$t+1$. The desired convergence result then follows. 

\bigskip

\begin{enumerate} \item At the beginning of Step 5, if $N(i) = j$ then $N(j) = i$.

In words, the last step always results in the execution of a number
of pairwise averages for disjoint pairs. This can be proven with a
case-by-case analysis.

\item Moreover, one of these averages has to happen on an edge with
the largest ``gap'' $\max_{(i,j) \in E(t)} |x_i(t)-x_j(t)|$, where
$E(t)$ is the set of edges in the communication graph at time $t$.

Indeed, let $J$ be the set of nodes incident on one of these
maximizing edges, and let $J'$ be the set of nodes in $J$ with the
smallest value. Then, at least one node $j' \in J'$ will receive an
offer at Step 2 that comes along a maximizing edge. Any offer $j'$
makes at Step 2 will be along an edge with strictly smaller gap.
Consequently, at step 3, $j'$ will send a ``$+$'' to a sender of an
offer along the maximizing edge --- let us call this sender $i'$ --- and a
``$-$'' to everyone else that sent $j'$ offers, as well as to any node
which made $j'$ an offer in Step $2$.

Since $i'$ and $j'$ are connected by a maximizing edge, there is no
way $i'$ receives a ``$+$''  in Step 2 associated with a gap larger
than $|x_{i'}(t)-x_{j'}(t)|$, so that $i'$ does not send $j'$ a
``$-$'' at step 3. Finally, $i'$ and $j'$ average at Step 5.

\item Without loss of generality, we will assume that
$\sum_i x_i(0)=0$. Since every averaging operation preserves the
sum, it follows that $\sum_i x_i(t)=0$. Let $V(t)=\sum_i x_i^2(t)$.
Its easy to see that $V(t)$ is nonincreasing, and moreover, an
averaging operation of nodes $a$ and $b$ at time $t$ reduces $V$ 
by $(1/2)(x_a-x_b)^2$.

\item Let $U(t) = \max_{i} |x_i(t)|$. Clearly, $V(t) \leq n U^2$.
However, $\sum_i x_i(t) = 0$ implies that at least one $x_i(t)$ is
negative, so that
\[ \max_{(i,j) \in E} |x_i(t)-x_j(t)| \geq \frac{U}{n},\] which implies that
\[ V(t+1) \leq V(t) - \frac{U^2}{2n^2},\] or
\[ V(t+1) \leq V(t) - \frac{1}{2n^3} V(t). \] It follows that after $O(n^3)$ steps, $V(t)$ shrinks by a
constant factor. Thus if the initial values $x_i(t)$  are in $[0,1]$
then everyone is within $\epsilon$ of the average in $O(n^3 log(n
\epsilon))$ time.
\end{enumerate}

\chapter{List of assumptions}

\newtheorem{assumpt}{Assumption}

\vspace{5mm}\noindent{\bf Assumption 2.1} (connectivity) \nonumber The graph \[ \cup_{s \geq t} G(s) \] is connected for every $t$. 

\bigskip

\vspace{5mm}\noindent{\bf Assumption 3.1} (non-vanishing weights) The matrix $A(t)$ is nonnegative, stochastic, and has
positive diagonal. Moreover, there exists some $\eta>0$ such that if
$a_{ij}>0$ then $a_{ij}>\eta$.

\bigskip

\vspace{5mm}\noindent{\bf Assumption 3.2} ($B$-connectivity) There exists an integer $B>0$
such that the directed graph \[ (N, E(kB) \cup E((k+1)B) \cup \cdots
\cup E((k+1)B-1)) \] is strongly connected for all integer $k \geq
0$.

\bigskip

\vspace{5mm}\noindent{\bf Assumption 3.3} (bounded delays)

\noindent (a) If $a_{ij}(t)=0$, then
$\tau^i_j(t)=t$.\\
(b) $\lim_{t\to\infty} \tau^i_j(t)=\infty$, for all $i$, $j$.\\
(c) $\tau^i_i(t)=t$, for all $i$, $t$.\\
(d) There exists some $B>0$ such that $t-B+1\leq \tau^i_j(t)\leq t$,
for all $i$, $j$, $t$.

\bigskip

\vspace{5mm}\noindent{\bf Assumption 3.4}  (double stochasticity)
matrix $A(t)$ is column-stochastic for all $t$, i.e., \[ \sum_{i=1}^n
a_{ij}(t) = 1,\] for al $j$ and $t$.

\bigskip

\vspace{5mm}\noindent{\bf Assumption 3.5}  (Bounded round-trip times)
There exists some $B>0$ such that whenever $(i,j)\in E(t)$, then
there exists some $\tau$ that satisfies $|t-\tau|<B$ and $(j,i)\in
E(\tau)$.

\bigskip

\vspace{5mm}\noindent{\bf Assumption 4.1} (connectivity relaxation)  Given an integer $t\ge 0$,
suppose that the components of $x(tB)$ have been reordered so that
they are in nonincreasing order. We assume that for every
$d\in\{1,\ldots,n-1\}$, we either have $x_d(tB)=x_{d+1}(tB)$, or
there exist some time $t\in\{tB,\ldots,(t+1)B-1\}$ and some
$i\in\{1,\ldots,d\}$, $j\in\{d+1,\ldots,n\}$ such that $(i,j)$ or
$(j,i)$ belongs to $\E(A(t))$.

\bigskip

\vspace{5mm}\noindent{\bf Assumption 7.1} 
For all $i$, $x_i(0)$ is a multiple of
$1/Q$.



\end{singlespace}
\end{document}